\documentclass[11pt]{amsart}

\usepackage[utf8]{inputenc} 	
\usepackage[T1]{fontenc}
\usepackage{lmodern}
\usepackage{yhmath}
\usepackage{stmaryrd} 
\usepackage[shortlabels]{enumitem} 
\usepackage{longtable}
\usepackage{booktabs}
\usepackage[ruled,vlined,linesnumbered,algo2e,commentsnumbered]{algorithm2e}

\usepackage{amsthm,amsmath,amssymb,amsbsy,bbm,mathrsfs,supertabular,
eurosym,graphicx,enumitem,xcolor}
\usepackage{mathtools}
\usepackage{hhline} 
\usepackage{tikz}
\usetikzlibrary{decorations.pathreplacing}

\usepackage[normalem]{ulem} 
\usepackage[authoryear]{natbib}

\newtheorem{thm}{Theorem}
\newtheorem{lem}{Lemma}
\newtheorem{prop}{Proposition}
\newtheorem{df}{Definition}
\newtheorem{cor}{Corollary}
\newtheorem{ass}{Assumption}

\newtheorem{definition}{Definition}

\newcounter{subass}[ass]
\renewcommand{\thesubass}{\theass.\arabic{subass}}
\newenvironment{subassumption}[1][]{
  \refstepcounter{subass}
  \par\noindent\textbf{\thesubass\ #1.}\ }{\par}

\theoremstyle{definition}

\newtheorem{rem}{Remark}
\newtheorem{ex}{Example}

\usepackage[english]{babel}

\usepackage{enumitem}
\usepackage{hyperref}  

\newcommand{\argmin}{\mathop{\rm argmin}}

\newcommand{\Var}{\mathop{\rm Var}\nolimits}

\newcommand{\E}{{\mathbb{E}}}
 
\renewcommand{\L}{{\mathbb{L}}}
\newcommand{\N}{{\mathbb{N}}}
\renewcommand{\P}{{\mathbb{P}}}
 
\newcommand{\R}{{\mathbb{R}}}

\newcommand{\rB}{{\mathrm{B}}}

\newcommand{\rL}{{\mathrm{L}}} 
\newcommand{\rT}{{\mathrm{T}}}

\newcommand{\frc}{{\mathfrak{c}}}

\newcommand{\frl}{{\mathfrak{l}}}

\newcommand{\frr}{{\mathfrak{r}}}

\DeclareMathAlphabet{\mathscrbf}{OMS}{mdugm}{b}{n}
\newcommand{\cA}{{\mathcal{A}}}
\newcommand{\cB}{{\mathcal{B}}}
\newcommand{\cC}{{\mathcal{C}}}
\newcommand{\cD}{{\mathcal{D}}}
\newcommand{\cE}{{\mathcal{E}}}
\newcommand{\cF}{{\mathcal{F}}}
 
\newcommand{\cH}{{\mathcal{H}}}

\newcommand{\cL}{{\mathcal{L}}} 

\newcommand{\cN}{{\mathcal{N}}}

\newcommand{\cR}{{\mathcal{R}}}
\newcommand{\cS}{{\mathcal{S}}} 
\newcommand{\cT}{{\mathcal{T}}}

\newcommand{\wh}[1]{\widehat{#1}}

\newcommand{\ov}[1]{\overline{#1}}
\newcommand{\op}{\mathrm{op}}

\def\source{f_{\cS}}
\def\target{f_{\cT}}
\def\wsource{\widehat{f}_{\cS}}
\def\wtarget{\widehat{f}_{\cT}}
\def\wTL{\widehat f_{\cT}^{\text{\scshape tl}}}
\def\wtess{\widehat f_{\cT}^{H}}

\def\gsl{g_{\ell}^{\star}}
\def\gslin{g^H_{\source}}
\def\gabh{g^H_{\ov{\theta}}}

\def\oh{\overline{h}}
\def\xhl{x_{H,\ell}}

\def\frlloc{\frl_{\mathrm{loc}}}
\def\betaloc{\beta_{\mathrm{loc}}}

\newlist{lista}{enumerate}{1}
\setlist[lista,1]{label=\alph*),ref=\alph*)}

\newlist{listi}{enumerate}{1}
\setlist[listi,1]{label=(\roman*),ref=(\roman*),align=left}

\newcommand{\thanksmark}[1]{\textsuperscript{#1}}

\renewcommand{\ge}{\geqslant}
\renewcommand{\le}{\leqslant}
\renewcommand{\geq}{\geqslant}
\renewcommand{\leq}{\leqslant}

\newcounter{algo}
\usepackage[textsize=tiny]{todonotes}

\title[Localized Transfer Learning]{Tessellation Localized  Transfer learning for nonparametric regression}

\author{Hélène Halconruy$^{\thanksmark{1},\thanksmark{2}}$}
\author{Benjamin Bobbia $^\thanksmark{3}$}
\author{Paul Lejamtel $^\thanksmark{4}$}

\thanks{SAMOVAR, Télécom SudParis, Institut Polytechnique de Paris,
91120 Palaiseau, France.}

\thanks{Modal'X, Université Paris Nanterre,
92000 Nanterre, France.}

\thanks{DISC, Fédération Enac ISAE-SUPAERO,université de Toulouse, 31055 Toulouse, France}

\thanks{ESILV, Pôle Léonard de Vinci, 92400 Courbevoie, France}

\begin{document}

\title{Tessellation localized  transfer learning for nonparametric regression}







\begin{abstract}
\noindent
\,Transfer learning can improve target prediction by leveraging related source tasks, but may also suffer from negative transfer when source-target similarity is heterogeneous. We propose a nonparametric regression transfer learning framework based on a local transfer assumption, in which the covariate space is partitioned into finitely many cells and the target regression function is locally approximated by a low-complexity transformation of the source regression function. This formulation allows similarity to vary spatially, enabling informative transfer while preventing harmful information sharing. We establish sharp minimax convergence rates for estimating the target regression function under a well-specified local linear transfer model, showing that transfer learning can mitigate the curse of dimensionality by exploiting reduced functional complexity. We also propose fully data-driven procedures that jointly select the partition and estimate both the transfer functions and the target regression, and derive oracle inequalities ensuring robustness to model misspecification. Numerical experiments illustrate the theoretical findings, including an application to stock return prediction using signature-based features.
\end{abstract}



\maketitle

\section{Introduction}

Transfer learning (TL) originates in educational psychology, where learning is viewed as the generalization of prior experience to new but related situations. As articulated by C. H. Judd, \emph{transfer is possible only when a meaningful similarity or structural connection exists between learning activities}. A classical example is that musical training on the violin facilitates subsequent learning of the piano, due to shared underlying concepts and skills. This principle was later formalized in the machine learning literature (see, e.g., the surveys \cite{pan2009survey,Zhuang_2020,zhu2025recent}), where transfer learning refers to improving performance on a target task by exploiting knowledge acquired from a sufficiently similar source task.\\
\noindent
This paradigm is particularly compelling when target observations are rare, costly, or otherwise constrained, but data from related environments are available. In such settings, transfer learning offers a natural way to mitigate data scarcity by exploiting structural similarities across datasets. Its applications are broad, spanning classical machine learning domains such as computer vision \cite{ganin2016domain,Wang_2018_visualDA} and natural language processing \citep{ruder2019transfer,devlin2019bert}, including sentiment analysis \citep{liu2019roberta}. It has also been widely adopted in recommender systems \citep{pan2009survey,man2017cross} and fraud detection \citep{lebichot2021transfer}, and has more recently attracted increasing attention in the statistics community. In this context, transfer learning has been applied to nonparametric classification \cite{cortes2014domain,Reeve_Cannings_Samworth_2021,cai2021transfer}, large-scale Gaussian graphical models \cite{li2023transfer}, and contextual multi-armed bandits. At the same time, these developments have underscored the risk of \emph{negative transfer}, \emph{i.e.}, situations in which exploiting source information increases the target estimation error, thereby motivating robust approaches that explicitly account for heterogeneous or unreliable sources, such as the framework of Fan \emph{et al.} \cite{fan2025robust}.

\medskip
\noindent
In this work, we propose a regression transfer learning framework that explicitly accounts for heterogeneity in the relationship between source and target tasks, with the goal of enabling transfer where it is beneficial while avoiding negative transfer elsewhere. Our analysis is built around two complementary models: a compositional transfer model and a local linear transfer model.
\\
Inspired by the local perspective advocated by Reeve \emph{et al.}, we adopt a \textbf{compositional transfer model} in which the covariate space $[0,1]^d$ is partitioned into finitely many cells $H^\star={\cA_\ell^\star:\ell\in[L^\star]}$ such that, on each cell, the target regression function $\target$ is obtained by composing the source regression function $\source$ with a cell-specific transfer map. More precisely, for every $\ell\in[L^\star]$, there exists a function $g_\ell:\R\to\R$ satisfying $\target(x)=g_\ell\bigl(\source(x)\bigr)$ for all $x\in\cA_\ell^\star$.
When such a tessellation exists, the model is said to be \emph{well specified}. This structure captures spatially heterogeneous similarity: the form of transfer is allowed to vary across cells, yet remains low-dimensional within each cell since $g_\ell$ acts on the scalar quantity $\source(x)$. In this compositional regime, we show that both the transfer functions $g_\ell$ and the target regression function can be consistently estimated.
\\
To accommodate broader forms of similarity, we also introduce a more flexible \textbf{local linear transfer model}. In this second setting, we no longer require an exact compositional representation of the form $g_\ell\circ\source$. Instead, we assume that, within each cell, the discrepancy between $\target$ and $\source$ satisfies suitable local smoothness or low-complexity conditions. This relaxation allows for approximate or non-compositional relationships while retaining the key idea that transfer should be localized in the covariate space.
\\
Within this unified framework, we establish minimax convergence rates for estimating the target regression function under well-specified local linear transfer. In the compositional case, the additional structure yields further gains through the estimation of the transfer maps themselves. More generally, our guarantees take the form of oracle inequalities that separate estimation and approximation errors, thereby ensuring robustness to misspecification: when the compositional structure holds only approximately, or when the oracle tessellation is not exactly contained in the admissible class, performance deteriorates smoothly through an explicit bias term rather than collapsing due to negative transfer.
\\
Finally, we illustrate the practical relevance of our approach through experiments on synthetic data, where we control the degree of transfer and misspecification. We also consider the \emph{Abalone} dataset, where gender-specific regression tasks exhibit partial but heterogeneous similarity. In addition, we study a financial time-series application in which AMD stock returns are predicted using signature-based features.
These experiments confirm our theoretical findings. They highlight the gains achievable under well-specified transfer and demonstrate the stability of the method in misspecified settings.

\subsection{Related work}

A rapidly growing literature studies transfer learning in regression, aiming to exploit structural similarities between source and target models to improve estimation accuracy while mitigating the risk of \emph{negative transfer}, that is, performance degradation due to misleading or mismatched source information.

\medskip
\noindent\textit{\underline{Linear and parametric regression.}}
Early work focused on linear regression, where transfer is naturally expressed through proximity between regression coefficients or fitted responses. In the data-enriched framework of Chen, Owen and Shi \cite{Chen_Owen_Shi_2015}, a small target sample is complemented by a larger, potentially biased auxiliary dataset, with theoretical guarantees characterizing the bias-variance trade-off. Related parametric approaches include robust modeling under population shifts \cite{Bouveyron_Jacques_2010}, formal assessments of \emph{transfer gain} \cite{Obst_2022}, and semi-supervised extensions where unlabeled target covariates are abundant \cite{lai2025robust}.

\medskip
\noindent\textit{\underline{High-dimensional transfer in regression.}}
In high-dimensional regimes, transfer learning requires carefully balancing information sharing against task heterogeneity. Li, Cai and Hongzhe \cite{Li_Cai_Hongzhe_2022} develop a minimax theory for high-dimensional linear regression with multiple sources, showing that optimal rates can be achieved when the target model is close to a subset of auxiliary models. Extensions address partially overlapping features \cite{chang2024heterogeneous}, benign overfitting in overparameterized settings \cite{kim2025transfer}, algorithmic analyses via approximate message passing \cite{wang2025glamp}, and minimax-optimal rates for additive models \cite{moon2025minimax}.

\medskip
\noindent\textit{\underline{Non- and semiparametric regression.}}
In nonparametric settings, similarity is encoded at the level of regression functions, enabling flexible but more delicate transfer mechanisms. Closest to our work is the framework of Cai and Pu \cite{Cai_Pu_2024_nonparam}, which introduces an explicit \emph{transfer condition} requiring the target function to be well approximated in $\L^2$ by a low-complexity combination of source functions. Related contributions include kernel-based transfer \cite{wang2023minimax}, smoothness-adaptive and robust approaches \cite{lin2024smoothness,lin2025model}, source-function weighting \cite{lin2023source}, and semiparametric representation transfer with valid inference \cite{he2024representation}.

\subsection{Organization of the paper}
The remainder of the paper is organized as follows. Section~\ref{sec:prob_and_algo} introduces the statistical framework, formalizes the local linear transfer assumption, and describes the proposed estimation procedure. In particular, we define the class of admissible tessellations, specify the transfer models, and detail the corresponding data-driven selection strategy. Section~\ref{sec:risk_bounds} contains the main theoretical results. We establish matching minimax upper and lower bounds for target regression estimation within the local linear transfer model, given in Theorems~\ref{Thm:fixed_tessellation_rate} and \ref{Thm:transfer_lower_bound}. We further derive oracle inequalities covering both well-specified and misspecified regimes. In the compositional setting, we further derive oracle risk bounds both for the estimation of the target regression function (Theorem~\ref{Thm:global_g_risk_fixed_y}) and for the transfer function itself (Theorem~\ref{Thm:compositional_rate_oracle}). Section~\ref{sec:experiments} presents illustrative examples, and Section~\ref{sec:conclusion} concludes. Proofs are deferred to the supplementary material in Appendix Section.

\section{Problem formulation and algorithm}\label{sec:prob_and_algo}


\subsection{Statistical setting}
Let $d \in \mathbb{N} = \{1,2,\dots\}$.  
Consider two datasets of differing sizes and qualities. The first dataset, called the \textit{source sample}, contains a large number $n_{\cS}$ of low-quality data points. The second dataset, referred to as the \textit{target sample}, consists of high-quality data but is smaller in size, with $n_{\cT}$ data points.  

\noindent
Let $\mathcal{D}_{\cS}$ denote the source sample, comprising input/output pairs $\big((X_i, Y_i),\, i \in \cS\big)$, and let $\mathcal{D}_{\cT}$ denote the target sample, with pairs $\big((X_i, Y_i),\, i \in \cT\big)$.  

\noindent
For all $i \in \cS \cup \cT$, the inputs satisfy $X_i \in [0,1]^d$ and the outputs $Y_i \in \mathbb{R}$. The data follow the regression models
\begin{equation}\label{Eq:source_problem}
Y_i = \source \left(X_i\right) + \varepsilon_i, \quad i \in \cS,
 \end{equation}
and
\begin{equation}\label{Eq:target_problem}
Y_i = \target(X_i) + \varepsilon_i, \quad i \in \cT,
\end{equation}
where the noise terms $\varepsilon_i$ are independent. They satisfy $\mathbb{E}[\varepsilon_i | X_i] = 0$, $\operatorname{Var}[\varepsilon_i | X_i]\le \sigma_{\cS}^2$ ($i \in \cS$) and $\operatorname{Var}[\varepsilon_i | X_i]\le \sigma_{\cT}^2$ ($i \in \cT$).\\

\noindent
The goal is to estimate the target regression function $\target:[0,1]^d\to\R$ using both the target and the source samples, and in particular by leveraging an estimate of the source regression function $\source:[0,1]^d\to\R$. As emphasized in the introduction, transferring information from the source to the target requires a suitable \emph{transferability condition}.
\\
Our main structural hypothesis is that the feature space $[0,1]^d$ can be partitioned into finitely many cells $\cA_\ell^\star$, within which the relationship between the source and target regression functions is simpler and exhibits a localized low-complexity structure.
More precisely, we consider two forms of transfer assumptions: a strong \emph{structural connection assumption}, under which the target function is obtained from the source function through a cell-specific transformation $\gsl$ (Assumption~\ref{Ass:transfer_function}), and a more flexible \emph{local linear transfer assumption} that only requires a suitable local smoothness or low-complexity link between the two functions on each cell (Assumption~\ref{Ass:local_transfer}). The former corresponds to a compositional model and enables sharper statistical rates as well as estimation of the transfer maps themselves, while the latter accommodates broader, potentially non-compositional forms of similarity.

\begin{ass}[Compositional model assumption]\label{Ass:transfer_function}
There exists a partition of $[0,1]^d$ into cells
$H^\star=\{\cA_\ell^{\star} : \ell\in[L^{\star}]\}$ such that for all $\ell\in[L^{\star}]$ there exists a function $\gsl:\R\to\R$ satisfying
\begin{equation}\label{Eq:transfer_function}
\forall x\in\mathcal A_\ell^{\star},\qquad
\target(x)=\gsl(\source(x)).
\end{equation}
We define the associated \emph{transfer function}
$g:(x,y)\in[0,1]^d\times\R\mapsto\sum_{\ell\in [L^\star]}\gsl(y)\mathbf 1_{\cA_{\ell}^\star}(x)$.
\end{ass}

\begin{ass}[Local linear transfer   assumption]\label{Ass:local_transfer}
There exists a partition of $[0,1]^d$ into cells
$H^\star=\{\cA_\ell^{\star} : \ell\in[L^{\star}]\}$ such that for all $\ell\in[L^{\star}]$ there exist functions $a^\star,b^\star:[0,1]^d\to\R$, constants
$\frlloc>0$ and $\betaloc\in[0,1]$, such that for all
$x,x'\in[0,1]^d$,
\begin{equation}\label{Eq:local_transfer}
\big|\target(x)-\big(a^\star(x')(\source(x)-\source(x'))+b^\star(x')\big)
\big|
\le
\frlloc|x-x'|^{1+\betaloc}.
\end{equation}
\end{ass}

\begin{rem}[Interpretation and comparison of the transfer assumptions]

The tessellation $H^\star$ should be understood as an \textit{idealized representation of heterogeneous similarity}, not as a literal partition known to the practitioner. The transfer hypothesis formalizes the existence of regions in the covariate space where the target regression function admits a simple relationship with its source counterpart, while allowing this relationship to deteriorate elsewhere. It does not require global similarity between the source and target distributions, nor does it assume that the partition is uniquely defined or identifiable.\\
Assumption~\ref{Ass:transfer_function} is structural: on each cell $\cA_\ell^\star$, the target depends on the covariates only through the scalar quantity $\source(x)$, with no regularity imposed on the transfer function itself. In contrast, Assumption~\ref{Ass:local_transfer} is a local smoothness condition: near any point $x'$, the target regression function is well approximated by an affine function of $\source(x)$, up to a H\"older remainder in the ambient variable $x$.\\
The two assumptions are complementary. The former enforces a strong piecewise compositional structure but allows arbitrary irregularity of the transfer map, while the latter imposes smooth local behavior without requiring a global compositional representation. Moreover, if the compositional assumption is strengthened with sufficient smoothness (e.g.\ if $g_\ell^\star\in\cC^{1+\beta_\ell}$ and $\source\in\cC^{1}$ with non-degenerate gradient) then a Taylor expansion shows that Assumption~\ref{Ass:local_transfer} holds locally within each cell, with $\beta_{\mathrm{loc}}=\beta_\ell$.\\
Overall, the role of these assumptions is purely statistical: they specify structural conditions under which information transfer is possible and quantify, through oracle inequalities, how deviations from this idealized structure affect the achievable risk. The proposed procedures are fully data-driven and designed to adapt to such latent structure without requiring explicit knowledge of the tessellation.
\end{rem}

\noindent
\subsubsection*{\textbf{well-specified compositional model}}
\label{subsec:well-specified}

The transfer learning framework is said to be \emph{well-specified in the compositional sense} if there exists a tessellation 
$H^\star=\{\cA_\ell^\star:\ell\in[L^\star]\}$ such that 
Assumption~\ref{Ass:transfer_function} holds on each cell. 
In this regime, the source-target relationship is exactly captured by a cellwise transfer function, and no systematic model error is present. The model is \emph{misspecified} if no admissible tessellation satisfies 
Assumption~\ref{Ass:transfer_function} exactly, resulting in an irreducible approximation bias. 
Our analysis makes this bias explicit through oracle inequalities that balance estimation and approximation errors.

\subsubsection*{\textbf{well-specified local linear transfer model}} \label{subsec:well-specified_local}
We say that the framework is \emph{well-specified in the local linear transfer sense} if there exists a tessellation 
$H^\star=\{\cA_\ell^\star:\ell\in[L^\star]\}$ such that  Assumption~\ref{Ass:local_transfer} holds on each cell. 
It is misspecified if no such partition exists.

\subsection{Algorithm}
As the target tessellation on which Assumption~\ref{Ass:transfer_function} or Assumption \ref{Ass:local_transfer} holds is unknown, we select a partition from a suitable class of \emph{admissible tessellations}, defined below. To this end, the target sample is split into two independent subsamples of (almost) equal size $|\cT_2|\simeq |\cT_1|=\lfloor n_{\cT}/2\rfloor$. The first subsample, $\mathcal D_{\mathcal T_1}=\{(X_i,Y_i),\,i\in\cT_1\}$, plays the role of a \emph{training sample} and is used to estimate the cellwise transfer functions associated with each candidate tessellation. The second subsample, $\mathcal D_{\mathcal T_2}=\{(X_i,Y_i),\,i\in\cT_2\}$, is reserved for \textit{model selection} and is used exclusively to choose the tessellation via empirical risk minimization.\\
We fix an integer $L_{\max}>0$, which represents the maximal number of cells allowed in the procedure. In the well-specified case, we assume that $L^\star \le L_{\max}$; otherwise, $L_{\max}$ simply acts as a complexity cap on the admissible tessellations.
\begin{definition}[Admissible tessellation class]\label{Def:admissible-partitions}
Let $h>0$ be the bandwidth used in the local linear transfer estimation.
Let $\mathcal H$ be a collection of \emph{tessellations}
$H = (\cA_{H,\ell})_{\ell\in[L_H]}$, where each $\cA_{H,\ell}$ is a \emph{cell} and
$L_H\in\N$ denotes the number of cells. We say that a tessellation $H$ is \emph{admissible} if it satisfies the following conditions:
\begin{enumerate}
\item[(i)] \textbf{Minimum cell mass:}  there exists $\frc_{\mathrm{mass}}>0$ such that for all $\ell\in[L_H]$,
\begin{equation*}
|\cT_1^{H,\ell}|
\ge
\frc_{\mathrm{mass}}\, n_{\mathcal T_1}\, h^{d},
\end{equation*}
where $\cT_1^{H,\ell} := \{i\in\cT_1 : X_i \in \cA_{H,\ell}\}$.
\item[(ii)] \textbf{Locality radius:} there exists $\frc_{\mathrm{rad}}>0$ such that for all $\ell\in[L_H]$,
\begin{equation*}
\operatorname{diam}(\cA_{H,\ell}) \le \frc_{\mathrm{rad}}h.
\end{equation*}
\item[(iii)] \textbf{Regular shape:}
There exists a constant $\frr_{\mathrm{loc}}>0$ such that, for each cell $\cA_{H,\ell}$, one can find a point $\xhl\in \cA_{H,\ell}$ such that
\begin{equation*}
\rB_d\big(\xhl,\frr_{\mathrm{loc}}h\big)
\subseteq
\cA_{H,\ell}.
\end{equation*}
The point $\xhl$ is referred to as the \emph{representative point} of the cell $\cA_{H,\ell}$. Since the cells need not admit a natural geometric center, we simply assume - without loss of generality - that $\xhl$ serves as a center, or more precisely an \emph{anchor point}, for $\cA_{H,\ell}$.
\end{enumerate}
\end{definition}

\noindent 
The essence of our transfer procedure $(\mathrm{TL})^2$ lies in combining a global source estimator with \emph{locally adapted transfer functions} whose spatial organization is learned from the data (see Algorithm \ref{Algo:transfer_fixed_tessellation}). Our main assumption is that the relationship between the source and target regression functions can be described by \emph{locally affine transformations} of the source estimator $\wsource$ (Step 1), with parameters that may vary across the covariate space. To capture this spatial heterogeneity, the covariate space is partitioned using an admissible tessellation $H$. On each cell $\cA_{H,\ell}$, a local transfer function $\widehat g_{H,\ell}$ is estimated (Step 2) using a first subsample of target data. This function aligns the source estimator with the target responses in a neighborhood of the cell center, leveraging both spatial proximity and similarity in source predictions, and yields a transfer estimator $\wtess$ (Step 3).\\
In a second stage (see Algorithm \ref{Algo:transfer}), model complexity is controlled by clustering (Step 1) similar cellwise transfer behaviors and selecting, among tessellations $H$ of the resulting size, the one minimizing a target empirical risk (Step 2). This step balances bias and variance by adapting the granularity of localization to the data, i.e., the size of the local neighborhoods (or bandwidth) used for estimation.
\\
Overall, the method jointly adapts the \emph{local transfer corrections} and the \emph{spatial tessellation}, interpolating between global transfer learning and fully local estimation.
\begin{algorithm2e}
\DontPrintSemicolon
\SetAlgoLongEnd 
\caption{Transfer Learning estimation on an admissible tessellation $H\in\cH$}
\label{Algo:transfer_fixed_tessellation}
\KwData{Source sample $\mathcal D_{\cS}$, target training sample $\mathcal D_{\cT_1}$, fixed tessellation $H\in\cH$ with cells $\{\cA_{H,\ell}\}_{\ell\in[L_H]}$ and anchor points $\{\xhl\}_{\ell\in[L_H]}$.}
\KwResult{Transfer estimator on $H$: $\wtess$.}

\textbf{Step 1: Source Nadaraya-Watson estimator}\\
Define
\begin{equation*}
\wsource(x)
=\frac{\sum_{i\in\cS} K_{h_\cS}(\|X_i-x\|)Y_i}
{\sum_{i\in\cS} K_{h_\cS}(\|X_i-x\|)},
\qquad
h_{\cS} = n_{\cS}^{-\frac{1}{2d+\beta_{\cS}}}.
\end{equation*}

\textbf{Step 2: Local transfer estimation on each cell}\\
\For{$\ell=1,\dots,L_H$}{
Estimate $(\widehat a_{H,\ell},\widehat b_{H,\ell})$ by solving
\begin{multline*}
(\widehat a_{H,\ell},\widehat b_{H,\ell})
\in
\argmin_{a,b\in\mathbb R}
\frac{1}{n_{H,\ell}}
\sum_{\substack{i\in\cT_{1}^{H,\ell}}}
\Big(Y_i - a\big(\wsource(X_i)-\wsource(x_{H,\ell})\big) - b\Big)^2
\\\times K_{x,h}(\|X_i-x_{H,\ell}\|)K_{z,\ov{h}}(|\wsource(X_i)-\wsource(x_{H,\ell})|),
\end{multline*}
where $\cT_1^{H,\ell}=\{i\in\cT_1:X_i\in\cA_{H,\ell}\}$ and
$n_{H,\ell}:=|\cT_1^{H,\ell}|$.\\
Define the cell-wise transfer function
\begin{equation*}
\widehat g_{H,\ell}(y,y_\ell)
=
\widehat a_{H,\ell}(y-y_\ell)+\widehat b_{H,\ell}.
\end{equation*}
}


\textbf{Step 3: Transfer estimator on $H$}\\
Let $\ell_{H}(x)$ be the index of the cell of $H$ containing $x$. Define the transfer estimator on the tessellation $H$:
\begin{equation*}
\wtess(x)
=
\widehat g_{H,\ell_{H}(x)}
\Big(
\wsource(x),\,
\wsource\big(x_{H,\ell_{H}(x)}\big)
\Big).
\end{equation*}
\end{algorithm2e}

\begin{algorithm2e}
\DontPrintSemicolon
\SetAlgoLongEnd 
\caption{Tessellation-Localized Transfer Learning (TL)$^2$}
\label{Algo:transfer}
\KwData{Target validation sample $\mathcal D_{\cT_2}$, cellwise transfer function
$\widehat g_{H,\ell}$ for each candidate tessellation $H=\{\cA_{H,\ell}\}_{\ell\in [L_H]}\in\cH$ and each $\ell\in [L_H]$.}
\KwResult{Transfer estimator $\wTL$}

\textbf{Step 1: Select best tessellation}\\
Choose
\begin{equation*}
\widehat H
\in
\argmin_{\substack{H\in\cH:\\ L_H\le L_{\max}}}
\frac{1}{|\cT_2|}
\sum_{i\in\cT_2}
\Big(
Y_i - \widehat g_{H,\ell_H(X_i)}\big(\wsource(X_i),\wsource(x_{H,\ell_H(X_i)})\big)
\Big)^2.
\end{equation*}

\textbf{Step 2: Final transfer estimator}\\
Let $\ell_{\widehat H}(x)$ be the cell of $\widehat H$ containing $x$. Define the transfer estimator:
\begin{equation*}
\wTL(x)
=
\widehat g_{\widehat H,\ell_{\widehat H}(x)}
\Big(
\wsource(x),\,
\wsource(x_{\widehat H,\ell_{\widehat H}(x)})
\Big).
\end{equation*}
\end{algorithm2e}

\begin{rem}[$(\rT\rL)^2$]
The notation $(\rT\rL)^2$ naturally conveys the idea of two layers of \emph{localization} and \emph{transfer}.  
The first “localize-transfer” step is implemented in Algorithm \ref{Algo:transfer_fixed_tessellation}, where source information is locally transferred to the target by estimating cellwise affine corrections on a fixed tessellation. Algorithm \ref{Algo:transfer} performs a second localization, this time at the level of the transfer structure itself, by selecting a tessellation whose cells correspond to homogeneous transfer behaviors. No new transfer functions are learned at this stage; rather, the previously estimated cellwise transfer functions are selected and applied according to the chosen tessellation to produce the final estimator.
\end{rem}

\section{Risk bounds for the transfer estimators}\label{sec:risk_bounds}
\subsection{Assumptions}
In this section, we introduce a set of assumptions that we will need to establish some or all of the following results: upper and lower bounds for the transfer estimator, and an upper bound for the transfer functions $g_\ell$.\\

\noindent
The following Assumptions~\ref{Ass:design}, \ref{Ass:regularity}, and \ref{Ass:noise} are standard in nonparametric regression and are adapted here to a setting involving two regression problems: a source problem~\eqref{Eq:source_problem} and a target problem~\eqref{Eq:target_problem}.
\begin{ass}[Design]\label{Ass:design}
We impose the following two conditions on the sampling distributions:

\begin{subassumption}[Target design]\label{Ass:design_target}
The target design points $X_i$ ($i\in \cT$) are i.i.d.
with density $p_{\cT}$ satisfying 
$p_{\cT}\in[p_{\cT}^{\min},p_{\cT}^{\max}]$ where $p_{\cT}^{\max}\ge p_{\cT}^{\min}>0$.
\end{subassumption}
\begin{subassumption}[Source design]\label{Ass:design_source}
The source design points $X_i$ ($i\in \cS$) are i.i.d.
with density $p_{\cS}$ satisfying 
$p_{\cS}\in[p_{\cS}^{\min},p_{\cS}^{\max}]$ where $p_{\cS}^{\max}\ge p_{\cS}^{\min}>0$.
\end{subassumption}
\end{ass}

\begin{ass}\label{Ass:regularity}
Recall that, for $\beta>0$, $\frl>0$, and a set $E\subset\R^d$, the H\"older class
$\mathrm{\text{H\"ol}}(\beta,\frl;E)$ consists of functions $f:E\to\R$ that are
$\lfloor\beta\rfloor$ times continuously differentiable on $E$ and whose partial
derivatives of order $\lfloor\beta\rfloor$ satisfy
\[
\bigl|\partial^\alpha f(x)-\partial^\alpha f(y)\bigr|
\le
\frl\,\|x-y\|^{\beta-\lfloor\beta\rfloor},
\qquad
\forall\,x,y\in E,
\]
for all multi-indices $\alpha$ with $|\alpha|=\lfloor\beta\rfloor$.
\begin{subassumption}[Source function regularity]\label{Ass:regularity_source}
The source regression function satisfies
\[
\source \in \mathrm{\text{H\"ol}}(\beta_{\cS},\frl_{\cS};[0,1]^d).
\]
\end{subassumption}
\begin{subassumption}[Target function regularity]\label{Ass:regularity_target}
The target regression function satisfies
\[
\target \in \mathrm{\text{H\"ol}}(\beta_{\cT},\frl_{\cT};[0,1]^d).
\]
\end{subassumption}
\end{ass}

\begin{ass}\label{Ass:noise}
As a reminder, the noise terms $\varepsilon_i$ ($i\in \cS\cup \cT$) are independent and satisfy $\mathbb{E}[\varepsilon_i | X_i] = 0$. Moreover, for any $i\in \cS\cup \cT$, the random variable $\varepsilon_i$ is sub-exponential, i.e.
\begin{equation*}
\|\varepsilon_i\|_{\psi_1}\le \sigma_{\cS}\; (i\in \cS)\quad \text{and} \quad \|\varepsilon_i\|_{\psi_1}\le \sigma_{\cT} \; (i\in \cT),
\end{equation*}
for some constants $\sigma_{\cS},\sigma_{\cT}>0$, where $\|\cdot\|_{\psi_1}$ is the Orlicz $\psi_1$-norm defined for any real random variable $X$ by
\begin{equation*}
\|X\|_{\psi_1} = \inf \left\{ C > 0 \; : \; \mathbb{E}\left[ \exp\left( \frac{|X|}{C} \right) \right] \leq 2 \right\}.
\end{equation*}
\end{ass}

\begin{ass}[Kernels]\label{Ass:kernels}
The kernels $K:\R_+\to\R$, $K_x:\R_+\to\R$ and $K_z:\R_+\to\R$ are bounded, nonnegative, and
compactly supported on $[0,1]$. Moreover, they are symmetric and satisfy
\begin{equation*}
\int_{\R} K(u)du = 1,\quad \int_{\R} K_x(u)du = 1,
\quad\text{and}\quad
\int_{\R} K_z(u)du = 1.
\end{equation*}
Moreover, $K_z$  is Lipschitz on $\R$ with constant $\frl_{K_z}$has a finite second moment:
\begin{equation*}
0 < \mu_2(K_z) := \int_{\R} u^2 K_z(u)\,du < \infty.
\end{equation*}
We define the rescaled kernels $K_{h}=h^{-d}K(\cdot/h)$,
$K_{x,h}=h^{-d}K_x(\cdot/h)$ and
$K_{z,\overline h}=\overline h^{-1}K_z(\cdot/\overline h)$ where $h>0$ and $\overline h>0$ are  bandwidths.
\end{ass}

\noindent
To support our minimax bound, let us introduce the function class
\begin{multline}\label{eq:minimax_function_class}
\mathcal{F}(H^\star,\beta_{\cS},\beta_{\cT},\beta_{\mathrm{loc}})
=\Big\{(\source,\target):
\source\in\mathrm{\text{H\"ol}}(\beta_{\cS},\frl_{\cS};[0,1]^d),
\ \target\in \mathrm{\text{H\"ol}}(\beta_{\cT},\frl_{\cT};[0,1]^d),\\
\text{and Assumption~\ref{Ass:local_transfer} holds on } H^\star
\text{ with exponent }\beta_{\mathrm{loc}}
\Big\}.
\end{multline}
The class $\mathcal{F}(H^\star,\beta_{\cS},\beta_{\cT},\beta_{\mathrm{loc}})$
corresponds to a \nameref{subsec:well-specified_local} transfer learning model associated with the
oracle tessellation $H^\star$. Indeed, for any $(\source,\target)$ in this class,
the local linear transfer relationship encoded in
Assumption~\ref{Ass:local_transfer} holds exactly on each cell of $H^\star$,
with transfer regularity $\beta_{\mathrm{loc}}$.\\
\noindent
In the following sections, we measure performance using the regression risk
\begin{equation*}
\cR(g)
:=
\mathbb E\left[(Y-g(X))^2\right],
\end{equation*}
for measurable functions $g:[0,1]^d\to\mathbb R$.
Under the regression model $Y=\target(X)+\varepsilon$ with
$\mathbb E[\varepsilon|X]=0$, we have 
\begin{equation*}
\cR(g)-\cR(\target)
=
\|g-\target\|_{\L^2(\mu_X)}^2,
\end{equation*}
so that excess risk coincides with squared $\L^2(\mu_X)$-error.

\noindent
To decompose and interpret the error bound, we introduce auxiliary functions.
Fix a tessellation $H\in\mathcal H$.
The \emph{population cellwise transfer linearization} is defined by
\begin{equation}
\label{Eq:gslin_def_repeat}
\gslin :
x\in[0,1]^d
\mapsto 
a_{H,\ell_H(x)}^\star
\bigl(
\source(x)-\source(x_{\ell_H(x)})
\bigr)
+
b_{H,\ell_H(x)}^\star,
\end{equation}
where $\ell_H(x)$ denotes the index of the cell of $H$ containing $x$.
The function $\gslin$ provides a cellwise linear approximation of the transfer relation
$g^\star\circ \source = \target$ on the tessellation $H$. In general, $\gslin$ does not necessarily coincide with the optimal transfer representation unless Assumption~\ref{Ass:transfer_function} holds for $H$.

\medskip
\noindent
We further define the \emph{source oracle} as
\begin{equation}
\label{Eq:source_oracle_def_repeat}
\gabh :
x\in[0,1]^d
\mapsto
\overline a_{H,\ell_H(x)}
\bigl(
\source(x)-\source(x_{\ell_H(x)})
\bigr)
+
\overline b_{H,\ell_H(x)},
\end{equation}
where $\overline\theta_{H,\ell}=(\overline a_{H,\ell},\overline b_{H,\ell})$
denotes the cellwise least-squares estimator computed from the target sample $\cT_1$,
while retaining access to the \emph{true} source function $\source$.
The difference $\gslin-\gabh$ thus isolates the statistical error due to the estimation
of the transfer coefficients.
The transfer estimator $\wtess$ on the tessellation $H$ additionally replaces $\source$
with its nonparametric estimator $\wsource$.

\medskip
\noindent
For any fixed tessellation $H$, the following deterministic decomposition holds:
\begin{align}
\label{eq:error_decomposition}
\cR(\wtess)-\cR(\target)
\nonumber
&=\|\target-\wtess\|_{\L^2(\mu_X)}^2\\
&\le
2\|\target - \gslin\|_{\L^2(\mu_X)}^2
+ 2\|\gslin - \gabh\|_{\L^2(\mu_X)}^2 \notag\\
&\qquad
+ 2 \|\gabh-\wtess\|_{\L^2(\mu_X)}^2 \\
&=: 2\mathrm{Approx}(H)
+ 2\mathrm{Fit}_{\cT_1}(H)
+ 2\mathrm{Plug}_{\cS}(H). \notag
\end{align}

\noindent
We have the following interpretation: the term $\mathrm{Approx}(H)$ quantifies the \emph{transfer bias} induced by approximating
the target regression function $\target$ with a piecewise linear transfer model on $H$.
The quantity $\mathrm{Fit}_{\cT_1}(H)$ captures the estimation error arising from replacing
the population transfer parameters $(a_{H,\ell}^\star,b_{H,\ell}^\star)$ with their
empirical counterparts estimated from the target sample $\cT_1$.
Finally, $\mathrm{Plug}_{\cS}(H)$ measures the additional error incurred by substituting
the unknown source function $\source$ with its nonparametric estimator $\wsource$
within the transfer model.
\\

\noindent
Beyond requiring the tessellations to be admissible, we also need the corresponding partition of $[0,1]^d$ to satisfy certain regularity conditions and to ensure an appropriate distribution of the target data within the cells.

\begin{ass}[Local design regularity on each cell]\label{Ass:local-design}
We impose the following local regularity condition on the design.

\begin{subassumption}\label{Ass:H_local-design}
\textbf{On admissible tessellations (local Gram regularity).}
For any admissible tessellation $H\in\cH$, there exist constants
$0<\lambda_{H,\min}\le \lambda_{H,\max}<\infty$ such that for all
$\ell\in[L_H]$,
\begin{equation*}
\lambda_{H,\min} I_2
\;\preceq\;
\frac{1}{|\cT_1^{H,\ell}|}
(\Psi_{H,\ell})^\top \Psi_{H,\ell}
\;\preceq\;
\lambda_{H,\max} I_2,
\end{equation*}
where $\Psi_{H,\ell}$ is the $n_{H,\ell}\times 2$ design matrix with rows
$\phi_{H,\ell}(X_i)^\top$, $i\in\cT_1^{H,\ell}$, and
\begin{equation}\label{eq:feature_map_def}
\phi_{H,\ell}(x)
=
\bigl(
1,\;
\source(x)-\source(x_{H,\ell})
\bigr)^\top,
\end{equation}
with $\cT_1^{H,\ell}:=\{i\in \cT_1 : X_i\in\cA_{H,\ell}\}$.

\smallskip
\noindent
\textbf{Weighted Gram condition on the stability event.}
On the event $\cE_{\mathrm{ess}}$ (defined in \eqref{Eq:event_ESS}), the corresponding
\emph{weighted} Gram matrices
\begin{equation*}
G_{H,\ell}
:=
\frac{1}{n_{H,\ell}}(\Psi_{H,\ell})^\top W_{H,\ell}\Psi_{H,\ell},
\qquad
W_{H,\ell}:=\mathrm{diag}(w_{i,\ell})_{i\in\cT_1^{H,\ell}},
\end{equation*}
satisfy, for all $H\in\cH$ and all $\ell\in[L_H]$,
\begin{equation*}
\lambda_{H,\min} I_2
\;\preceq\;
G_{H,\ell}
\;\preceq\;
\lambda_{H,\max} I_2.
\end{equation*}
\end{subassumption}

\begin{subassumption}\label{Ass:star_local-design}
\textbf{On the target tessellation.}
There exist constants
$0<\lambda_{\min}^\star\le \lambda_{\max}^\star<\infty$
such that for all $\ell\in[L^\star]$,
\begin{equation}
\lambda_{\min}^\star I_2
\preceq
\frac{1}{|\cT_1^{\star,\ell}|}
(\Psi_\ell^\star)^\top \Psi_\ell^\star
\preceq
\lambda_{\max}^\star I_2,
\end{equation}
where $\Psi_\ell^\star$ is the $n_\ell\times 2$ design matrix with rows
$\phi_\ell^\star(X_i)^\top$ ($i\in\cT_1^{\star,\ell}$), and
\begin{equation}
\phi_\ell^\star(x)
=
\bigl(
1,
\source(x)-\source(x_{\ell})
\bigr)^\top,
\end{equation}
with $\cT_1^{\star,\ell}:=\{i\in \cT_1 : X_i\in\cA_\ell^\star\}$.
\end{subassumption}
\end{ass}

\medskip
\noindent
For any $\ell\in[L_H]$, define
\begin{equation}
n_{H,\ell}
:=
\sum_{i\in \cT_1} \mathbf 1_{\{X_i\in\cA_{H,\ell}\}}
\quad\text{and}\quad
p_{H,\ell}
:=
\mathbb P(X\in \cA_{H,\ell}).
\end{equation}

\noindent
For a fixed cell $\ell$, we write $m_{H,\ell}:=g_\ell\circ\source$ for the
corresponding regression function on $\cA_{H,\ell}$.
We assume that $m_{H,\ell}$ admits a local linear expansion at the representative
point $x_{H,\ell}$:
there exist $\theta^\star_{H,\ell}=(a^\star_{H,\ell},b^\star_{H,\ell})\in\R^2$ and a remainder
$r_{H,\ell}(\,\cdot\,):=r_{H,\ell}(\,\cdot\,;x_{H,\ell})$ such that, for all
$u\in\cA_{H,\ell}$,
\begin{equation}\label{eq:local_lin_decomp}
m_{H,\ell}(u)
=
(\theta^\star_{H,\ell})^\top\phi_{H,\ell}(u)
+
r_{H,\ell}(u),
\end{equation}
where the feature map is defined by \eqref{eq:feature_map_def}.
Last we need:
\begin{ass}[Uniform boundedness of local features and residuals]
\label{Ass:bounded_features}
There exist constants $\phi_{\max}, r_{\max} < \infty$ such that, for all
$H \in \cH$, all $\ell \in [L_H]$ and all $x \in \cA_{H,\ell}$,
\begin{equation*}
\|\phi_{H,\ell}(x)\|_2 \le \phi_{\max},
\qquad
|r_{H,\ell}(x)| \le r_{\max}.
\end{equation*}
\end{ass}

\begin{ass}[Cellwise lower-mass condition]\label{Ass:mass_cell}
For all $\delta\in(0,1)$ and all $H\in\cH$,
\begin{equation}
n_{\cT_1}
\min_{\ell\in[L_H]} p_{H,\ell}
\;\ge\;
8\log\Big(\frac{L_H}{\delta}\Big).
\end{equation}
\end{ass}

\begin{rem}[On Assumption \ref{Ass:star_local-design}]\label{rem:gram-regularity}
Assuming local Gram invertibility uniformly over all 
$H\in\mathcal H$ simply means that $\mathcal H$ is restricted to a class 
of \emph{admissible} tessellations whose cells are regular enough for local 
linear estimation.  
This type of condition is standard in the theory of 
tessellation-based estimators: see, for example, 
Gy\"orfi et al. (\cite{gyorfi2002distribution}, Chapter 12), Scornet et al. \cite{scornet2015consistency}, 
and Wager and Athey \cite{wager2018estimation}, where the analysis excludes 
degenerate cells to ensure identifiability of local fits.  
Our assumption avoids stronger geometric or density conditions on the function $\source$ and isolates precisely the regularity needed for the 
local estimator to be well-posed.
\end{rem}

\noindent
To express the approximation and estimation bounds purely in terms of the number of
cells $L_H$, we impose the following mild regularity condition on the geometry of
the tessellations, which rules out highly irregular partitions with a few large
cells and many tiny ones.
\begin{ass}[Quasi-uniform mesh condition]
\begin{subassumption}\label{Ass:quasi_uniform} \textbf{On admissible tessellations.}
There exists a constant $\frc_\Delta>0$ such that for all $H\in\cH$,
\begin{equation*}
\Delta_{\max}(H):=\max_{\ell\in[L_H]}\Delta_{H,\ell}
\le
\frc_\Delta\,L_H^{-1/d},
\end{equation*}
where $\Delta_{H,\ell}:=\mathrm{diam}(\cA_{H,\ell})$.
\end{subassumption}
\begin{subassumption}\label{Ass:quasi_uniform_target} \textbf{On the target tessellation.} There exist constants $0<\frc_{\min}^\star\le \frc_{\max}^\star<\infty$ such that
\begin{equation}
\frc_{\min}^\star(L^\star)^{-1/d}
\le
\Delta_{\min}(H^\star)
\le
\Delta_{\max}(H^\star)
\le
\frc_{\max}^\star(L^\star)^{-1/d},
\end{equation}
where $\Delta_{\min}(H^\star){=}\min_{\ell\in[L^\star]}\mathrm{diam}(\cA_\ell^\star)$ and $\Delta_{\max}(H^\star){=}\max_{\ell\in[L^\star]}\mathrm{diam}(\cA_\ell^\star)$.
\end{subassumption}
\end{ass}

\begin{ass}[Effective sample size and plug-in regime]\label{Ass:ESS_plugin}
For any fixed admissible tessellation $H$ and cell
$\cA_{H,\ell}$, define the effective sample size count
\begin{equation*}
N_{H,\ell}=\#\Big\{i\in\cT_1 :
\|X_i-x_{H,\ell}\|\le h, 
|\wsource(X_i)-\wsource(x_{H,\ell})|\le \overline h
\Big\}.
\end{equation*}
\begin{subassumption}\label{subass:ESS}
Fix $\delta\in(0,1)$.
There exist constants
$\frc_1^{\mathrm{ess}},\frc_2^{\mathrm{ess}}>0$
such that, with probability at least $1-\delta$, the following event holds:
\begin{equation}
\frac{1}{\frc_2^{\mathrm{ess}}}
n_{\cT_1} h^d \overline h
\le
\sup_{H\in\cH,\cA_{H,\ell}} N_{H,\ell}
\le
\frc_2^{\mathrm{ess}}
n_{\cT_1} h^d \overline h .
\end{equation}
We define the event
\begin{equation}\label{Eq:event_ESS}
\cE_{\mathrm{ess}}=\Big\{\forall H\in\cH,\forall \ell\in[L_H]:\frac{1}{\frc_2^{\mathrm{ess}}}
n_{\cT_1} h^d \overline h
\le
N_{H,\ell}
\le
\frc_2^{\mathrm{ess}}
n_{\cT_1} h^d \overline h\Big\}.
\end{equation}
In particular, the effective sample size of the local weighted estimator
within each cell is of order $n_{\cT_1} h^d \overline h$, uniformly over
$(H,\ell)$.
\end{subassumption}
\begin{subassumption}\label{subass:plug-in}
We further assume the \textit{plug-in regime}
\begin{equation}\label{eq:plug-in_condition}
n_{\cT_1} h^d \overline h
\ge
\frc_1^{\mathrm{ess}}
\log\Big(\frac{1}{\delta}\Big).
\end{equation}
\end{subassumption}
\end{ass}

\subsection{Nonasymptotic risk upper bounds}

\subsubsection{Oracle rate in a well-specified compositional model}

\noindent
We begin by considering an idealized, well-specified setting in which the source-target relationship follows the compositional model - corresponding to Assumption \eqref{Ass:transfer_function} - on the oracle tessellation. This regime is a natural starting point because the structural connection $\target = g_\ell\circ \source$ on each cell not only yields the strongest form of transfer, but also makes the transfer functions $g_\ell$ identifiable and estimable from the data (see Subsection~\ref{Subsec:transfer_function_bound}). Under additional regularity assumptions on the maps $(g_\ell)_{\ell\in[L^\star]}$, the resulting transfer estimator leverages this low-dimensional structure to achieve a fast oracle rate.

\begin{ass}\label{Ass:regularity_transfer}
For all $\ell\in[L^\star]$, the transfer function
$g_\ell^\star$ belongs to the H\"older class
$\mathrm{\text{H\"ol}}(\beta_g,\frl_g;\R)$ for some $\beta_g>1$ and
$\frl_g>0$.
\end{ass}

\begin{thm}[Oracle rate under a \nameref{subsec:well-specified}]
\label{Thm:compositional_rate_oracle}
Assume that Assumption~\ref{Ass:transfer_function} holds on the oracle
tessellation $H^\star$, and that Assumptions~\ref{Ass:design},
\ref{Ass:regularity_source}, \ref{Ass:noise}, \ref{Ass:kernels},
\ref{Ass:local-design}, \ref{Ass:quasi_uniform}, and
\ref{Ass:regularity_transfer} hold, together with the plug-in
condition~\eqref{eq:plug-in_condition}.
Then the transfer estimator
$\wTL=\wh{f}_{\cT}^{H^\star}$ associated with the oracle tessellation
satisfies the following bound in expectation:
\begin{equation}\label{Eq:oracle_upper_bound_plugin_expect}
\E\big[\cR(\wTL)-\cR(\target)\big]
\lesssim
(L^\star)^{-2\beta_g\beta_{\cS}/d}
+
\frac{1+\log(L^\star)}{n_{\cT_1}h^d\overline h}
+
\bigl(1+\log(L^\star)\bigr)^2\,
n_{\cS}^{-\frac{2\beta_{\cS}}{2\beta_{\cS}+d}}.
\end{equation}
\end{thm}
\noindent
The oracle rate established above will serve as a benchmark in the minimax
lower bound analysis of Section~\ref{Subsec:lower_bound}, where we show that no
estimator can uniformly outperform this rate, even when the oracle
tessellation is known.\\
\noindent
While Theorem~\ref{Thm:compositional_rate_oracle} provides an
oracle benchmark under strong structural assumptions, such conditions are
often unrealistic in practice. In order to obtain robustness with respect
to model misspecification and heterogeneous transfer relationships, we now
replace the global compositional assumption by a weaker local regularity
condition.

\subsubsection{Upper bounds for a fixed tessellation in a local linear transfer model}

\noindent
This subsection constitutes a first step toward establishing the overall risk bound for the transfer estimator, which will later be combined with the empirical risk minimization (ERM) argument developed in the subsequent subsection. Under the local linear transfer Assumption \eqref{Ass:local_transfer} (corresponding to the local linear transfer model motivated by the minimax rate we aim to achieve) we derive a high-probability risk bound for the transfer estimator associated with an arbitrary fixed tessellation.

\begin{thm}[Fixed tessellation rate in \nameref{subsec:well-specified_local}]
\label{Thm:fixed_tessellation_rate}
Let $H\in\cH$ be an admissible tessellation.
Suppose that Assumptions~\ref{Ass:local_transfer},
\ref{Ass:design}, \ref{Ass:regularity}, \ref{Ass:noise},
\ref{Ass:kernels}, \ref{Ass:bounded_features}, \ref{Ass:local-design}, \ref{Ass:quasi_uniform} and
\ref{subass:ESS} hold.
Then, for all $\delta\in(0,1)$ such that
$n_{\cT_1}h^d\overline h\gtrsim \log(1/\delta)$,
the transfer estimator $\wtess$ associated with $H$ satisfies, with
probability at least $1-\delta$,
\begin{multline}\label{Eq:bound_per-tessellation-L2}
\cR(\wtess)-\cR(\target)
\lesssim
L_H^{-2(1+\beta_{\mathrm{loc}})/d}
+
\frac{1}{n_{\cT_1}h^d\overline h}
\log\Big(\frac{L_H}{\delta}\Big)\\
+
\log\Big(\frac{L_H}{\delta}\Big)
\Bigl(1+\log\Big(\frac{L_H}{\delta}\Big)\Bigr)
\,n_{\cS}^{-\frac{2\beta_{\cS}}{2\beta_{\cS}+d}} .
\end{multline}
\end{thm}

\begin{rem}[Interpretation of the error decomposition]
The bound~\eqref{Eq:bound_per-tessellation-L2} decomposes the excess risk
into three main contributions:
\textbf{(i)} a transfer approximation error governed by the local smoothness
$\beta_{\mathrm{loc}}$,
\textbf{(ii)} a target-side estimation error controlled by the effective
sample size $n_{\cT_1}h^d\overline h$, and
\textbf{(iii)} a source plug-in error corresponding to the minimax rate for
estimating the source regression function.
\end{rem}

\noindent
Integrating the high-probability bound of
Theorem~\ref{Thm:fixed_tessellation_rate} with respect to the confidence
parameter yields the following expectation bound.

\begin{cor}[Expectation bound for a fixed tessellation]
\label{Cor:fixed_tessellation_rate_exp}
Let $H\in\cH$ be an admissible tessellation and suppose that the assumptions
of Theorem~\ref{Thm:fixed_tessellation_rate} hold. Then,
\begin{equation}\label{Eq:bound_per-tessellation-L2_explicitlogs}
\E\big[\cR(\wtess)-\cR(\target)\big]
\lesssim
L_H^{-2(1+\beta_{\mathrm{loc}})/d}
+
\frac{1+\log(L_H)}{n_{\cT_1}h^d\overline h}
+
\bigl(1+\log(L_H)\bigr)^2\,
n_{\cS}^{-\frac{2\beta_{\cS}}{2\beta_{\cS}+d}}.
\end{equation}
\end{cor}

\begin{proof}
Apply Theorem~\ref{Thm:fixed_tessellation_rate} with $\delta=e^{-t}$ and
integrate the resulting tail bound over $t\ge 0$.
\end{proof}

\medskip
Finally, under the quasi-uniformity assumption, a natural choice of
bandwidths is $h^d\overline h\asymp 1/L_H$, which leads to the following
simplified rate.

\begin{cor}[Effective sample size regime]
\label{Cor:fixed_tessellation_rate_plugin}
Let $H\in\cH$ be an admissible tessellation.
Fix $\delta\in(0,1)$ and suppose that the assumptions of
Theorem~\ref{Thm:fixed_tessellation_rate} hold.
Assume moreover that Assumption~\ref{Ass:ESS_plugin} holds at level
$\delta$, and that $h^d\overline h\asymp 1/L_H$.
Then, with probability at least $1-\delta$,
\begin{equation}\label{Eq:bound_per-tessellation-L2_plugin}
\cR(\wtess)-\cR(\target)
\lesssim
L_H^{-2(1+\beta_{\mathrm{loc}})/d}
+
\frac{L_H}{n_{\cT_1}}
\log\Big(\frac{L_H}{\delta}\Big)
+
\log\Big(\frac{L_H}{\delta}\Big)
\Bigl(1+\log\Big(\frac{L_H}{\delta}\Big)\Bigr)
\,n_{\cS}^{-\frac{2\beta_{\cS}}{2\beta_{\cS}+d}} .
\end{equation}
\end{cor}

\begin{proof}
The result follows directly from
Theorem~\ref{Thm:fixed_tessellation_rate} upon substituting
$h^d\overline h\asymp 1/L_H$.
\end{proof}

\subsubsection{Empirical risk minimization and risk bound for the transfer estimator}

\noindent
We now analyze the additional error induced by the data-driven selection of an admissible tessellation. While the previous section provides risk bounds for transfer estimators with a \emph{fixed} tessellation, in practice the tessellation is chosen from a \textbf{finite} collection $\cH$ using an independent validation sample. This model selection step introduces an additional estimation error, which we quantify below.\\
For any $H\in\cH$, let $\wtess$ denote the transfer estimator constructed
on the tessellation $H$, and define its population risk by
\begin{equation}
\cR(H)
:=
\cR(\wtess)
=
\E\left[(Y-\wtess(X))^2\right].
\end{equation}
Let $\wh\cR(H)$ denote the corresponding empirical risk computed on the
validation sample $\cT_2$.
The selected tessellation is defined via the empirical risk minimization
rule
\begin{equation}
\wh H \in \argmin_{H\in\cH} \wh\cR(H).
\end{equation}

\medskip
\noindent
\noindent
\textbf{Oracle inequality in expectation.}
We first establish an oracle inequality for this empirical selection
procedure in expectation.
This result relies only on finite-moment assumptions and is therefore
compatible with sub-exponential noise.

\begin{prop}[ERM oracle inequality in expectation]\label{Prop:ERM_expectation}
Let $H^{\mathrm{or}}\in\arg\min_{H\in\cH}\cR(H)$.
Assume that $\E[\varepsilon^4]<\infty$ and that the validation sample $\cT_2$ is
independent of the data used to construct $\{\wtess:H\in\cH\}$.
Assume moreover that Assumptions~\ref{Ass:noise},
\ref{Ass:H_local-design}, \ref{Ass:bounded_features} and
\ref{Ass:ESS_plugin} hold.
Then there exists a universal constant $\frc>0$ such that
\begin{equation*}
\E\left[\cR(\wh H)\right]
\le
\cR(H^{\mathrm{or}})
+
\frc\,\sqrt{\frac{|\cH|}{n_{\cT_2}}}.
\end{equation*}
\end{prop}

\begin{proof}
The proof is postponed to Appendix Section \ref{App:ERM_expectation}
\end{proof}

\medskip
\noindent
We now combine the above oracle inequality with the fixed-tessellation
risk bounds (Theorem \ref{Thm:fixed_tessellation_rate}) derived previously to obtain a bound for the selected
transfer estimator.

\begin{thm}[Oracle inequality for the transfer estimator]\label{thm:transfer-L2}
Assume that the class $\cH$ of admissible tessellations is finite.
Suppose that Assumptions~\ref{Ass:local_transfer}, \ref{Ass:design},
\ref{Ass:regularity}, \ref{Ass:noise},~\ref{Ass:local-design} and \ref{Ass:bounded_features} hold. Then there exists a constant $\frc>0$ such that
\begin{equation}\label{eq:transfer-risk-final}
\E\!\left[\cR(\wTL)-\cR(\target)\right]
\le
\frc\inf_{H\in\cH}
\Bigg\{
L_H^{-2(1+\beta_{\mathrm{loc}})/d}
+
\frac{L_H}{n_{\cT}}
+
n_{\cS}^{-\frac{2\beta_{\cS}}{2\beta_{\cS}+d}}
\Bigg\}
+
\frc\,\sqrt{\frac{|\cH|}{n_{\cT}}}.
\end{equation}
Here $\wTL := \wtarget^{\wh H}$ denotes the transfer estimator associated
with the selected tessellation $\wh H$.
\end{thm}

\noindent
The bound~\eqref{eq:transfer-risk-final} is an oracle inequality showing that the selected transfer estimator achieves, up to a model-selection penalty, the performance of the best estimator associated with a tessellation in $\cH$.

\medskip
\noindent
\textbf{High-probability selection via median-of-means.}
While the expectation bound above is sufficient for our purposes, a
high-probability oracle inequality can be obtained under the same moment
assumptions by replacing the empirical risk $\wh\cR(H)$ with a
median-of-means (MoM) version.
To this end, partition the validation sample $\cT_2$ into $B$ disjoint
blocks of equal size.
For any $H\in\cH$, let $\wh{\cR}_{\mathrm{MoM}}(H)$ denote the median of the
blockwise empirical means of the squared validation loss
\begin{equation*}
\cL_H(X,Y):=\bigl(Y-\wtess(X)\bigr)^2 .
\end{equation*}
The selected tessellation is then defined by
\begin{equation}
\wh H_{\mathrm{MoM}}
\in
\argmin_{H\in\cH}\wh{\cR}_{\mathrm{MoM}}(H).
\end{equation}

\medskip
\noindent
The following proposition provides a high-probability oracle inequality
for the MoM-based selection step.

\begin{prop}[Oracle inequality for MoM-based Tessellation selection]
\label{Prop:ERM_MoM}
Assume that $\E[\varepsilon^4]<\infty$ and that Assumptions~\ref{Ass:bounded_features} and \ref{Ass:H_local-design} hold.
Then there exist universal constants $\frc,\frc'>0$ such that for any $\delta\in(0,1)$ and any integer
$B$ satisfying
\begin{equation*}
B \ge \frc\,\log\!\left(\frac{|\cH|}{\delta}\right),
\end{equation*}
the median-of-means selected tessellation $\wh H_{\mathrm{MoM}}$ satisfies, with probability at least $1-\delta$,
\begin{equation*}
\cR(\wh H_{\mathrm{MoM}})
\le
\min_{H\in\cH}\cR(H)
+
\frc'
(\sigma^2+\cR_{\max})
\sqrt{\frac{\log(|\cH|/\delta)}{n_{\cT_2}}},
\end{equation*}
where $\cR_{\max}:=\max_{H\in\cH}\cR(H)$ and $\sigma^2:=\E[\varepsilon^2]$.
\end{prop}
\begin{proof}
 The proof is postponed to Appendix section \ref{App:ERM_MoM} 
\end{proof}

\medskip
\noindent
We finally combine Proposition~\ref{Prop:ERM_MoM} with the
high-probability fixed-tessellation bounds (Theorem \ref{Thm:fixed_tessellation_rate}) established earlier.

\begin{thm}[Oracle risk bound for the MoM-selected transfer estimator]
\label{Thm:transfer-risk-final-MoM}
Assume that the class $\cH$ of admissible tessellations is finite.
Fix $\delta\in(0,1)$.
Suppose that the assumptions of
Theorem~\ref{Thm:fixed_tessellation_rate} hold for every $H\in\cH$.
Assume moreover that Assumption~\ref{Ass:ESS_plugin} holds at level
$\delta/(2|\cH|)$ and that $\E[\varepsilon^4]<\infty$.
\noindent
Let $\wh H_{\mathrm{MoM}}$ be defined as above with an integer $B$
satisfying
\begin{equation}
B \ge \frc\,\log\!\left(\frac{2|\cH|}{\delta}\right).
\end{equation}
Then, with probability at least $1-\delta$,
\begin{multline}\label{eq:transfer-risk-final-MoM}
\cR(\wtarget^{\wh H_{\mathrm{MoM}}})-\cR(\target)
\lesssim
\inf_{H\in\cH}
\Bigg\{
L_H^{-2(1+\beta_{\mathrm{loc}})/d}
+
\frac{L_H}{n_{\cT}}
+
n_{\cS}^{-\frac{2\beta_{\cS}}{2\beta_{\cS}+d}}
\Bigg\}
\\+
\frc''
(\sigma^2+\cR_{\max})
\sqrt{\frac{\log(2|\cH|/\delta)}{n_{\cT}}},
\end{multline}
for some universal constant $\frc''>0$.
\end{thm}

\noindent
The bound~\eqref{eq:transfer-risk-final-MoM} is a high-probability oracle
inequality showing that the selected transfer estimator achieves, up to a
model-selection penalty, the same performance as the best estimator
associated with a tessellation in $\cH$.

\subsection{Minimax lower bound in the local linear transfer model}\label{Subsec:lower_bound}

Specifically, we work under the local linear transfer model and consider the associated minimax function class defined in \eqref{eq:minimax_function_class}. We assume that the true tessellation
$H^\star={\cA_\ell^\star : \ell \in [L^\star]}$, on which Assumption~\ref{Ass:local_transfer} holds, is known. Within this idealized setting, we characterize the fundamental limits of target regression estimation over the prescribed function class. Moreover, since centered Gaussian noise $\mathcal N(0,\sigma^2)$ is sub-Gaussian and therefore sub-exponential, it is sufficient to establish the minimax lower bound under the Gaussian noise submodel. Any lower bound derived under this restriction applies \emph{a fortiori} to the broader class of sub-exponential noise distributions.

\begin{ass}[Balanced cell allocation]
\label{Ass:lower_mass-balance}
For any $\ell\in[L^\star]$, let $p_\ell^\star := \mathbb{P}(X\in \cA_\ell^\star)$.
There exist constants $0<\mathfrak a\le \mathfrak b<\infty$ such that
\begin{equation}
\frac{\mathfrak a}{L^\star} \le p_\ell^\star \le \frac{\mathfrak b}{L^\star}.
\end{equation}
\end{ass}

\begin{thm}[Oracle lower bound for transfer risk]\label{Thm:transfer_lower_bound}
Assume that the source and target regression functions satisfy
Assumption~\ref{Ass:regularity} with parameters $\beta_{\cS}=\beta_{\cT}=\beta$.
Assume further that the oracle tessellation
$H^\star=\{\cA_\ell^\star : \ell \in [L^\star]\}$ and the associated set of
representative points $\{x_{H^\star,\ell} : \ell\in[L^\star]\}$ are known, and that
Assumptions~\ref{Ass:local_transfer}, \ref{Ass:design}, \ref{Ass:star_local-design}, and
\ref{Ass:lower_mass-balance} hold. Then there exists a constant $\frc>0$, depending only on
$d,\beta,\sigma,A_{\max},B_{\max}$ and the constants appearing in
Assumptions~\ref{Ass:design}, \ref{Ass:star_local-design}, and
\ref{Ass:lower_mass-balance}, such that
\begin{equation*}
\inf_{\widehat f}
\sup_{(\source,\target)\in\cF(H^\star,\beta,\beta,\beta_{\mathrm{loc}})}
\Big[
\cR(\widehat f)-\cR(\target)
\Big]
\ge
\frc\bigg[
\frac{\sigma^2 L^\star}{n_{\cT}}
+
(\Delta_{\min}(H^\star))^{2(1+\beta_{\mathrm{loc}})}
+
n_{\cS}^{-\frac{2\beta}{2\beta+d}}
\bigg].
\end{equation*}
In particular, the same lower bound applies to the transfer estimator
$\wTL$ defined in Algorithm~\ref{Algo:transfer}:
\begin{equation*}
\sup_{(\source,\target)\in\cF(H^\star,\beta,\beta,\beta_{\mathrm{loc}})}
\Big[
\cR(\wTL)-\cR(\target)
\Big]
\ge
\frc\bigg[
\frac{\sigma^2 L^\star}{n_{\cT}}
+
(\Delta_{\min}(H^\star))^{2(1+\beta_{\mathrm{loc}})}
+
n_{\cS}^{-\frac{2\beta}{2\beta+d}}
\bigg].
\end{equation*}
\end{thm}

\noindent
As an immediate consequence, we obtain the following specialization for
quasi-uniform tessellations.

\begin{cor}
\label{cor:lower-bound-quasi-uniform}
Suppose the assumptions of Theorem~\ref{Thm:transfer_lower_bound} and Assumption \ref{Ass:quasi_uniform_target} hold. Then, there exists a constant $\frc>0$ such that
\begin{equation}
\inf_{\widehat f}
\sup_{(\source,\target)\in\cF(H^\star,\beta,\beta,\beta_{\mathrm{loc}})}
\big(\cR(\widehat f)-\cR(\target)\big)
\;\ge\;
\frc\Bigg[
\frac{\sigma^2 L^\star}{n_{\cT}}
+
(L^\star)^{-\frac{2(1+\beta_{\mathrm{loc}})}{d}}
+
n_{\cS}^{-\frac{2\beta}{2\beta+d}}
\Bigg].
\end{equation}
\end{cor}
\noindent
\textbf{Minimax optimality.}
Combining Corollary~\ref{Cor:fixed_tessellation_rate_plugin} with
Corollary~\ref{cor:lower-bound-quasi-uniform}, we conclude that, under the \nameref{subsec:well-specified_local} and the
quasi-uniform tessellation assumption, and up to logarithmic factors,
the proposed transfer estimator $\wTL$ is minimax rate-optimal over
$\cF(H^\star,\beta,\beta,\beta_{\mathrm{loc}})$.

\begin{rem}[On the role of the dimension]\label{Rem:dimension}
For a fixed tessellation $H^\star$, the target-side contribution
\[
\frac{\sigma^2 L^\star}{n_{\cT}}
+
(\Delta_{\min}(H^\star))^{2(1+\beta_{\mathrm{loc}})}
\]
is parametric in $n_{\cT}$. In this sense, when the tessellation is fixed
and known, transfer learning removes the curse of dimensionality on the
target side.\\
If $L^\star$ is allowed to grow with $n_{\cT}$, for instance under a
regular tessellation with $\Delta_{\min}(H^\star)\asymp (L^\star)^{-1/d}$,
balancing the above terms yields the rate
\[
n_{\cT}^{-\frac{2(1+\beta_{\mathrm{loc}})}{2(1+\beta_{\mathrm{loc}})+d}},
\]
which exhibits an explicit dependence on the ambient dimension $d$ and is
unavoidable without further structural assumptions.
\end{rem}

\subsection{Estimation of the transfer function under an oracle tessellation}
\label{Subsec:transfer_function_bound}

In this section, we study the estimation error of the transfer function in an oracle setting corresponding to the \nameref{subsec:well-specified}, where the true tessellation $H^\star={\cA_\ell^\star:\ell\in[L^\star]}$ satisfying Assumption~\ref{Ass:transfer_function} is assumed to be known. This assumption allows us to focus exclusively on the statistical complexity of estimating the transfer map, abstracting away from the additional error induced by tessellation selection.
This assumption allows us to focus exclusively on the statistical complexity of
estimating the transfer map, abstracting away from the additional error induced by
tessellation selection. The transfer function admits the cellwise representation
\begin{equation}
g(x,y)
:=
\sum_{\ell\in[L^\star]} g_\ell^\star(y)\mathbf 1_{\cA_\ell^\star}(x).
\end{equation}
Accordingly, for each $\ell\in[L^\star]$ we consider the regression model
\begin{equation}\label{Eq:regression_wsource}
Y_i=g_\ell^\star\big(\source(X_i)\big)+\varepsilon_i,
\qquad
i\in\cT_1,
\end{equation}
where $\E[\varepsilon_i| X_i]=0$ and the noise on the target sample satisfies
Assumption~\ref{Ass:noise}, namely $\|\varepsilon_i\|_{\psi_1}\le\sigma_{\cT}$ for all
$i\in\cT_1$.
\noindent
Since the true source score $\source(X_i)$ is unknown, we approximate it by its nonparametric estimator $\wsource(X_i)$ and estimate $g_\ell^\star$ via a local regression of $Y_i$ on this plug-in covariate. This induces an additional source plug-in error that will be quantified below.
Using the estimated source score $\wsource$, we define the estimator
\begin{equation}
\widehat g_{\wsource}(x,y)
:=
\sum_{\ell\in[L^\star]} \widehat g_{\wsource,\ell}(y)\mathbf 1_{\cA_\ell^\star}(x),
\qquad
\widehat g_{\wsource,\ell}(y)
:=
\widehat a_\ell\big[y-\wsource(x_{\ell})\big]+\widehat b_\ell,
\end{equation}
where $(\widehat a_\ell,\widehat b_\ell)$ is obtained by local least squares around the
representative point $x_{\ell}$ of the cell $\cA_\ell^\star$, namely
\begin{multline}\label{Eq:LS_regression_wsource_df}
(\widehat b_{\ell},\widehat a_{\ell})
\in
\argmin_{b,a\in\R}
\Bigg\{\sum_{i\in\cT_1}
\Big[Y_i- a\big(\wsource(X_i)-\wsource(x_{\ell})\big)-b\Big]^2\\
\times
K_{x,h}(\|X_i-x_{\ell}\|)
\, K_{z,\overline h}\big(|\wsource(X_i)-\wsource(x_{\ell})|\big)
\Bigg\}.
\end{multline}
\noindent
Here, $K_x:\R_+\to\R$ and $K_z:\R_+\to\R$ are bounded kernels supported on $[0,1]$
and satisfying Assumption~\ref{Ass:kernels}.
The bandwidths $h>0$ and $\overline h>0$ define the rescaled kernels
$K_{x,h}=h^{-d}K_x(\cdot/h)$ and
$K_{z,\overline h}=\overline h^{-1}K_z(\cdot/\overline h)$.

\noindent
Throughout this section, we assume that
Assumptions~\ref{Ass:design_source}, \ref{Ass:regularity_source},
\ref{Ass:noise}, and~\ref{Ass:kernels} hold.
In particular, these assumptions ensure uniform control of the source estimator
$\wsource$ in sup-norm.
Specifically, for any $\delta_{\cS}\in(0,1)$, there exist constants
$\frc_{\cS}',\frc_{\cS}>0$ and an event $\cE_{\cS}$ with
$\P(\cE_{\cS})\ge 1-\delta_{\cS}$ on which
\begin{equation}\label{Eq:epsS_rate}
\epsilon_{\cS}
:=
\|\wsource-\source\|_\infty
\le
\frc_{\cS}'
\Bigg(
\frac{\log(\frc_{\cS}/\delta_{\cS})}{n_{\cS}}
\Bigg)^{\frac{\beta_{\cS}}{2\beta_{\cS}+d}}.
\end{equation}

\begin{thm}[Pointwise risk bound for the transfer map]\label{Thm:global_g_risk_fixed_y}
Let $X$ be a target covariate with distribution $\mu_{\cT}$ (density $p_{\cT}$),
independent of the source sample $\cD_{\cS}$. Suppose Assumptions
\ref{Ass:design_source}, \ref{Ass:regularity_source}, \ref{Ass:noise}, and
\ref{Ass:kernels} hold.
Assume moreover that Assumptions \ref{Ass:star_local-design} and \ref{Ass:ESS_plugin} hold on each cell $\ell\in[L^\star]$ for the local estimators
\eqref{Eq:LS_regression_wsource_df}, and define for $x\in[0,1]^d$, $y_{\ell^\star(x)}:=\wsource(x_{\ell^\star(x)})$ where $\ell^\star(x)$ is the unique index such that $x\in\cA_{\ell^\star(x)}^\star$. Fix $y\in\R$ and consider the localization event
\begin{equation}\label{Eq:Ey_def}
\cE_y
:=
\Big\{|y-y_{\ell(X)}|\le \overline h\Big\}.
\end{equation}
Then, on the event $\cE_{\cS}\cap \cE_{\mathrm{ess}}\cap \cE_y$, where
$\cE_{\mathrm{ess}}$ is defined by \eqref{Eq:event_ESS}, we have
\begin{equation}\label{Eq:global_g_risk_fixed_y_bound}
\E\Big[\big(\widehat g_{\wsource}(X,y)-g(X,y)\big)^2\, \big|\, \cD_{\cS},\cD_{\cT_1}\Big]
\lesssim
\frl_g^2
\Bigg(
\frac{\log(c_{\cS}/\delta_{\cS})}{n_{\cS}}
\Bigg)^{\frac{2\beta_g\beta_{\cS}}{2\beta_{\cS}+d}}
+
\frl_g^2\,\overline h^{2\beta_g}
+
\frac{\sigma_{\cT}^2}{n_{\cT_1}h^d\overline h},
\end{equation}
where the implicit constant depends only on $\frc_2^{\mathrm{ess}}$,
$\lambda_0$, and the kernel envelopes $\|K_x\|_\infty,\|K_z\|_\infty$. In particular, with the choice $\oh = (n_{\cT_1}h^d)^{-1/(2\beta_g+1)}$, the bound \eqref{Eq:global_g_risk_fixed_y_bound} yields, on
$\cE_{\cS}\cap \cE_{\mathrm{ess}}\cap \cE_y$,
\begin{equation}\label{Eq:global_g_risk_fixed_y_rate}
\E\Big[\big(\widehat g_{\wsource}(X,y)-g(X,y)\big)^2\,\big|\, \cD_{\cS},\cD_{\cT_1}\Big]
\lesssim
\Bigg(
\frac{\log(c_{\cS}/\delta_{\cS})}{n_{\cS}}
\Bigg)^{\frac{2\beta_g\beta_{\cS}}{2\beta_{\cS}+d}}
+
(n_{\cT_1}h^d)^{-\frac{2\beta_g}{2\beta_g+1}}.
\end{equation}
Moreover, if $\sup_{\ell\in[L^\star]}|y-y_\ell|\le \overline h$, then $\cE_y$ holds
automatically and the same bound is valid (on $\cE_{\cS}\cap \cE_{\mathrm{ess}}$).
\end{thm}

\begin{proof}
The proof is postponed to Appendix section \ref{proof:thm_global_g_risk}
\end{proof}

\noindent\textit{Connection with the global excess-risk bound.} 
Theorem ~\ref{Thm:global_g_risk_fixed_y} controls the pointwise-in-$y$ estimation
error of the transfer function, averaged over the target covariate $X$.
In the subsequent analysis, this bound is integrated over the random argument
$y=\wsource(X)$ and combined with the approximation and parametric estimation
errors arising from the tessellation structure.
This decomposition yields the source plug-in term appearing in the global
excess-risk bounds for the transfer estimator, as stated in
Theorems~\ref{Thm:fixed_tessellation_rate} and~\ref{thm:transfer-L2}.

\section{Experiments and applications}\label{sec:experiments}

\noindent
\noindent
In this section, we present numerical experiments designed to illustrate both the performance and the limitations of the proposed transfer learning approach. Throughout the study, performance is assessed in terms of the error reduction
\begin{equation}\label{eq:error_reduction}
\rm E_{\rm red}:=\frac{\rm MSE_{NW}-MSE_{(TL)^2}}{MSE_{NW}},
\end{equation}
where $\mathrm{MSE}_{\mathrm{NW}}$ and $\mathrm{MSE}_{(\mathrm{TL})^2}$ denote, respectively, the mean squared error of the classical Nadaraya-Watson estimator of $\target$ computed on the full target sample $\mathcal{T}$ (i.e., without transfer), and that of the proposed $(\mathrm{TL})^2$ estimator defined in Algorithm~\ref{Algo:transfer}. In practice, values of $\mathrm{E}_{\mathrm{red}}$ close to $1$ indicate highly effective (positive) transfer, whereas negative values correspond to negative transfer, meaning that incorporating source information deteriorates estimation accuracy.
\\
\noindent
The first step of the procedure consists in specifying the collection $\mathcal{H}$ of admissible tessellations. In this section, we focus on \textit{axis-aligned square} tessellations. More precisely, letting $d$ denote the dimension of the regressors, each cell of the partition is of the form
\begin{equation*}
\prod_{i=1}^d \left( \frac{k_i}{n_{\mathcal T}}, \frac{k_i+1}{n_{\mathcal T}} \right], 
\qquad 0 \le k_i \le n_{\mathcal T}-1.
\end{equation*}
As a consequence, the maximal number of cells is fixed and given by $L_{\max}=n_{\mathcal{T}}^d$.

\begin{rem}
The transfer estimation procedures described in Algorithms~\ref{Algo:transfer_fixed_tessellation} and~\ref{Algo:transfer} require solving an optimization problem over a finite but potentially large collection of partitions of $[0,1]^d$. A naive exhaustive search may therefore be computationally demanding. In the present numerical study, we rely on a simulated annealing algorithm \cite{recuit_simul} to perform this optimization. This choice is purely algorithmic: the optimization strategy is independent of the proposed transfer methodology, and alternative optimization schemes could equally well be employed.
\end{rem}

\noindent
\noindent
Throughout this section, all kernels are Gaussian, that is, $K$, $K_x$, and $K_z$ are taken to be Gaussian densities (e.g., $u\in\R \mapsto (2\pi)^{-1/2} e^{-u^2/2}$). The bandwidths $h$ and $\overline h$ are chosen of order $n^{-1/3}$. Since the purpose of this section is to assess the intrinsic properties of the transfer learning procedure, the influence of tuning parameter selection is not investigated.\\
It is worth noting that this choice does not coincide with the classical bandwidth optimal for mean squared error. This departure is deliberate and better aligned with the philosophy of transfer learning, whose primary objective is often variance reduction rather than pointwise optimality. Bandwidths optimized for mean squared error may increase variance through the classical bias-variance trade-off, potentially counteracting the benefits of transfer.
\medskip

\noindent
\noindent
The results reported in this section are based on 100 Monte Carlo simulations, corresponding to 100 replications for simulated data and 100 random subsamplings in the application study. All displays report the median across repetitions. The median is preferred to the mean because, in the simulated setting, the transferred estimator can occasionally exhibit a severe increase in MSE, which would disproportionately affect the average.\\
This phenomenon appears to stem from the optimization procedure (e.g., simulated annealing) rather than from the transfer learning methodology itself. Indeed, when the number of admissible tessellations is small (so that the optimal partition can be identified explicitly) such extreme errors are never observed.

\subsection{Empirical results}

To illustrate the performance of the proposed transfer learning method and to examine its behavior as the dimension increases, we consider two synthetic regression targets defined on $[0,1]^d$ with values in $\mathbb{R}$, for varying dimensions $d \geq 1$. In both experiments, the same source regression function is used, namely $\source : x \in [0,1]^d \mapsto \|x\|^2$. Both source and target covariates (denoted respectively by $X_{\mathcal S}$ and $X_{\mathcal T}$) are generated independently from the uniform distribution on $[0,1]^d$. Source outputs are simulated, for all $i \in \mathcal S$, according to $Y_i=\source(X_i)+\varepsilon_i$ with $\varepsilon_i\sim \mathcal{N}(0,0.1)$. In contrast, target outputs are generated, for $i \in \mathcal T$, as $Y_i=f_k(X_i)+\varepsilon_i$ with $\varepsilon_i\sim \mathcal{N}(0,0.1)$ where $f_k$ denotes one of the target regression functions:\vspace{2pt}\\
\noindent
\textit{Target 1:} 
$f_1 : x \in [0,1]^d \mapsto 
\mathbf{1}_{\{x_1 \ge 1/2\}} \sin(\|x\|_2)
+ \mathbf{1}_{\{x_1 < 1/2\}} e^{\|x\|_2}$.
\\
\textit{Target 2:} 
$f_2 : x \in [0,1]^d \mapsto 
\mathbf{1}_{\{\|x\|_2 \ge 1/2\}} \sin(\|x\|_2)$.

\medskip

\noindent
In all simulations, the target sample size is fixed at $n_{\mathcal T}=20$. Consequently, admissible partitions may only split the domain at points of the form $k/20$, with $k \in \{1,\dots,20\}$. This setting is favorable for the estimation of $f_1$, as its discontinuity can be exactly captured by an admissible tessellation. In contrast, the discontinuity of $f_2$ lies on a sphere of radius $1/2$, which cannot be perfectly approximated by axis-aligned partitions, resulting in an intrinsic model misspecification. \\To further assess the robustness of the method to model misspecification, we also consider the estimation of $f_1$ using partitions split at points of the form $k/19$, with $k \in \{1,\dots,19\}$. This configuration is referred to as Target 1 (misspecification) in Figure~\ref{fig:E_red_3t}. In this case, misspecification arises because the optimal partition $H^\star$ appearing in Assumption~\eqref{Ass:transfer_function} does not belong to the class of admissible tessellations explored by the algorithm.
\noindent

\begin{figure}[ht]
\centering
\begin{tikzpicture}
\node (img) at (0,0) {\includegraphics[scale=0.4]{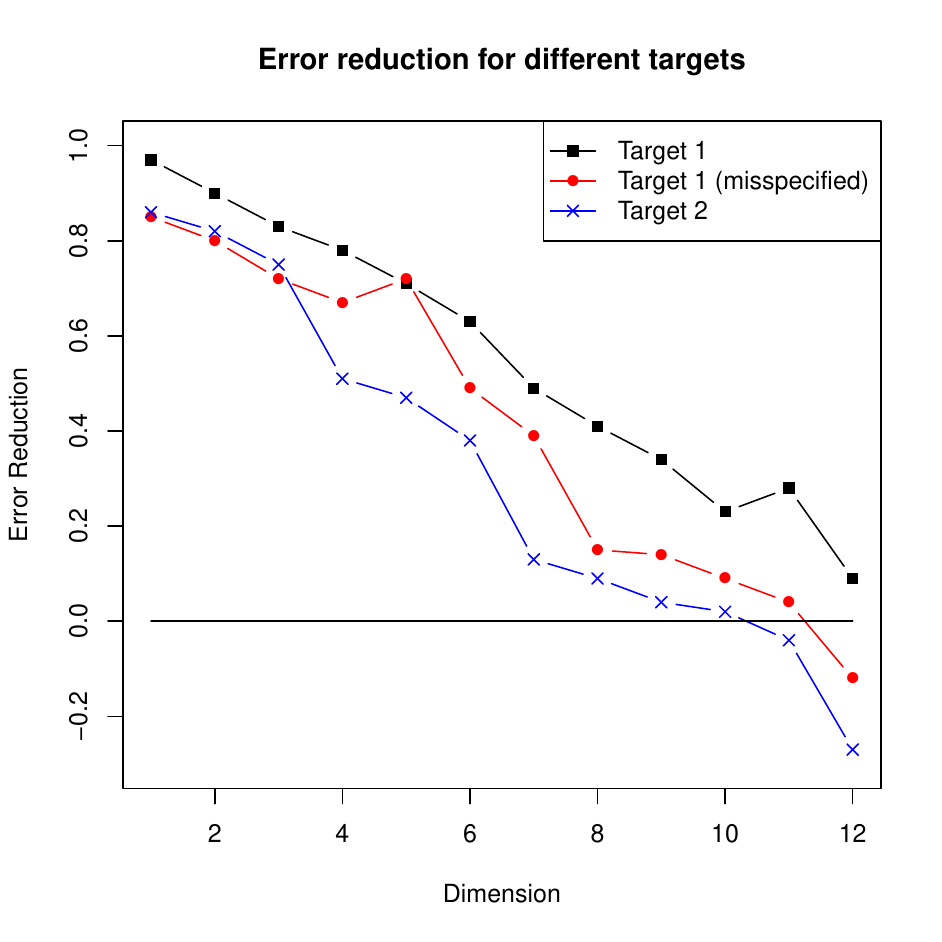}};
\draw [decorate,decoration={brace,amplitude=10pt}]
(3.5,2.6) -- (3.5,-1.0)
node[midway,right=10pt] {Positive Transfer};
\draw [decorate,decoration={brace,amplitude=10pt}]
(3.5,-1.2) -- (3.5,-2.5)
node[midway,right=10pt] {Negative Transfer};
\end{tikzpicture}
\caption{Error reduction \eqref{eq:error_reduction} for the estimation of $f_1$ and $f_2$ as a function of the regressor dimension $d$.}
\label{fig:E_red_3t}
\end{figure}

\noindent
Figure~\ref{fig:E_red_3t} reports the error reduction as a function of the dimension $d$. In these experiments, the source sample size is set to $n_{\mathcal{S}} = 100\,d$. 
\\
When the model is well specified, a substantial error reduction is observed. Nevertheless, the three curves exhibit similar qualitative behavior, decreasing monotonically as the dimension increases, highlighting the influence of the source estimator’s quality. For instance, in dimension $d=12$, only $1200$ source observations are available, which may lead to poor estimation of the source function due to the curse of dimensionality. Consequently, for Target~2 the error reduction becomes negative, indicating that transfer increases the estimation error (by approximately $26\%$ in this case).
\\
This issue can be mitigated by improving the source estimation: Table~\ref{tab:d12} shows that increasing the source sample size for $d=12$ restores a positive and increasing error reduction for Target~2.

\begin{table}
\caption{Error reduction for Target~2 as the source sample size increases, with $d = 12$. }\label{tab:d12} 
\begin{center} \begin{tabular}{|c|c|c|c|} 
\hhline{|----|}
$n_{\cS}$ & 2000 & 4000 & 6000 \\ 
\hhline{|----|}
$\rm E_{red}$ & 0.13 & 0.23 & 0.26 \\ 
\hhline{|----|}
\end{tabular} 
\end{center} 
\end{table}

\begin{rem}
These simulations illustrate the respective roles of the different terms in the error decomposition given in Equation~\eqref{eq:error_decomposition}. In particular, the decreasing behavior of the curves in Figure~\ref{fig:E_red_3t}, together with the results reported in Table~\ref{tab:d12}, highlights the impact of the term $\mathrm{Plug}_{\mathcal S}(H)$ and underscores the importance of accurately estimating the source function in order to achieve positive transfer. By contrast, the relative positions of the curves in Figure~\ref{fig:E_red_3t} reflect the influence of the approximation error $\mathrm{Approx}(H)$ and the fitting term $\mathrm{Fit}_{\mathcal T_1}(H)$. This observation emphasizes the need to identify a tessellation for which the transfer function is sufficiently smooth to enable accurate local linear estimation.
\end{rem}

\subsection{Applications}

In this section, we assess the performance of $\rm (TL)^2$ on two datasets. The first is the well-known Abalone dataset \cite{abalone}, available from the UCI Machine Learning Repository, while the second consists of daily log-returns of the AMD stock price, retrieved from the \textit{yfinance} Python package \cite{yfinance}.

\subsubsection*{Toy dataset: Abalone \cite{abalone}}
In this experiment, we consider a real-data regression problem aimed at predicting the age of abalones from a set of physical measurements. Specifically, the regression task maps $\mathbb{R}^7$ to $\mathbb{R}$, where the response variable is the abalone’s age and the covariates include its length, diameter, height, and the weights of its various organs.
The target task corresponds to estimating the age of female abalones ($n_{\mathcal T}=1307$), while observations from male abalones ($n_{\mathcal S}=1528$) are used as source data for transfer learning. In all experiments, the full source sample is used, whereas the target sample consists of $600$ females randomly subsampled from the $1307$ available observations. Performance is then assessed by computing the RMS error on the remaining $707$ female abalones.

\begin{figure}[ht]
\centering
\includegraphics[scale=0.35]{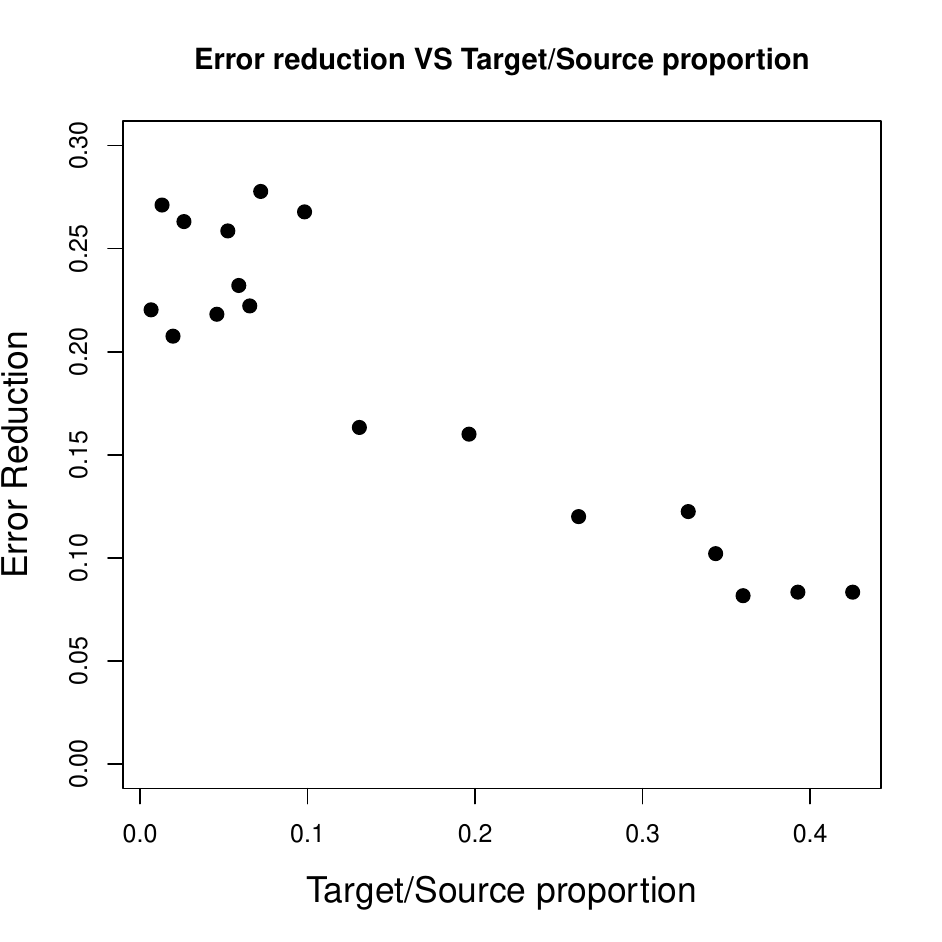}
\caption{Error reduction for the estimation of female abalone age as a function of the target sample size.}
\label{fig:abalone}
\end{figure}

\noindent
Figure~\ref{fig:abalone} illustrates the error reduction achieved by the proposed transfer learning method as the number of target observations increases. As expected, the benefit of transfer is strongest in low-sample regimes: when fewer than $100$ target observations are available, the error is reduced by approximately $26\%$. As the target sample size grows, the performance of the classical Nadaraya-Watson estimator improves, thereby reducing the marginal gain from transfer, which gradually stabilizes around $8\%$. Notably, this improvement remains positive even when the full dataset is used (about $1300$ target observations), and becomes essentially constant from a target-to-source sample size ratio of roughly $0.3$ onward.
  
\subsubsection*{Dataset : Signatures}
To illustrate the relevance of our framework on structured, high-dimensional features, we consider an application to financial time series based on signature transforms. Signatures \cite{lyons1998differential} provide a systematic nonlinear representation of sequential data, capturing temporal interactions through iterated integrals, even at low truncation orders. A key property of signature features is that a broad class of nonlinear path-dependent functionals can be approximated arbitrarily well by linear functionals of truncated signatures, as a consequence of the universality of signatures and a Stone–Weierstrass-type theorem on path space \cite{hambly2010uniqueness}. 
\\
This property explains the widespread use of linear models on top of signature features in practice \cite{chevyrev2025primer}. Although the relationship between the raw time series and the response variable may be highly nonlinear, it can often be captured by linear predictors in signature space. This makes signature-based regression particularly well suited to our transfer learning framework, which combines nonparametric estimation in the covariate space with low-complexity local linear transfer maps between source and target regression functions.
Signatures thus provide a rich feature representation in which linear transfer is expressive enough to capture local similarities between tasks, while remaining compatible with our theoretical assumptions and estimation procedures.
\\

\noindent
We consider the task of predicting AMD’s daily stock returns from its historical returns and trading volumes. The training sample comprises daily observations from 30 June 2022 to 28 July 2022 (20 observations), while the out-of-sample period spans from 29 July 2022 to 23 September 2022. Following \cite{RosenCao2023RiskOT}, we use the signature transform of log returns and log trading volumes as regression features.

\begin{definition}[Signature Transform]\label{sig_trans}
Let $X : [0,T] \to \mathbb{R}^d$ be a continuous path of bounded variation. 
The signature of $X$ up to level $M \in \mathbb{N}$ is defined as the collection of iterated integrals
\[
S(X)_{0,T}^{(M)} := \big(1, S^{(1)}(X)_{0,T}, \dots, S^{(M)}(X)_{0,T}\big),
\]
where, for each $k\in\{ 1,\dots,M\}$, the level-$k$ term is given by
\[
S^{(k)}(X)_{0,T}
:= \int_{0 < t_1 < \dots < t_k < T}
dX_{t_1} \otimes \dots \otimes dX_{t_k},
\]
which takes values in $(\mathbb{R}^d)^{\otimes k}$ and is interpreted as a \textit{Riemann-Stieltjes integral}.
\end{definition}

\noindent
More precisely, let $(l_t, v_t)_{t \ge 0}$ denote the daily return and trading volume of the AMD stock, respectively. For a fixed lag $L$, the covariate associated with the daily log return $l_t$ is defined as the signature transform of the rough path
\[
\left(\log\!\left(\frac{l_{t-i}}{l_{t-i-1}}\right),\;
      \log\!\left(\frac{v_{t-i}}{v_{t-i-1}}\right)\right)_{i=1}^{L}.
\]

\begin{table}
\caption{Error reduction for AMD stocks for different order $M$ and lags $L$.}\label{tab:signatures} 
\begin{center} \begin{tabular}{|c|c|c|c|c|c|c|} 
\hhline{|-------|}
Lag $L$ & 2 & 3 & 5 & 10 & 15 & 20 \\ 
\hhline{|-------|}
$M=2$  & 0.09 & -0.15 & 0.001 & 0.30 & -0.55 & -2.28 \\ 
\hhline{|-------|} 
$M=3$ & 0.08 & -0.02 & 0.13 & 0.36 & -0.62 & -2.61 \\ 
\hhline{|-------|} 
$M=4$ & 0.09 & -0.003 & -0.05 & 0.36 & -0.63 & -1.23\\
\hhline{|-------|}
\end{tabular} 
\end{center} 
\end{table}
\noindent
Since our theoretical framework assumes covariates supported on $[0,1]^d$, the signature features are normalized via an affine transformation so that their empirical range lies in a bounded domain, identified with $[0,1]^d$ without loss of generality.\\
Table~\ref{tab:signatures} presents the performance of the proposed transfer learning approach for different values of the truncation $M$ and the lag $L$. It reveals a strong dependence of transfer performance on the lag $L$, while showing only limited sensitivity to the truncation order $M$. The weak effect of $M$ suggests that the approximation term $\mathrm{Approx}(H)$ is already small for low-order signatures, so increasing the truncation level brings little additional benefit. In contrast, the pronounced influence of $L$ reflects a trade-off between the plug-in term $\mathrm{Plug}_{\cS}(H)$ and the fitting term $\mathrm{Fit}_{\cT}(H)$. For short lags, the lack of temporal structure limits effective transfer, whereas excessively long lags increase dimensionality and estimation error, ultimately degrading performance. The optimal results observed around $L=10$ indicate a regime in which the source function is well estimated and the transfer map remains sufficiently simple to be reliably learned. This behavior contrasts with the model studied in \cite{RosenCao2023RiskOT}, where optimal performance is achieved for much shorter lags (two or three days), highlighting the role of the underlying task and representation in determining the relevant temporal scale for transfer. Overall, these findings suggest that successful transfer primarily hinges on selecting an appropriate temporal scale, while low-order signatures already provide an adequate and robust representation.

\section{Conclusion and prospects}\label{sec:conclusion}
In this work, we introduced a regression transfer learning framework that explicitly accounts for heterogeneity in the source-target relationship. By modeling transfer locally through an oracle tessellation and cellwise transfer functions, we showed that meaningful gains can be achieved even when global similarity assumptions fail. Our analysis establishes minimax-optimal convergence rates under a well-specified local linear transfer model, together with oracle inequalities that disentangle estimation and approximation errors, thereby clarifying when and how transfer learning can mitigate the curse of dimensionality. Beyond the well-specified setting, we further show that local linear transfer yields stable and interpretable performance under mild misspecification, providing a principled safeguard against negative transfer.
\\
Despite its flexibility, the framework relies on structural assumptions that may not hold universally. In particular, the existence of a finite tessellation supporting simple local linear transfer relationships may be restrictive when source-target similarity varies continuously or lacks spatial coherence. Moreover, although our adaptive procedure removes the need to know the oracle partition, its computational cost may increase with the richness of the tessellation class, potentially limiting scalability in very high-dimensional settings. Our analysis is also restricted to regression with bounded responses; extensions to more general losses or heavy-tailed noise would require additional technical developments.
\\
Several directions for future work naturally arise. Relaxing the piecewise-constant tessellation toward smoother or hierarchical partitions could enhance flexibility while preserving interpretability. Extensions to multiple heterogeneous sources or to sequential and online transfer settings are also of interest. From a theoretical perspective, an important direction is to characterize negative transfer more precisely by identifying regimes of transferability, depending for instance on the target-to-source sample size ratio $n_{\cT}/n_{\cS}$ and on structural properties of the local linear transfer functions. Such a refined understanding could lead to sharp phase transitions delineating when transfer is beneficial, neutral, or detrimental.




\bibliographystyle{plain}
\bibliography{Transfer}


\appendix
This page provides a summary of the notation used throughout the appendix.

\begin{longtable}{@{}p{0.3\linewidth}p{0.65\linewidth}@{}}
\toprule
\textbf{Symbol} & \textbf{Description} \\
\hline
\endfirsthead

\multicolumn{2}{l}%
{{\bfseries \tablename\ \thetable{} - continued from previous page}} \\
\toprule
\textbf{Symbol} & \textbf{Description/definition} \\

\endhead

\bottomrule
\endfoot

\multicolumn{2}{l}{\textbf{General mathematical notation}} \\[2pt]
$\N$ & The set $\{1,2,\dots\}$ \\
$\N_0$ & $\N\cup\{0\}$ \\
$[n]$ & The set $\{1, 2, \dots, n\}$ \\
$\mathbf{1}(\cdot)$ & The indicator function \\
$I_d$ & Identity matrix of size $d\times d$\\
$a\lesssim b$ & $\exists\frc>0$ such that $a\le \frc b$ \\
$A\preceq B$& $B-A$ is positive semidefinite ($A,B$ matrices)\\
$\frc>0$ & Universal constant\\
$h\asymp g$ & $\exists \frc_1,\frc_2>0$ s.t. $\frc_1 g\le h\le \frc_2 g$\\
\multicolumn{2}{l}{\textbf{Source and target}} \\[2pt]
$\cD_{\cS}$ & Source dataset: $\{(X_i,Y_i),i\in\cS\}$ \\
$\cD_{\cT}$ & Target dataset: $\{(X_i,Y_i),i\in\cT\}$ \\
$\cT_1\subset \cT$ & Target training sample \\
$\cT_2\subset \cT$ & Target validation sample \\

\multicolumn{2}{l}{\textbf{Tessellation and cells}} \\[2pt]
$\cH$ & Set of tessellations. \\
$H=(\cA_{H,\ell})_{\ell\in[L_H]}$ & Tessellation. \\
$L_{\max}$ & maximum number of admissible cells\\
$\cA_{H,\ell}$ & Cell with index $\ell$ in tessellation $H$.\\
$\cT_1^{H,\ell}$ & $\cT_1^{H,\ell}=\{i\in\cT_1\,:\,X_i\in\cA_{H,\ell}\}$\\
$\Delta_{H,\ell}=\mathrm{diam}(\cA_{H,\ell})$\\

\multicolumn{2}{l}{\textbf{Estimators and oracles}} \\[2pt]
$\theta_{H,\ell}^\star=(a_{H,\ell}^\star,b_{H,\ell}^\star)$ & Population transfer parameters on $\cA_{H,\ell}$ \\
$\overline\theta_{H,\ell}=(\overline a_{H,\ell},\overline b_{H,\ell})$ & Least-squares estimator computed on $\cA_{H,\ell}$ from $\cT_1$ \\
$\gslin$ & Population cellwise transfer linearization\\
& $x
\mapsto 
a_{H,\ell_H(x)}^\star
\bigl(
\source(x)-\source(x_{\ell_H(x)})
\bigr)
+
b_{H,\ell_H(x)}^\star$\\
$\gabh$ & Source oracle \\
 & $x
\mapsto
\overline a_{H,\ell_H(x)}
\bigl(
\source(x)-\source(x_{\ell_H(x)})
\bigr)
+
\overline b_{H,\ell_H(x)}$\\
$\wsource$ & Nadarya-Watson estimator of $f_{\cS}$\\
$\widehat g_{H,\ell}$ & Local transfer function estimator on $\cA_{H,\ell}$\\
&$(y,y_\ell)\mapsto\widehat a_{H,\ell}(y-y_\ell)+\widehat b_{H,\ell}$\\
$\wtess$ & Transfer estimator of $f_{\cT}$ on tessellation $\cH$\\
& $x\mapsto
\widehat g_{H,\ell_{H}(x)}
\big(
\wsource(x),\,
\wsource(x_{H,\ell_{H}(x)})
\big)$\\
$\wTL$ & Final transfer estimator of $f_{\cT}$\\
& $x\mapsto \widehat g_{\widehat H,\ell_{\widehat H}(x)}
\big(
\wsource(x),\,
\wsource(x_{\widehat H,\ell_{\widehat H}(x)})
\big)$\\
\bottomrule

\end{longtable}

The supplement contains additional technical results and detailed proofs omitted from the main text. In particular, it includes auxiliary concentration and deviation inequalities, perturbation bounds, oracle slope controls, and proofs of Theorems 1-6. It also provides technical material on local polynomial and Nadaraya-Watson estimators used in the analysis.

\appendix

\tableofcontents

\section*{Definitions and assumptions}
\noindent
To enhance the readability of the appendix and facilitate navigation throughout the document, we restate in this section the definitions and assumptions introduced in the main text.

\subsection*{Definitions}

\begin{df}[Admissible tessellation class]\label{Def:admissible-partitions}
Let $h>0$ be the bandwidth used in the local transfer estimation.
Let $\mathcal H$ be a collection of \emph{tessellations}
$H = (\cA_{H,\ell})_{\ell\in[L_H]}$, where each $\cA_{H,\ell}$ is a \emph{cell} and
$L_H\in\N$ denotes the number of cells. We say that a tessellation $H$ is \emph{admissible} if it satisfies the following conditions:
\begin{enumerate}
\item[(i)] \textbf{Minimum cell mass:}  there exists $\frc_{\mathrm{mass}}>0$ such that for all $\ell\in[L_H]$,
\begin{equation*}
|\cT_1^{H,\ell}|
\ge
\frc_{\mathrm{mass}}\, n_{\mathcal T_1}\, h^{d},
\end{equation*}
where $\cT_1^{H,\ell} := \{i\in\cT_1 : X_i \in \cA_{H,\ell}\}$.
\item[(ii)] \textbf{Locality radius:} there exists $\frc_{\mathrm{rad}}>0$ such that for all $\ell\in[L_H]$,
\begin{equation*}
\operatorname{diam}(\cA_{H,\ell}) \le \frc_{\mathrm{rad}}h.
\end{equation*}
\item[(iii)] \textbf{Regular shape:}
There exists a constant $\frr_{\mathrm{loc}}>0$ such that, for each cell $\cA_{H,\ell}$, one can find a point $\xhl\in \cA_{H,\ell}$ such that
\begin{equation*}
\rB_d\big(\xhl,\frr_{\mathrm{loc}}h\big)
\subseteq
\cA_{H,\ell}.
\end{equation*}
The point $\xhl$ is referred to as the \emph{representative point} of the cell $\cA_{H,\ell}$. Since the cells need not admit a natural geometric center, we simply assume - without loss of generality- that $\xhl$ serves as a center, or more precisely an \emph{anchor point}, for $\cA_{H,\ell}$.
\end{enumerate}
\end{df}

\subsection*{Assumptions}

\begin{ass}\label{Ass:transfer_function}
There exists a partition of $[0,1]^d$ into cells
$H^\star=\{\cA_\ell^{\star} : \ell\in[L^{\star}]\}$ such that for all $\ell\in[L^{\star}]$ there exists a function $\gsl:\R\to\R$ satisfying
\begin{equation}\label{Eq:transfer_function}
\forall x\in\mathcal A_\ell^{\star},\qquad
\target(x)=\gsl(\source(x)).
\end{equation}
We define the associated \emph{transfer function}
\begin{equation*}
g:(x,y)\in[0,1]^d\times\R\mapsto\sum_{\ell\in [L^\star]}\gsl(y)\mathbf 1_{\cA_{\ell}^\star}(x).
\end{equation*}
\end{ass}

\begin{ass}[Local linear transfer]\label{Ass:local_transfer}
There exists a partition of $[0,1]^d$ into cells
$H^\star=\{\cA_\ell^{\star} : \ell\in[L^{\star}]\}$ such that for all $\ell\in[L^{\star}]$ there exist functions $a^\star,b^\star:[0,1]^d\to\R$, constants
$\frlloc>0$ and $\betaloc\in[0,1]$, such that for all
$x,x'\in[0,1]^d$,
\begin{equation}\label{Eq:local_transfer}
\big|\target(x)-\big(a^\star(x')(\source(x)-\source(x'))+b^\star(x')\big)
\big|
\le
\frlloc|x-x'|^{1+\betaloc}.
\end{equation}
\end{ass}

\begin{ass}[Design]\label{Ass:design}
We impose the following two conditions on the sampling distributions:

\begin{subassumption}[Target design]\label{Ass:design_target}
The target design points $X_i$ ($i\in \cT$) are i.i.d.
with density $p_{\cT}$ satisfying 
$p_{\cT}\in[p_{\cT}^{\min},p_{\cT}^{\max}]$ where $p_{\cT}^{\max}\ge p_{\cT}^{\min}>0$.
\end{subassumption}
\begin{subassumption}[Source design]\label{Ass:design_source}
The source design points $X_i$ ($i\in \cS$) are i.i.d.
with density $p_{\cS}$ satisfying 
$p_{\cS}\in[p_{\cS}^{\min},p_{\cS}^{\max}]$ where $p_{\cS}^{\max}\ge p_{\cS}^{\min}>0$.
\end{subassumption}
\end{ass}

\begin{ass}\label{Ass:regularity}
Recall that, for $\beta>0$, $\frl>0$, and a set $E\subset\R^d$, the H\"older class
$\mathrm{H\"ol}(\beta,\frl;E)$ consists of functions $f:E\to\R$ that are
$\lfloor\beta\rfloor$ times continuously differentiable on $E$ and whose partial
derivatives of order $\lfloor\beta\rfloor$ satisfy
\[
\bigl|\partial^\alpha f(x)-\partial^\alpha f(y)\bigr|
\le
\frl\,\|x-y\|^{\beta-\lfloor\beta\rfloor},
\qquad
\forall\,x,y\in E,
\]
for all multi-indices $\alpha$ with $|\alpha|=\lfloor\beta\rfloor$.
\begin{subassumption}[Source function regularity]\label{Ass:regularity_source}
The source regression function satisfies
\[
\source \in \mathrm{H\"ol}(\beta_{\cS},\frl_{\cS};[0,1]^d).
\]
\end{subassumption}
\begin{subassumption}[Target function regularity]\label{Ass:regularity_target}
The target regression function satisfies
\[
\target \in \mathrm{H\"ol}(\beta_{\cT},\frl_{\cT};[0,1]^d).
\]
\end{subassumption}
\end{ass}

\begin{ass}\label{Ass:noise}
As a reminder, the noise terms $\varepsilon_i$ ($i\in \cS\cup \cT$) are independent and satisfy $\mathbb{E}[\varepsilon_i | X_i] = 0$. Moreover, for any $i\in \cS\cup \cT$, the random variable $\varepsilon_i$ is sub-exponential, i.e.
\begin{equation*}
\|\varepsilon_i\|_{\psi_1}\le \sigma_{\cS}\; (i\in \cS)\quad \text{and} \quad \|\varepsilon_i\|_{\psi_1}\le \sigma_{\cT} \; (i\in \cT),
\end{equation*}
for some constants $\sigma_{\cS},\sigma_{\cT}>0$, where $\|\cdot\|_{\psi_1}$ is the Orlicz $\psi_1$-norm defined for any real random variable $X$ by
\begin{equation*}
\|X\|_{\psi_1} = \inf \left\{ C > 0 \; : \; \mathbb{E}\left[ \exp\left( \frac{|X|}{C} \right) \right] \leq 2 \right\}.
\end{equation*}
\end{ass}

\begin{ass}[Kernels]\label{Ass:kernels}
The kernels $K:\R_+\to\R$, $K_x:\R_+\to\R$ and $K_z:\R_+\to\R$ are bounded, nonnegative, and
compactly supported on $[0,1]$. Moreover, they are symmetric and satisfy
\begin{equation*}
\int_{\R} K(u)du = 1,\quad \int_{\R} K_x(u)du = 1,
\quad\text{and}\quad
\int_{\R} K_z(u)du = 1.
\end{equation*}
Moreover, $K_z$  is Lipschitz on $\R$ with constant $\frl_{K_z}$has a finite second moment:
\begin{equation*}
0 < \mu_2(K_z) := \int_{\R} u^2 K_z(u)\,du < \infty.
\end{equation*}
We define the rescaled kernels $K_{h}=h^{-d}K(\cdot/h)$,
$K_{x,h}=h^{-d}K_x(\cdot/h)$ and
$K_{z,\overline h}=\overline h^{-1}K_z(\cdot/\overline h)$ where $h>0$ and $\overline h>0$ are  bandwidths.
\end{ass}

\begin{ass}[Local design regularity on each cell]\label{Ass:local-design}
We impose the following local regularity condition on the design.

\begin{subassumption}\label{Ass:H_local-design}
\textbf{On admissible tessellations (local Gram regularity).}
For any admissible tessellation $H\in\cH$, there exist constants
$0<\lambda_{H,\min}\le \lambda_{H,\max}<\infty$ such that for all
$\ell\in[L_H]$,
\begin{equation*}
\lambda_{H,\min} I_2
\;\preceq\;
\frac{1}{|\cT_1^{H,\ell}|}
(\Psi_{H,\ell})^\top \Psi_{H,\ell}
\;\preceq\;
\lambda_{H,\max} I_2,
\end{equation*}
where $\Psi_{H,\ell}$ is the $n_{H,\ell}\times 2$ design matrix with rows
$\phi_{H,\ell}(X_i)^\top$, $i\in\cT_1^{H,\ell}$, and
\begin{equation*}
\phi_{H,\ell}(x)
=
\bigl(
1,\;
\source(x)-\source(x_{H,\ell})
\bigr)^\top,
\end{equation*}
with $\cT_1^{H,\ell}:=\{i\in \cT_1 : X_i\in\cA_{H,\ell}\}$.

\smallskip
\noindent
\textbf{Weighted Gram condition on the stability event.}
On the event $\cE_{\mathrm{ess}}$ (defined in \eqref{Eq:event_ESS}), the corresponding
\emph{weighted} Gram matrices
\begin{equation*}
G_{H,\ell}
:=
\frac{1}{n_{H,\ell}}(\Psi_{H,\ell})^\top W_{H,\ell}\Psi_{H,\ell},
\qquad
W_{H,\ell}:=\mathrm{diag}(w_{i,\ell})_{i\in\cT_1^{H,\ell}},
\end{equation*}
satisfy, for all $H\in\cH$ and all $\ell\in[L_H]$,
\begin{equation*}
\lambda_{H,\min} I_2
\;\preceq\;
G_{H,\ell}
\;\preceq\;
\lambda_{H,\max} I_2.
\end{equation*}
\end{subassumption}

\begin{subassumption}\label{Ass:star_local-design}
\textbf{On the target tessellation.}
There exist constants
$0<\lambda_{\min}^\star\le \lambda_{\max}^\star<\infty$
such that for all $\ell\in[L^\star]$,
\begin{equation}
\lambda_{\min}^\star I_2
\preceq
\frac{1}{|\cT_1^{\star,\ell}|}
(\Psi_\ell^\star)^\top \Psi_\ell^\star
\preceq
\lambda_{\max}^\star I_2,
\end{equation}
where $\Psi_\ell^\star$ is the $n_\ell\times 2$ design matrix with rows
$\phi_\ell^\star(X_i)^\top$ ($i\in\cT_1^{\star,\ell}$), and
\begin{equation}
\phi_\ell^\star(x)
=
\bigl(
1,
\source(x)-\source(x_{\ell})
\bigr)^\top,
\end{equation}
with $\cT_1^{\star,\ell}:=\{i\in \cT_1 : X_i\in\cA_\ell^\star\}$.
\end{subassumption}
\end{ass}

\begin{ass}[Uniform boundedness of local features and residuals]
\label{Ass:bounded_features}
There exist constants $\phi_{\max}, r_{\max} < \infty$ such that, for all
$H \in \cH$, all $\ell \in [L_H]$ and all $x \in \cA_{H,\ell}$,
\begin{equation*}
\|\phi_{H,\ell}(x)\|_2 \le \phi_{\max},
\qquad
|r_{H,\ell}(x)| \le r_{\max}.
\end{equation*}
\end{ass}

\begin{ass}[Cellwise lower-mass condition]\label{Ass:mass_cell}
For all $\delta\in(0,1)$ and all $H\in\cH$,
\begin{equation}
n_{\cT_1}
\min_{\ell\in[L_H]} p_{H,\ell}
\;\ge\;
8\log\Big(\frac{L_H}{\delta}\Big).
\end{equation}
\end{ass}

\begin{ass}[Quasi-uniform mesh condition]
\begin{subassumption}\label{Ass:quasi_uniform} \textbf{On admissible tessellations.}
There exists a constant $\frc_\Delta>0$ such that for all $H\in\cH$,
\begin{equation*}
\Delta_{\max}(H):=\max_{\ell\in[L_H]}\Delta_{H,\ell}
\le
\frc_\Delta\,L_H^{-1/d},
\end{equation*}
where $\Delta_{H,\ell}:=\mathrm{diam}(\cA_{H,\ell})$.
\end{subassumption}
\begin{subassumption}\label{Ass:quasi_uniform_target} \textbf{On the target tessellation.} There exist constants $0<\frc_{\min}^\star\le \frc_{\max}^\star<\infty$ such that
\begin{equation}
\frc_{\min}^\star(L^\star)^{-1/d}
\le
\Delta_{\min}(H^\star)
\le
\Delta_{\max}(H^\star)
\le
\frc_{\max}^\star(L^\star)^{-1/d},
\end{equation}
where $\Delta_{\min}(H^\star){=}\min_{\ell\in[L^\star]}\mathrm{diam}(\cA_\ell^\star)$ and $\Delta_{\max}(H^\star){=}\max_{\ell\in[L^\star]}\mathrm{diam}(\cA_\ell^\star)$.
\end{subassumption}
\end{ass}

\begin{ass}[Effective sample size and plug-in regime]\label{Ass:ESS_plugin}
For any fixed admissible tessellation $H$ and cell
$\cA_{H,\ell}$, define the effective sample size count
\begin{equation*}
N_{H,\ell}=\#\Big\{i\in\cT_1 :
\|X_i-x_{H,\ell}\|\le h, 
|\wsource(X_i)-\wsource(x_{H,\ell})|\le \overline h
\Big\}.
\end{equation*}
\begin{subassumption}\label{subass:ESS}
Fix $\delta\in(0,1)$.
There exist constants
$\frc_1^{\mathrm{ess}},\frc_2^{\mathrm{ess}}>0$
such that, with probability at least $1-\delta$, the following event holds:
\begin{equation}
\frac{1}{\frc_2^{\mathrm{ess}}}
n_{\cT_1} h^d \overline h
\le
\sup_{H\in\cH,\cA_{H,\ell}} N_{H,\ell}
\le
\frc_2^{\mathrm{ess}}
n_{\cT_1} h^d \overline h .
\end{equation}
We define the event
\begin{equation}\label{Eq:event_ESS}
\cE_{\mathrm{ess}}=\Big\{\forall H\in\cH,\forall \ell\in[L_H]:\frac{1}{\frc_2^{\mathrm{ess}}}
n_{\cT_1} h^d \overline h
\le
N_{H,\ell}
\le
\frc_2^{\mathrm{ess}}
n_{\cT_1} h^d \overline h\Big\}.
\end{equation}
In particular, the effective sample size of the local weighted estimator
within each cell is of order $n_{\cT_1} h^d \overline h$, uniformly over
$(H,\ell)$.
\end{subassumption}
\begin{subassumption}\label{subass:plug-in}
We further assume the \textit{plug-in regime}
\begin{equation}\label{eq:plug-in_condition}
n_{\cT_1} h^d \overline h
\ge
\frc_1^{\mathrm{ess}}
\log\Big(\frac{1}{\delta}\Big).
\end{equation}
\end{subassumption}
\end{ass}

\begin{ass}\label{Ass:regularity_transfer}
For all $\ell\in[L^\star]$, the transfer function
$g_\ell^\star$ belongs to the H\"older class
$\mathrm{H\"ol}(\beta_g,\frl_g;\R)$ for some $\beta_g>1$ and
$\frl_g>0$.
\end{ass}

\begin{ass}[Balanced cell allocation]
\label{Ass:lower_mass-balance}
For any $\ell\in[L^\star]$, let $p_\ell^\star := \mathbb{P}(X\in \cA_\ell^\star)$.
There exist constants $0<\mathfrak a\le \mathfrak b<\infty$ such that
\begin{equation}
\frac{\mathfrak a}{L^\star} \le p_\ell^\star \le \frac{\mathfrak b}{L^\star}.
\end{equation}
\end{ass}

\noindent
Under Assumptions~\ref{Ass:design_source}, \ref{Ass:regularity_source},
\ref{Ass:noise}, and~\ref{Ass:kernels}, we recall that  for any $\delta_{\cS}\in(0,1)$, there exist constants
$\frc_{\cS}',\frc_{\cS}>0$ and an event $\cE_{\cS}$ with
$\P(\cE_{\cS})\ge 1-\delta_{\cS}$ on which
\begin{equation}\label{Eq:epsS_rate}
\epsilon_{\cS}
:=
\|\wsource-\source\|_\infty
\le
\frc_{\cS}'
\Bigg(
\frac{\log(\frc_{\cS}/\delta_{\cS})}{n_{\cS}}
\Bigg)^{\frac{\beta_{\cS}}{2\beta_{\cS}+d}}.
\end{equation}

\noindent
\subsubsection*{\textbf{well-specified compositional model}}
\label{subsec:well-specified}

The transfer learning framework is said to be \emph{well-specified in the compositional sense} if there exists a tessellation 
$H^\star=\{\cA_\ell^\star:\ell\in[L^\star]\}$ such that 
Assumption~\ref{Ass:transfer_function} holds on each cell. 
In this regime, the source-target relationship is exactly captured by a cellwise transfer function, and no systematic model error is present. The model is \emph{misspecified} if no admissible tessellation satisfies 
Assumption~\ref{Ass:transfer_function} exactly, resulting in an irreducible approximation bias. 
Our analysis makes this bias explicit through oracle inequalities that balance estimation and approximation errors.

\subsubsection*{\textbf{well-specified local transfer model}} \label{subsec:well-specified_local}
We say that the framework is \emph{well-specified in the local transfer sense} if there exists a tessellation 
$H^\star=\{\cA_\ell^\star:\ell\in[L^\star]\}$ such that  Assumption~\ref{Ass:local_transfer} holds on each cell. 
It is misspecified if no such partition exists.

\section{Useful preliminary results}\label{app:preliminaries}
\noindent
In this section, we fix a cell $\cA_{H,\ell}$ of an admissible tessellation
$H\in\cH$ (see Definition~\ref{Def:admissible-partitions}), with anchor
point $x_{H,\ell}$. We work on the target subsample $\cT_1$ and consider
the oracle-score regression model
\begin{equation}\label{Eq:regression_fs}
Y_i = g_\ell(\source(X_i)) + \varepsilon_i,
\qquad i \in \cT_1^{H,\ell},
\end{equation}
where $\E[\varepsilon_i| X_i]=0$ and the noise is conditionally
sub-exponential with proxy variance $\sigma_{\cT}^2$.
In practice, the unknown score $\source$ is replaced by its estimator
$\wsource$, yielding a plug-in regression procedure whose additional
approximation error will be quantified below.

\noindent
For a fixed cell $\ell$, we write $m_{H,\ell}:=g_\ell\circ\source$ for the
corresponding regression function on $\cA_{H,\ell}$.
We assume that $m_{H,\ell}$ admits a local linear expansion at the representative
point $x_{H,\ell}$:
there exist $\theta^\star_{H,\ell}=(a^\star_{H,\ell},b^\star_{H,\ell})\in\R^2$ and a remainder
$r_{H,\ell}(\,\cdot\,):=r_{H,\ell}(\,\cdot\,;x_{H,\ell})$ such that, for all
$u\in\cA_{H,\ell}$,
\begin{equation}\label{eq:local_lin_decomp}
m_{H,\ell}(u)
=
(\theta^\star_{H,\ell})^\top\phi_{H,\ell}(u)
+
r_{H,\ell}(u),
\end{equation}
where the feature map is defined by
\begin{equation}\label{eq:feature_map_fixed_center}
\phi_{H,\ell}:u\in\cA_{H,\ell}\mapsto
\bigl(1,\;\source(u)-\source(x_{H,\ell})\bigr)^\top\in\R^2.
\end{equation}
\noindent
We define the cell index set and its cardinality by
\begin{equation}\label{eq:T1_cell_def}
\cT_1^{H,\ell}:=\{i\in \cT_1:\ X_i\in \cA_{H,\ell}\}
\quad\text{and}\quad
n_{H,\ell}:=|\cT_1^{H,\ell}|.
\end{equation}

\noindent
We consider the following weighted least-squares estimators on the cell
$\cA_{H,\ell}$:
\begin{multline}\label{Eq:LS_regression_fs}
\overline\theta_{H,\ell}
=
(\overline b_{H,\ell},\overline a_{H,\ell})
\in
\argmin_{b,a\in\R}
\Bigg\{
\frac{1}{n_{H,\ell}}\sum_{i\in\cT_1^{H,\ell}}
\Big[Y_i- a\big(\source(X_i)-\source(x_{H,\ell})\big)-b\Big]^2 \\
\times K_{x,h}(\|X_i-x_{H,\ell}\|)\,
K_{z,\overline h}\big(|\source(X_i)-\source(x_{H,\ell})|\big)
\Bigg\},
\end{multline}
and
\begin{multline}\label{Eq:LS_regression_wfs}
\widehat\theta_{H,\ell}
=
(\widehat b_{H,\ell},\widehat a_{H,\ell})
\in
\argmin_{b,a\in\R}
\Bigg\{
\frac{1}{n_{H,\ell}}\sum_{i\in\cT_1^{H,\ell}}
\Big[Y_i- a\big(\wsource(X_i)-\wsource(x_{H,\ell})\big)-b\Big]^2\\
\times K_{x,h}(\|X_i-x_{H,\ell}\|)\,
K_{z,\overline h}\big(|\wsource(X_i)-\wsource(x_{H,\ell})|\big)
\Bigg\}.
\end{multline}

\noindent
For the oracle estimator, introduce the weights: for all $ i\in\cT_1^{H,\ell}$, $\ell\in[L_H]$,
\begin{equation}\label{eq:oracle_weights}
w_{i,\ell}
:=
K_{x,h}(\|X_i-x_{H,\ell}\|)\,
K_{z,\overline h}\big(|\source(X_i)-\source(x_{H,\ell})|\big)\mathbf 1_{\{X_i\in \cA_{H,\ell}\}},
\end{equation}
and the (unnormalized) weighted design matrices and vectors
\begin{equation*}\label{eq:weighted_objects}
\Psi_{H,\ell}:=
\begin{pmatrix}
\phi_{H,\ell}(X_i)^\top
\end{pmatrix}_{i\in\cT_1^{H,\ell}}
\quad\text{and}\quad
W_{H,\ell}:=\mathrm{diag}(w_{i,\ell})_{i\in\cT_1^{H,\ell}}.
\end{equation*}
\noindent
Define the weighted Gram matrix and score vector
\begin{equation*}\label{eq:weighted_gram_score}
G_{H,\ell}:=\frac{1}{n_{H,\ell}}\Psi_{H,\ell}^\top W_{H,\ell}\Psi_{H,\ell},
\quad \text{and}\quad
s_{H,\ell}:=\frac{1}{n_{H,\ell}}\Psi_{H,\ell}^\top W_{H,\ell}Y,
\end{equation*}
so that, whenever $G_{H,\ell}$ is invertible,
\begin{equation}\label{eq:theta_closed_form}
\overline\theta_{H,\ell}=G_{H,\ell}^{-1}s_{H,\ell}.
\end{equation}

\subsection{A weighted least-squares tail bound}

\noindent
We state a cellwise concentration bound for the oracle weighted estimator
$\overline\theta_{H,\ell}$, uniformly over $\ell\in[L_H]$.

\begin{lem}\label{lem:cellwise_ols_tail_union_ess}
Suppose Assumption~\ref{Ass:noise} holds and fix an admissible tessellation
$H\in\cH$ together with bandwidths $(h,\overline h)$.
Assume that on the event $\cE_{\mathrm{ess}}$ (defined in
\eqref{Eq:event_ESS}) the weighted Gram matrices satisfy, for all
$\ell\in[L_H]$,
\begin{equation}\label{eq:weighted_design_condition_ess_corr}
\lambda_{H,\min} I_2
\preceq
G_{H,\ell}
\preceq
\lambda_{H,\max} I_2.
\end{equation}
Then there exists a universal constant $\frc>0$ such that for all
$\delta\in(0,1)$,
\begin{equation}\label{eq:cellwise_wls_tail_union_ess_corr}
\P\Bigg(
\exists \ell\in[L_H]:\
\|\overline\theta_{H,\ell}-\theta^\star_{H,\ell}\|_2^2
>
\frc\Big(
\frac{\sigma_{\cT}^2}{\lambda_{H,\min}^2}
\frac{\log(2L_H/\delta)}{n_{\mathrm{eff}}^{H,\ell}}
+
\frac{B_{H,\ell}^2}{\lambda_{H,\min}^2}
\Big)
\;\cap\;\cE_{\mathrm{ess}}
\Bigg)
\le \delta,
\end{equation}
where
\begin{equation}\label{eq:bias_term_def_ess_corr}
B_{H,\ell}
:=
\Big\|
\frac{1}{n_{H,\ell}}\Psi_{H,\ell}^\top W_{H,\ell} r_{H,\ell}
\Big\|_2,
\qquad
r_{H,\ell}:=\big(r_{H,\ell}(X_i)\big)_{i\in\cT_1^{H,\ell}},
\end{equation}
and
\begin{equation*}
n_{\mathrm{eff}}^{H,\ell}
:=
\frac{\big(\sum_{i\in\cT_1^{H,\ell}} w_{i,\ell}\big)^2}
{\sum_{i\in\cT_1^{H,\ell}} w_{i,\ell}^2}
\end{equation*}
denotes the kernel-weighted effective sample size in cell $\ell$.

Moreover, if there exist constants $\phi_{\max},r_{\max}<\infty$
such that for all $H\in\cH$, all $\ell\in[L_H]$ and all
$x\in\cA_{H,\ell}$,
\begin{equation*}
\|\phi_{H,\ell}(x)\|_2\le \phi_{\max},
\qquad
|r_{H,\ell}(x)|\le r_{\max},
\end{equation*}
then for all $\ell\in[L_H]$,
\begin{equation}\label{eq:bias_term_simple_ess_corr}
B_{H,\ell}\le \phi_{\max}\, r_{\max}.
\end{equation}
\end{lem}

\begin{proof}
Fix $\ell\in[L_H]$ and condition on the design
$\{X_i\}_{i\in\cT_1^{H,\ell}}$.
Using the local linear decomposition \eqref{eq:local_lin_decomp}, for all
$i\in\cT_1^{H,\ell}$ we may write
\begin{equation*}
Y_i
=
(\theta^\star_{H,\ell})^\top\phi_{H,\ell}(X_i)
+
r_{H,\ell}(X_i)
+
\varepsilon_i.
\end{equation*}
Recalling the definitions of $G_{H,\ell}$ and $s_{H,\ell}$, this yields
\begin{equation*}
s_{H,\ell}-G_{H,\ell}\theta^\star_{H,\ell}
=
\frac{1}{n_{H,\ell}}\Psi_{H,\ell}^\top W_{H,\ell}\varepsilon
+
\frac{1}{n_{H,\ell}}\Psi_{H,\ell}^\top W_{H,\ell} r_{H,\ell}.
\end{equation*}
On the event that $G_{H,\ell}$ is invertible,
\begin{equation*}
\overline\theta_{H,\ell}-\theta^\star_{H,\ell}
=
G_{H,\ell}^{-1}
\Bigg(
\frac{1}{n_{H,\ell}}\Psi_{H,\ell}^\top W_{H,\ell}\varepsilon
+
\frac{1}{n_{H,\ell}}\Psi_{H,\ell}^\top W_{H,\ell} r_{H,\ell}
\Bigg).
\end{equation*}
Using \eqref{eq:weighted_design_condition_ess_corr},
\begin{equation}\label{eq:wls_decomp_rewrite_ess_corr}
\|\overline\theta_{H,\ell}-\theta^\star_{H,\ell}\|_2
\le
\frac{1}{\lambda_{H,\min}}
\Bigg(
\Big\|
\frac{1}{n_{H,\ell}}\Psi_{H,\ell}^\top W_{H,\ell}\varepsilon
\Big\|_2
+
B_{H,\ell}
\Bigg).
\end{equation}

\medskip
\noindent
Let $v\in\mathbb S^1$ be fixed. Then
\begin{equation*}
v^\top\Psi_{H,\ell}^\top W_{H,\ell}\varepsilon
=
\sum_{i\in\cT_1^{H,\ell}}
w_{i,\ell}\,
(v^\top\phi_{H,\ell}(X_i))\,\varepsilon_i.
\end{equation*}
By Assumption~\ref{Ass:noise}, the variables
$(\varepsilon_i)_{i\in\cT_1^{H,\ell}}$
are independent, centered, and sub-exponential with
$\|\varepsilon_i\|_{\psi_1}\le \sigma_{\cT}$.
Moreover,
\begin{equation*}
|v^\top\phi_{H,\ell}(X_i)|
\le \|\phi_{H,\ell}(X_i)\|_2
\le \phi_{\max}.
\end{equation*}
Hence the summands
$w_{i,\ell}(v^\top\phi_{H,\ell}(X_i))\varepsilon_i$
are independent, centered, sub-exponential random variables with
$\psi_1$-norm bounded by $\phi_{\max}|w_{i,\ell}|\sigma_{\cT}$.
\\
Applying Bernstein’s inequality (see Appendix \ref{App:concentration}) for weighted sums of sub-exponential
random variables yields that there exists a universal constant $\frc_1>0$
such that for all $t>0$,
\begin{equation*}
\P\Bigg(
\big|
v^\top\Psi_{H,\ell}^\top W_{H,\ell}\varepsilon
\big|
\ge
\frc_1\sigma_{\cT}\phi_{\max}\Bigg(
\sqrt{t\sum_{i} w_{i,\ell}^2 }
+
t\max_{i}|w_{i,\ell}|
\Bigg)
\Bigg)
\le 2e^{-t}.
\end{equation*}

On $\cE_{\mathrm{ess}}$, we have
\begin{equation*}
\sum_{i} w_{i,\ell}^2
=
\frac{(\sum_i w_{i,\ell})^2}{n_{\mathrm{eff}}^{H,\ell}}.
\end{equation*}
Using this relation and a standard $1/2$-net argument on $\mathbb S^1$
(see \cite[Section~5.2.2]{wainwright2019high}),
we deduce that there exists a universal constant $\frc_2>0$ such that
for all $t>0$,
\begin{equation*}
\P\Bigg(
\Big\|
\frac{1}{n_{H,\ell}}\Psi_{H,\ell}^\top W_{H,\ell}\varepsilon
\Big\|_2^2
\ge
\frc_2\,\frac{\sigma_{\cT}^2 t}{n_{\mathrm{eff}}^{H,\ell}}
\Bigg)
\le 2e^{-t}.
\end{equation*}
\noindent
Combining this bound with
\eqref{eq:wls_decomp_rewrite_ess_corr} gives that with probability at least
$1-2e^{-t}$,
\begin{equation*}
\|\overline\theta_{H,\ell}-\theta^\star_{H,\ell}\|_2^2
\le
\frc\Bigg(
\frac{\sigma_{\cT}^2}{\lambda_{H,\min}^2}
\frac{t}{n_{\mathrm{eff}}^{H,\ell}}
+
\frac{B_{H,\ell}^2}{\lambda_{H,\min}^2}
\Bigg).
\end{equation*}
Taking $t=\log(2L_H/\delta)$ and applying a union bound over
$\ell\in[L_H]$ yields \eqref{eq:cellwise_wls_tail_union_ess_corr}.
\\
Finally, if $|r_{H,\ell}(u)|\le r_{\max}$ and
$\|\phi_{H,\ell}(u)\|_2\le \phi_{\max}$ for all
$u\in\cA_{H,\ell}$, then
\begin{equation*}
B_{H,\ell}
=
\Big\|
\frac{1}{n_{H,\ell}}
\sum_{i\in\cT_1^{H,\ell}}
w_{i,\ell}\phi_{H,\ell}(X_i)r_{H,\ell}(X_i)
\Big\|_2
\le
\phi_{\max} r_{\max}.
\end{equation*}
This completes the proof.
\end{proof}

\subsection{A perturbation bound: oracle vs. plug-in}

The second objective of this section is to compare the estimators
$\overline\theta_{H,\ell}=(\overline b_{H,\ell},\overline a_{H,\ell})$ and
$\widehat\theta_{H,\ell}=(\widehat b_{H,\ell},\widehat a_{H,\ell})$ obtained by
replacing the oracle score
\begin{equation*}
Z_{i,\ell}:=\source(X_i)-\source(x_{H,\ell})
\end{equation*}
by the plug-in score
\begin{equation*}
\widehat Z_{i,\ell}:=\wsource(X_i)-\wsource(x_{H,\ell}).
\end{equation*}

\begin{lem}[Oracle vs.\ plug-in local estimator]\label{lem:oracle-plugin-gap}
Suppose Assumptions~\ref{Ass:kernels},~\ref{Ass:H_local-design}, and
\ref{Ass:ESS_plugin} hold.
Fix $H\in\cH$ and a cell index $\ell\in[L_H]$.
Assume moreover that
\begin{equation*}
\epsilon_{\cS}:=\|\wsource-\source\|_\infty
\end{equation*}
satisfies
\begin{equation}\label{eq:eps_over_hbar_small}
\frac{\epsilon_{\cS}}{\overline h}\le \frc_0,
\end{equation}
for a sufficiently small constant $\frc_0>0$ depending only on kernel envelopes
and Lipschitz constants.
Then there exists a constant $\frc>0$ such that for all $\delta\in(0,1)$,
on the event $\cE_{\mathrm{ess}}$ from Assumption~\ref{Ass:ESS_plugin},
with conditional probability at least $1-\delta$ (given the source sample),
\begin{equation}\label{eq:oracle_plugin_gap}
\|\widehat\theta_{H,\ell}-\overline\theta_{H,\ell}\|_2
\le
\frc\Bigg[
\frac{\epsilon_{\cS}}{\overline h}
+
\sqrt{\frac{\log(1/\delta)}{n_{\cT_1} h^d \overline h}}
+
\frac{\log(1/\delta)}{n_{\cT_1} h^d \overline h}
\Bigg].
\end{equation}
In particular, unconditionally, the above holds with probability at least
$1-2\delta$.
The constant $\frc>0$ depends only on kernel envelopes and Lipschitz constants,
$\lambda_{H,\min},\lambda_{H,\max}$, and admissible-partition constants.
\end{lem}

\begin{proof}
Fix a cell $\ell$ and write $x_\ell:=x_{H,\ell}$ for brevity.
For all $i\in \cT_1^{H,\ell}$, define
\begin{equation*}
Z_{i,\ell}:=\source(X_i)-\source(x_\ell),
\qquad
\widehat Z_{i,\ell}:=\wsource(X_i)-\wsource(x_\ell),
\end{equation*}
the feature vectors
\begin{equation*}
\phi_i = (1,Z_{i,\ell}/\overline h)^\top,
\qquad
\widehat\phi_i = (1,\widehat Z_{i,\ell}/\overline h)^\top,
\end{equation*}
and the weights
\begin{equation*}
w_i
=
K_{x,h}(\|X_i-x_\ell\|)\,
K_{z,\overline h}(|Z_{i,\ell}|),
\qquad
\widehat w_i
=
K_{x,h}(\|X_i-x_\ell\|)\,
K_{z,\overline h}(|\widehat Z_{i,\ell}|).
\end{equation*}
Let
\begin{equation*}
S=\sum_{i\in\cT_1^{H,\ell}} w_i,
\qquad
\widehat S=\sum_{i\in\cT_1^{H,\ell}} \widehat w_i,
\end{equation*}
and define
\begin{equation*}
M = S^{-1}\sum_{i\in \cT_1^{H,\ell}} w_i\phi_i\phi_i^\top,
\qquad
\widehat M = \widehat S^{-1}\sum_{i\in \cT_1^{H,\ell}} \widehat w_i\widehat\phi_i\widehat\phi_i^\top,
\end{equation*}
as well as
\begin{equation*}
m = S^{-1}\sum_{i\in \cT_1^{H,\ell}} w_i\phi_i Y_i,
\qquad
\widehat m = \widehat S^{-1}\sum_{i\in \cT_1^{H,\ell}} \widehat w_i\widehat\phi_i Y_i.
\end{equation*}
Then $\overline\theta_{H,\ell}=M^{-1}m$ and
$\widehat\theta_{H,\ell}=\widehat M^{-1}\widehat m$.

\smallskip
\noindent
\underline{\textit{Step 1: Effective sample size and normalization.}}
On the event $\cE_{\mathrm{ess}}$,
\begin{equation*}
n_{\mathrm{eff}}^{H,\ell}
=
\frac{S^2}{\sum_{i} w_i^2}
\asymp
n_{\cT_1}h^d\overline h.
\end{equation*}
Since $K_z$ is Lipschitz and
\begin{equation*}
|Z_{i,\ell}-\widehat Z_{i,\ell}|
\le 2\epsilon_{\cS},
\end{equation*}
we have
\begin{equation*}
|w_i-\widehat w_i|
\lesssim
\frac{\epsilon_{\cS}}{\overline h}.
\end{equation*}
Summing gives
\begin{equation*}
|S-\widehat S|
\lesssim
\frac{\epsilon_{\cS}}{\overline h}\, n_{H,\ell}.
\end{equation*}
Since $S\asymp n_{\cT_1}h^d\overline h$, for $\frc_0$ small enough,
$\widehat S/S$ is bounded away from $0$ and $\infty$.

\smallskip
\noindent
\underline{\textit{Step 2: Control of $\|M-\widehat M\|_{\op}$.}}
Let
\begin{equation*}
A=\sum w_i\phi_i\phi_i^\top,
\qquad
\widehat A=\sum \widehat w_i\widehat\phi_i\widehat\phi_i^\top.
\end{equation*}
Then
\begin{equation*}
M-\widehat M
=
\frac{A-\widehat A}{S}
+
\widehat A\Big(\frac{1}{S}-\frac{1}{\widehat S}\Big).
\end{equation*}
On the kernel supports,
$\|\phi_i\|_2,\|\widehat\phi_i\|_2\lesssim 1$, and
\begin{equation*}
\|\phi_i\phi_i^\top-\widehat\phi_i\widehat\phi_i^\top\|_{\op}
\lesssim
\frac{\epsilon_{\cS}}{\overline h}.
\end{equation*}
Hence
\begin{equation*}
\|A-\widehat A\|_{\op}
\lesssim
\frac{\epsilon_{\cS}}{\overline h}\, S.
\end{equation*}
Using stability of $S$ and $\widehat S$,
\begin{equation*}
\|M-\widehat M\|_{\op}
\lesssim
\frac{\epsilon_{\cS}}{\overline h}.
\end{equation*}

\smallskip
\noindent
\underline{\textit{Step 3: Concentration terms.}}
Conditionally on $(X_i)$, by Bernstein's inequality for weighted
sub-exponential sums,
with probability at least $1-\delta/4$,
\begin{equation*}
\|m-\E[m|(X_i)]\|_2
+
\|\widehat m-\E[\widehat m|(X_i)]\|_2
\lesssim
\sqrt{\frac{\log(1/\delta)}{n_{\mathrm{eff}}^{H,\ell}}}
+
\frac{\log(1/\delta)}{n_{\mathrm{eff}}^{H,\ell}}.
\end{equation*}
On $\cE_{\mathrm{ess}}$, this equals
\begin{equation*}
\sqrt{\frac{\log(1/\delta)}{n_{\cT_1}h^d\overline h}}
+
\frac{\log(1/\delta)}{n_{\cT_1}h^d\overline h}.
\end{equation*}
Moreover,
\begin{equation*}
\|\E[m|(X_i)]-\E[\widehat m|(X_i)]\|_2
\lesssim
\frac{\epsilon_{\cS}}{\overline h}.
\end{equation*}

\smallskip
\noindent
\underline{\textit{Final step.}}
We write
\begin{equation*}
\widehat\theta_{H,\ell}-\overline\theta_{H,\ell}
=
(\widehat M^{-1}-M^{-1})m
+
\widehat M^{-1}(\widehat m-m).
\end{equation*}
By Assumption~\ref{Ass:H_local-design} and Weyl's inequality,
\begin{equation*}
\|M^{-1}\|_{\op},\|\widehat M^{-1}\|_{\op}
\lesssim
\frac{1}{\lambda_{H,\min}}.
\end{equation*}
Moreover,
\begin{equation*}
\|\widehat M^{-1}-M^{-1}\|_{\op}
\le
\|M^{-1}\|_{\op}\,\|M-\widehat M\|_{\op}\,\|\widehat M^{-1}\|_{\op}
\lesssim
\frac{\epsilon_{\cS}}{\overline h}.
\end{equation*}
Combining the previous bounds and applying a union bound yields
\eqref{eq:oracle_plugin_gap}.
\end{proof}

\color{black}

\subsection{A high-probability control of the oracle slopes}

\noindent
The following lemma provides a high-probability control of the oracle slopes.

\begin{lem}\label{lem:hp_oahl}
Assume Assumptions~\ref{Ass:design_target}, \ref{Ass:noise},
\ref{Ass:star_local-design}, and~\ref{Ass:ESS_plugin} hold.
Then there exists a constant $\frc_a>0$ such that for all $\delta\in(0,1)$, the event
\begin{equation}\label{eq:def_Adelta}
\cE_{a,H,\delta}
:=\left\{
\max_{\ell\in[L_H]}|\overline a_{H,\ell}|^2
\le \frc_a \log\Big(\frac{L_H}{\delta}\Big)
\right\}
\end{equation}
satisfies $\P(\cE_{a,H,\delta})\ge 1-\delta$.
\end{lem}

\begin{proof}
Fix $\delta\in(0,1)$.
Recall that for a fixed tessellation $H\in\cH$ and a cell $\cA_{H,\ell}$, the
cellwise transfer coefficients
$\overline\theta_{H,\ell}=(\overline b_{H,\ell},\overline a_{H,\ell})^\top$
are defined as the solution of a weighted least-squares problem based on the target
subsample $\cT_1$.
They admit the closed-form representation
\begin{equation*}
\overline\theta_{H,\ell}
=
M_{H,\ell}^{-1} m_{H,\ell},
\end{equation*}
where the Gram matrix $M_{H,\ell}\in\mathbb R^{2\times 2}$ and the vector
$m_{H,\ell}\in\mathbb R^{2}$ are given by
\begin{equation*}
M_{H,\ell}
:=
\sum_{i\in\cT_{1}^{H,\ell}}
\alpha_{i,\ell}\,
\phi_{H,\ell}(X_i)\phi_{H,\ell}(X_i)^\top
\quad\text{and}\quad
m_{H,\ell}
:=
\sum_{i\in\cT_{1}^{H,\ell}}
\alpha_{i,\ell}\,
\phi_{H,\ell}(X_i)\,Y_i,
\end{equation*}
with $\phi_{H,\ell}(x)=(1,z_{H,\ell}(x))^\top$,
$z_{H,\ell}(x):=\source(x)-\source(x_{H,\ell})$, and normalized weights
\begin{equation*}
\alpha_{i,\ell}
:=
\frac{w_{i,\ell}}{\sum_{j\in\cT_{1}^{H,\ell}} w_{j,\ell}},
\qquad
w_{i,\ell}
=
K_{x,h}\big(\|X_i-x_{H,\ell}\|\big)
K_{z,\overline h}\big(|\source(X_i)-\source(x_{H,\ell})|\big).
\end{equation*}
By construction, $\alpha_{i,\ell}\ge 0$ and $\sum_{i\in\cT_{1}^{H,\ell}}\alpha_{i,\ell}=1$.
\\
Define the conditioning event
\begin{equation*}
\cE_{\lambda,\ell}:=\{\lambda_{\min}(M_{H,\ell})\ge \lambda_0\}.
\end{equation*}
By Assumptions~\ref{Ass:star_local-design} and \ref{Ass:ESS_plugin}, we may choose
$\lambda_0$ (and the implicit constants in the effective sample size regime) so that
\begin{equation}\label{eq:gram_good_prob}
\P\Big(\bigcap_{\ell\in[L_H]}\cE_{\lambda,\ell}\Big)\ge 1-\frac{\delta}{2}.
\end{equation}
On $\cE_{\lambda,\ell}$, we have $\|M_{H,\ell}^{-1}\|_{\op}\le \lambda_0^{-1}$ and hence
\begin{equation}\label{eq:oahl_basic_bound}
|\overline a_{H,\ell}|
\le \|\overline\theta_{H,\ell}\|_2
\le \|M_{H,\ell}^{-1}\|_{\op}\,\|m_{H,\ell}\|_2
\le \lambda_0^{-1}\|m_{H,\ell}\|_2.
\end{equation}
Moreover, using $\|\phi_{H,\ell}(X_i)\|_2\le \frc_\phi$ on each cell, we obtain
\begin{equation}\label{eq:m_bound}
\|m_{H,\ell}\|_2
\le \sum_{i\in \cT_{1}^{H,\ell}} \alpha_{i,\ell}\|\phi_{H,\ell}(X_i)\|_2|Y_i|
\le \frc_\phi\sum_{i\in \cT_{1}^{H,\ell}} \alpha_{i,\ell}|Y_i|.
\end{equation}
\noindent
Consider the event
\begin{equation*}
\cA^{\mathrm{ess}}
:=
\Bigg\{
\max_{\ell\in[L_H]}
\sum_{i\in\cT_{1}^{H,\ell}}\alpha_{i,\ell}^2
\le
\frac{\frc_{\mathrm{ess}}^2}{n_{\cT_1}h^d\overline h}
\Bigg\}.
\end{equation*}
By Assumption~\ref{Ass:ESS_plugin} and boundedness of the kernels, this event
satisfies $\P(\cA^{\mathrm{ess}})\ge 1-\delta/2$.
Conditionally on the weights $(\alpha_{i,\ell})_i$ and the design $(X_i)_i$,
the variables $(Y_i)_{i}$ are independent and, by
Assumption~\ref{Ass:noise}, sub-exponential with uniformly bounded $\psi_1$-norm.
Bernstein's inequality (see Lemma \ref{lem:Bernstein_subexp}) for weighted sums of sub-exponential variables yields that
there exists a constant $\frc>0$ such that, for every $\ell\in[L_H]$,
\begin{equation}\label{eq:bernstein_weighted}
\P\left(
\sum_{i\in\cT_1^{H,\ell}} \alpha_{i,\ell}|Y_i|
\ge \frc\Big(1+\sqrt{\log(L_H/\delta)}\Big)
\, \big|\, (\alpha_{i,\ell})_i
\right)
\le \frac{\delta}{2L_H},
\end{equation}
on the event $\cA^{\mathrm{ess}}$.
Taking a union bound over $\ell\in[L_H]$ yields that, on $\cA^{\mathrm{ess}}$, with
probability at least $1-\delta/2$,
\begin{equation}\label{eq:sum_uniform}
\max_{\ell\in[L_H]}\sum_{i\in\cT_1^{H,\ell}} \alpha_{i,\ell}|Y_i|
\le \frc\Big(1+\sqrt{\log(L_H/\delta)}\Big).
\end{equation}
Intersecting \eqref{eq:sum_uniform} with \eqref{eq:gram_good_prob}, and using
\eqref{eq:oahl_basic_bound}-\eqref{eq:m_bound}, yields \eqref{eq:def_Adelta}.
\end{proof}

\section[Proofs of Section 3.2.: Upper bounds]{Proofs of Section 3.2.: Upper bounds}\label{app:proof_upper_bound}
\noindent
Appendix~B.1 is devoted to Theorem~1, which establishes the oracle rate under the well-specified compositional model. The result is stated in expectation and provides a sharp risk bound for the transfer estimator when the oracle tessellation is known and the compositional structure (Assumption 1) holds.
Appendix~B.2 proves Theorem~2 in the local linear transfer model (Assumption 2). This theorem provides a high-probability bound on the excess risk for a data-dependent nonparametric estimator constructed on a fixed tessellation $H\in\cH$. It also quantifies the additional error arising from the data-driven selection of the tessellation. As a consequence, it yields Corollary~1 (plug-in version), Corollary~2 (expectation bound), and Corollary~3 (oracle rate in the well-specified case).
Finally, Appendix~B.3 analyzes the empirical risk minimization step and establishes the corresponding oracle guarantees, leading to Theorem~3.

\subsection{Proof of Theorem 1 (well-specified compositional model)}
We work on the oracle tessellation $H^\star=\{\cA_\ell^\star:\ell\in[L^\star]\}$ where Assumption \ref{Ass:transfer_function} holds and recall that,
for the squared-loss risk $\cR(f)=\E[(Y-f(X))^2]$ under the model $Y=\target(X)+\varepsilon$
with $\E[\varepsilon| X]=0$,
\begin{equation*}
\cR(f)-\cR(\target)=\|f-\target\|_{\L^2(\mu_X)}^2.
\end{equation*}
Hence it suffices to control $\|\wTL-\target\|_{\L^2(\mu_X)}^2$, where $\wTL=\wh f_{\cT}^{H^\star}$
denotes the transfer estimator constructed on $H^\star$.

\medskip
\noindent
\textit{Auxiliary functions and deterministic decomposition.}
Fix $\ell\in[L^\star]$ and denote by $x_{\ell}$ the representative point of $\cA_\ell^\star$.
Let $\ell^\star(x)$ be the unique index such that $x\in\cA_{\ell^\star(x)}^\star$.
Define the centered oracle score on cell $\ell$ for any $x\in\cA_\ell^\star$ by
\begin{equation*}
z_\ell(x):=\source(x)-\source(x_{\ell}),
\end{equation*}
with $x_{\ell^\star(x)}=x_\ell$ as $x\in\cA_\ell^\star$.

\smallskip
\noindent
\textit{Transfer linearization $\gslin$.}
For each cell $\ell\in[L^\star]$, let $(a_\ell^\star,b_\ell^\star)\in\R^2$ denote the population
cellwise least-squares coefficients,
\begin{equation*}
(a_\ell^\star,b_\ell^\star)
\in
\argmin_{a,b\in\R}
\E\Big[\big(\target(X)-a z_\ell(X)-b\big)^2 \,\Big|\, X\in\cA_\ell^\star\Big],
\end{equation*}
and define the piecewise-linear approximation
\begin{equation*}
\gslin(x)
=
a_{\ell^\star(x)}^\star z_{\ell^\star(x)}(x)+b_{\ell}^\star.
\end{equation*}

\smallskip
\noindent
\textit{Source oracle $\gabh$.}
Let $\overline\theta_\ell=(\overline b_\ell,\overline a_\ell)$ be the cellwise oracle estimator
computed from $\cT_1$ (using the true score $\source(X_i)-\source(x_{\ell})$) with the weights
$K_{x,h}(\|X_i-x_{\ell}\|)K_{z,\overline h}(|\source(X_i)-\source(x_{\ell})|)$,
and define the associated oracle transfer predictor
\begin{equation*}
\gabh(x)
=
\overline a_{\ell^\star(x)}z_{\ell^\star(x)}(x)+\overline b_{\ell^\star(x)}.
\end{equation*}

\smallskip
\noindent
\textit{Plug-in transfer estimator $\wTL$.}
Let $\widehat\theta_\ell=(\widehat b_\ell,\widehat a_\ell)$ be the corresponding plug-in estimator
obtained by replacing $\source$ with $\wsource$ in the score and in the $K_z$-weights, and define
\begin{equation*}
\wTL(x)
=
\widehat a_{\ell^\star(x)}\big(\wsource(x)-\wsource(x_{\ell^\star(x)})\big)
+
\widehat b_{\ell^\star(x)}.
\end{equation*}

\smallskip
\noindent
Using $(u+v+w)^2\le 3(u^2+v^2+w^2)$ and the orthogonality property of least-squares projections
(cellwise, hence globally), we have 
\begin{equation*}
\|\wTL-\target\|_{\L^2(\mu_X)}^2
\le
3\,\mathrm{Approx}(H^\star)
+
3\,\mathrm{Fit}_{\cT_1}(H^\star)
+
3\,\mathrm{Plug}_{\cS}(H^\star),
\end{equation*}
where
\begin{equation}\label{eq:approx_fit_star_def}
\mathrm{Approx}(H^\star):=\|\target-\gslin\|_{\L^2(\mu_X)}^2,
\qquad
\mathrm{Fit}_{\cT_1}(H^\star):=\|\gslin-\gabh\|_{\L^2(\mu_X)}^2,
\end{equation}
and
\begin{equation}\label{eq:plug_star_def}
\mathrm{Plug}_{\cS}(H^\star):=\|\wTL-\gabh\|_{\L^2(\mu_X)}^2.
\end{equation}
We control the three terms in expectation.

\medskip
\noindent
\underline{\textbf{Step 1: Control of $\E[\mathrm{Approx}(H^\star)]$ \eqref{eq:approx_fit_star_def}.}}
Fix $\ell\in[L^\star]$. Under Assumption~\ref{Ass:transfer_function}, for $x\in\cA_\ell^\star$ we have
\begin{equation*}
\target(x)=g_\ell^\star(\source(x)).
\end{equation*}
Write $y=\source(x)$.
Since $g_\ell^\star\in \mathrm{\text{H\"ol}}(\beta_g,\frl_g;\R)$ with $\beta_g>1$
(Assumption~\ref{Ass:regularity_transfer}), A first-order Taylor expansion with H\"older remainder yields
\begin{equation*}
\big|g_\ell^\star(y)-g_\ell^\star(y_{\ell})-(g_\ell^\star)'(y_{\ell})(y-y_{\ell})\big|
\lesssim
\frl_g\,|y-y_{\ell}|^{\beta_g}.
\end{equation*}
Using the source H\"older regularity $\source\in\mathrm{\text{H\"ol}}(\beta_{\cS},\frl_{\cS};[0,1]^d)$
(Assumption~\ref{Ass:regularity_source}), for any $x\in\cA_\ell^\star$,
\begin{equation*}
|\source(x)-\source(x_{\ell})|
\lesssim
\frl_{\cS}\,\|x-x_{\ell}\|^{\beta_{\cS}}
\le
\frl_{\cS}\,\Delta_\ell^{\beta_{\cS}},
\end{equation*}
where $\Delta_\ell:=\mathrm{diam}(\cA_\ell^\star)$.
Hence, defining the explicit Taylor linearization on cell $\ell$ (for $x\in\cA_\ell^\star$) by
\begin{equation*}
\widetilde g_\ell(x)
:=
g_\ell^\star(\source(x_{\ell}))+(g_\ell^\star)'(\source(x_{\ell}))\big(\source(x)-\source(x_{\ell})\big),
\end{equation*}
we obtain the uniform bound on $\cA_\ell^\star$,
\begin{equation*}
|\target(x)-\widetilde g_\ell(x)|
=
\big|g_\ell^\star(\source(x))-\widetilde g_\ell(x)\big|
\lesssim
\frl_g\,\Delta_\ell^{\beta_g\beta_{\cS}}.
\end{equation*}
Since $\gslin$ is the $\L^2(\mu_X)$-projection of $\target$ onto the cellwise linear class
$\{x\mapsto a(\source(x)-\source(x_{\ell}))+b\}$ on $\cA_\ell^\star$, we have
\begin{align*}
\int_{\cA_\ell^\star}\big(\target(x)-\gslin(x)\big)^2\,d\mu_X(x)
&\le
\int_{\cA_\ell^\star}\big(\target(x)-\widetilde g_\ell(x)\big)^2\,d\mu_X(x)\\
&\lesssim
\frl_g^2\,\Delta_\ell^{2\beta_g\beta_{\cS}}\,\mu_X(\cA_\ell^\star).
\end{align*}
Summing over $\ell$ yields
\begin{equation*}
\mathrm{Approx}(H^\star)
\lesssim
\frl_g^2\sum_{\ell=1}^{L^\star}\mu_X(\cA_\ell^\star)\Delta_\ell^{2\beta_g\beta_{\cS}}
\le
\frl_g^2\,\Delta_{\max}(H^\star)^{2\beta_g\beta_{\cS}},
\end{equation*}
where $\Delta_{\max}(H^\star):=\max_{\ell\in[L^\star]}\Delta_\ell$.
By Assumption~\ref{Ass:quasi_uniform_target} (or the corresponding part of Assumption~\ref{Ass:quasi_uniform}),
$\Delta_{\max}(H^\star)\lesssim (L^\star)^{-1/d}$, hence
\begin{equation*}
\mathrm{Approx}(H^\star)\lesssim (L^\star)^{-2\beta_g\beta_{\cS}/d}.
\end{equation*}
In particular, this bound is deterministic and therefore 
\begin{equation*}
\E[\mathrm{Approx}(H^\star)]\lesssim (L^\star)^{-2\beta_g\beta_{\cS}/d}.
\end{equation*}

\medskip
\noindent
\underline{\textbf{Step 2: Control of $\E[\mathrm{Fit}_{\cT_1}(H^\star)]$ \eqref{eq:approx_fit_star_def}.}}
We control the statistical error incurred by estimating the linear coefficients on each cell
using the oracle score.
Under Assumptions~\ref{Ass:local-design} and~\ref{Ass:ESS_plugin}, the weighted Gram matrices are uniformly
well-conditioned on each cell and the effective sample size satisfies
$n_{\mathrm{eff}}^{\star,\ell}\asymp n_{\cT_1}h^d\overline h$ (uniformly in $\ell$) on the event
$\cE_{\mathrm{ess}}$.
Moreover, by Assumption~\ref{Ass:noise} and standard concentration for weighted least squares with
sub-exponential noise (applied cellwise and then combined over $\ell$), there exists a constant $\frc>0$ such that,
for all $t\ge 1$, on an event of probability at least $1-2e^{-t}$,
\begin{equation*}
\max_{\ell\in[L^\star]}\|\overline\theta_\ell-\theta_\ell^\star\|_2^2
\lesssim
\frac{t}{n_{\cT_1}h^d\overline h}.
\end{equation*}
Integrating this tail bound with respect to $t$ (or equivalently using $\E[\max_\ell U_\ell]\lesssim \log(L^\star)\sup_t$-type bounds),
we obtain
\begin{equation*}
\E\Big[\max_{\ell\in[L^\star]}\|\overline\theta_\ell-\theta_\ell^\star\|_2^2\Big]
\lesssim
\frac{1+\log(L^\star)}{n_{\cT_1}h^d\overline h}.
\end{equation*}
Finally, using that on each cell $\cA_\ell^\star$ the feature vector is two-dimensional and uniformly bounded
(on the support of the local weights, by boundedness of $K_z$ and the score localization $|\source(X)-\source(x_{\ell})|\le \overline h$),
we have
\begin{equation*}
\mathrm{Fit}_{\cT_1}(H^\star)
=
\sum_{\ell=1}^{L^\star}\mu_X(\cA_\ell^\star)\,
\E\Big[\langle\phi_\ell^\star(X),\overline\theta_\ell-\theta_\ell^\star\rangle^2\,\Big|\,X\in\cA_\ell^\star\Big]
\lesssim
\max_{\ell\in[L^\star]}\|\overline\theta_\ell-\theta_\ell^\star\|_2^2,
\end{equation*}
since features $\phi_\ell^\star$ are automatically bounded (by compactness of $[0,1]^d$ and smoothness of $\source$). Hence, by Lemma \ref{lem:cellwise_ols_tail_union_ess},
\begin{equation*}
\E[\mathrm{Fit}_{\cT_1}(H^\star)]
\lesssim
\frac{1+\log(L^\star)}{n_{\cT_1}h^d\overline h}.
\end{equation*}

\medskip
\noindent
\underline{\textbf{Step 3: Control of $\E[\mathrm{Plug}_{\cS}(H^\star)]$ \eqref{eq:plug_star_def}.}}
We now control the additional error incurred by replacing $\source$ with $\wsource$ in both the score and the
weights, and in the final prediction.
Write for $x\in\cA_{\ell^\star(x)}^\star$,
\begin{equation*}
\gabh(x)=\overline a_{\ell^\star(x)}\big(\source(x)-\source(x_{\ell^\star(x)})\big)+\overline b_{\ell^\star(x)},
\end{equation*}
and
\begin{equation*}
\wTL(x)=\widehat a_{\ell^\star(x)}\big(\wsource(x)-\wsource(x_{\ell^\star(x)})\big)+\widehat b_{\ell^\star(x)}.
\end{equation*}
so that we have
\begin{multline*}
\wTL(x)-\gabh(x)
=
\langle\widehat\theta_{\ell^\star(x)}-\overline\theta_{\ell^\star(x)},\widehat\phi_{\ell^\star(x)}(x)\rangle
\\+
\overline a_{\ell^\star(x)}\Big((\wsource-\source)(x)-(\wsource-\source)(x_{\ell^\star(x)})\Big),
\end{multline*}
where $\widehat\phi_{\ell^\star(x)}(x)=(1,\wsource(x)-\wsource(x_{\ell^\star(x)}))^\top$.
Therefore,
\begin{multline*}
\mathrm{Plug}_{\cS}(H^\star)
\lesssim
\sum_{\ell=1}^{L^\star}\mu_X(\cA_\ell^\star)\,
\|\widehat\theta_\ell-\overline\theta_\ell\|_2^2
\\+
\Big(\max_{\ell\in[L^\star]}|\overline a_\ell|^2\Big)
\Bigg(
\|\wsource-\source\|_{\L^2(\mu_X)}^2
+
\max_{\ell\in[L^\star]}|\wsource(x_{\ell})-\source(x_{\ell})|^2
\Bigg).
\end{multline*}

\smallskip
\noindent
\textit{(i) Oracle vs.\ plug-in coefficient gap.}
By Lemma~\ref{lem:oracle-plugin-gap} applied on each cell and a union bound over $\ell\in[L^\star]$,
on the event $\cE_{\mathrm{ess}}$ and under the plug-in condition \eqref{eq:plug-in_condition},
we obtain
\begin{equation*}
\E\Big[\max_{\ell\in[L^\star]}\|\widehat\theta_\ell-\overline\theta_\ell\|_2^2\Big]
\lesssim
\Big(\frac{\E[\epsilon_{\cS}]}{\overline h}\Big)^2
+
\frac{1+\log(L^\star)}{n_{\cT_1}h^d\overline h}.
\end{equation*}
Moreover, the plug-in regime \eqref{eq:plug-in_condition} ensures that the second term dominates the purely quadratic
$\big(\log(L^\star)/(n_{\cT_1}h^d\overline h)\big)^2$ contribution arising in Lemma~\ref{lem:oracle-plugin-gap}.
Using the sup-norm control \eqref{Eq:epsS_rate} for $\epsilon_{\cS}$ and the fact that
$\E[\epsilon_{\cS}^2]\lesssim (1+\log(L^\star))\,n_{\cS}^{-2\beta_{\cS}/(2\beta_{\cS}+d)}$ (by integrating the tail bound),
we obtain
\begin{equation*}
\E\Big[\sum_{\ell=1}^{L^\star}\mu_X(\cA_\ell^\star)\|\widehat\theta_\ell-\overline\theta_\ell\|_2^2\Big]
\lesssim
\frac{1+\log(L^\star)}{n_{\cT_1}h^d\overline h}
+
(1+\log(L^\star))\,n_{\cS}^{-\frac{2\beta_{\cS}}{2\beta_{\cS}+d}}.
\end{equation*}

\smallskip
\noindent
\textit{(ii) Pure source plug-in error in the score.}
By standard nonparametric regression bounds for the Nadaraya-Watson estimator (see Appendix \ref{App:NW}) under
Assumptions~\ref{Ass:design_source}, \ref{Ass:regularity_source}, \ref{Ass:noise} and \ref{Ass:kernels}, we have
\begin{equation*}
\E\big[\|\wsource-\source\|_{\L^2(\mu_X)}^2\big]\lesssim n_{\cS}^{-\frac{2\beta_{\cS}}{2\beta_{\cS}+d}},
\end{equation*}
and, by a union bound over the $L^\star$ anchor points (together with the usual pointwise deviation for $\wsource$),
\begin{equation*}
\E\Big[\max_{\ell\in[L^\star]}|\wsource(x_{\ell})-\source(x_{\ell})|^2\Big]
\lesssim
(1+\log(L^\star))\,n_{\cS}^{-\frac{2\beta_{\cS}}{2\beta_{\cS}+d}}.
\end{equation*}
Finally, under Assumption~\ref{Ass:regularity_transfer} and the local Gram regularity
(Assumption~\ref{Ass:local-design}), the oracle slope coefficients $\overline a_\ell$ are uniformly bounded in
second moment, and in particular
\begin{equation*}
\E\Big[\max_{\ell\in[L^\star]}|\overline a_\ell|^2\Big]\lesssim 1+\log(L^\star),
\end{equation*}
so that
\begin{multline*}
\E\Big[\Big(\max_{\ell\in[L^\star]}|\overline a_\ell|^2\Big)
\Big(\|\wsource-\source\|_{\L^2(\mu_X)}^2+\max_{\ell\in[L^\star]}|\wsource(x_{\ell})-\source(x_{\ell})|^2\Big)\Big]
\\\lesssim
(1+\log(L^\star))^2\,n_{\cS}^{-\frac{2\beta_{\cS}}{2\beta_{\cS}+d}}.
\end{multline*}
Combining $(i)$ and $(ii)$, we obtain
\begin{equation*}
\E[\mathrm{Plug}_{\cS}(H^\star)]
\lesssim
\frac{1+\log(L^\star)}{n_{\cT_1}h^d\overline h}
+
(1+\log(L^\star))^2\,n_{\cS}^{-\frac{2\beta_{\cS}}{2\beta_{\cS}+d}}.
\end{equation*}

\medskip
\noindent
\underline{\textbf{Step 4: Conclusion.}}
Taking expectations in the deterministic decomposition and combining the bounds from Steps~1-3 yields
\begin{align*}
\E\big[\cR(\wTL)-\cR(\target)\big]
&=
\E\big[\|\wTL-\target\|_{\L^2(\mu_X)}^2\big]\\
&\lesssim
(L^\star)^{-2\beta_g\beta_{\cS}/d}
+
\frac{1+\log(L^\star)}{n_{\cT_1}h^d\overline h}
+
(1+\log(L^\star))^2\,n_{\cS}^{-\frac{2\beta_{\cS}}{2\beta_{\cS}+d}},
\end{align*}
which provides the expected bound.

\subsection{Proof of Theorem~2 (local linear transfer model)}
\label{App:fixed_tessellation_rate}

The proof of Theorem~2 follows the same overall structure as that of Theorem~1. The main differences are the following:
\begin{itemize}
\item the bounds are established in high probability rather than in expectation;
\item the analysis is carried out for a fixed tessellation $H$, instead of the oracle tessellation;
\item Theorem~2 is stated under the \nameref{subsec:well-specified_local} setting, rather than the \nameref{subsec:well-specified}, and consequently the approximation term $\mathrm{Approx}(H)$ differs.
\end{itemize}
For these reasons, and for the sake of completeness, we provide the full proof of Theorem~2, even though it shares substantial similarities (and some repetitions) with the proof of Theorem~1.
\\

\noindent
Fix a tessellation $H\in\cH$.
Recall that for the squared-loss risk $\cR(f)=\E[(Y-f(X))^2]$ under the model
$Y=\target(X)+\varepsilon$ with $\E[\varepsilon| X]=0$,
\begin{equation*}
\cR(f)-\cR(\target)=\|f-\target\|_{\L^2(\mu_X)}^2.
\end{equation*}
We therefore control $\|\wtess-\target\|_{\L^2(\mu_X)}^2$.

\medskip
\noindent
\textit{Auxiliary functions and deterministic decomposition.}
For any $\ell\in[L_H]$, define the centered source regressor
\begin{equation*}
z_{H,\ell}: x\in\cA_{H,\ell}\mapsto \source(x)-\source(x_{H,\ell}),
\end{equation*}
where $x_{H,\ell}$ is the anchor point of $\cA_{H,\ell}$.
Let $\ell_H:[0,1]^d\to[L_H]$ denote the index of the cell containing $x$.

\smallskip
\noindent
\textit{Transfer linearization $\gslin$.}
For each cell $\cA_{H,\ell}$, define the population cellwise least-squares coefficients
\begin{equation*}
(a_{H,\ell}^\star,b_{H,\ell}^\star)
\in
\argmin_{a,b\in\R}
\E\Big[\big(\target(X)-a z_{H,\ell}(X)-b\big)^2 \,\Big|\, X\in\cA_{H,\ell}\Big].
\end{equation*}
Define the piecewise-linear function
\begin{equation*}
\gslin(x)
=
a_{H,\ell_H(x)}^\star
\bigl(\source(x)-\source(x_{H,\ell_H(x)})\bigr)
+
b_{H,\ell_H(x)}^\star.
\end{equation*}

\smallskip
\noindent
\textit{Source oracle $\gabh$.}
Let $\ov{\theta}_{H,\ell}=(\ov b_{H,\ell},\ov a_{H,\ell})$ denote the cellwise oracle estimator
computed from $\cT_1$ (using the true score $\source(X_i)-\source(x_{H,\ell})$), and define
\begin{equation*}
\gabh(x)
=
\ov a_{H,\ell_H(x)}
\bigl(\source(x)-\source(x_{H,\ell_H(x)})\bigr)
+
\ov b_{H,\ell_H(x)}.
\end{equation*}

\smallskip
\noindent
Using Pythagoras'
theorem (since $\gslin$ is the $\L^2(\mu_X)$-projection of $\target$ on the
cellwise linear class induced by $z_{H,\ell}$), we obtain the deterministic decomposition
\begin{equation*}
\|\target-\wtess\|_{\L^2(\mu_X)}^2
\le
3\,\mathrm{Approx}(H)
+
3\,\mathrm{Fit}_{\cT_1}(H)
+
3\,\mathrm{Plug}_{\cS}(H),
\end{equation*}
where
\begin{equation}\label{eq:approx_fit_def}
\mathrm{Approx}(H):=\|\target-\gslin\|_{\L^2(\mu_X)}^2,
\qquad
\mathrm{Fit}_{\cT_1}(H):=\|\gslin-\gabh\|_{\L^2(\mu_X)}^2,
\end{equation}
and
\begin{equation}\label{eq:plug_def}
\mathrm{Plug}_{\cS}(H):=\|\wtess-\gabh\|_{\L^2(\mu_X)}^2.
\end{equation}

\medskip
\noindent
\underline{\textbf{Control of $\mathrm{Approx}(H)$ \eqref{eq:approx_fit_def}.}}
Fix $\ell\in[L_H]$ and take $x'=x_{H,\ell}$ in Assumption~\ref{Ass:local_transfer}.
For all $x\in\cA_{H,\ell}$,
\begin{align*}
\Big|\target(x)-\big(a^\star(x_{H,\ell})(\source(x)-\source(x_{H,\ell}))+b^\star(x_{H,\ell})\big)\Big|
&\le
\frlloc\,\|x-x_{H,\ell}\|^{1+\betaloc}\\
&\le
\frlloc\,\Delta_{H,\ell}^{1+\betaloc}.
\end{align*}
Hence the best linear predictor in the score $z_{H,\ell}$ satisfies
\begin{equation*}
\E\Big[\big(\target(X)-a z_{H,\ell}(X)-b\big)^2 \,\Big|\, X\in\cA_{H,\ell}\Big]
\lesssim
\Delta_{H,\ell}^{2(1+\betaloc)},
\end{equation*}
and therefore
\begin{equation*}
\mathrm{Approx}(H)
\le
\sum_{\ell\in[L_H]} p_{H,\ell}\,\Delta_{H,\ell}^{2(1+\betaloc)}
\le
\Delta_{\max}(H)^{2(1+\betaloc)}.
\end{equation*}
By Assumption~\ref{Ass:quasi_uniform}, $\Delta_{\max}(H)\lesssim L_H^{-1/d}$, hence
\begin{equation*}
\mathrm{Approx}(H)\lesssim L_H^{-2(1+\betaloc)/d}.
\end{equation*}

\medskip
\noindent
\underline{\textbf{Control of $\mathrm{Fit}_{\cT_1}(H)$ \eqref{eq:approx_fit_def}.}}
Work on the event $\cE_{\mathrm{ess}}$ from Assumption~\ref{subass:ESS}.
Apply Lemma~\ref{lem:cellwise_ols_tail_union_ess} with confidence level $\delta/3$.
There exists an event $\mathcal O_{\delta/3}$ such that
\begin{equation*}
\P(\mathcal O_{\delta/3}| \cE_{\mathrm{ess}})\ge 1-\delta/3,
\end{equation*}
and on $\mathcal O_{\delta/3}$, simultaneously for all $\ell\in[L_H]$,
\begin{equation*}
\|\ov{\theta}_{H,\ell}-\theta^\star_{H,\ell}\|_2^2
\lesssim
\frac{\sigma_{\cT}^2}{\lambda_{H,\min}^2}\,
\frac{\log(2L_H/\delta)}{n_{\mathrm{eff}}^{H,\ell}}
+
\frac{B_{H,\ell}^2}{\lambda_{H,\min}^2}.
\end{equation*}
Under Assumption~\ref{subass:ESS}, on $\cE_{\mathrm{ess}}$ we have
$n_{\mathrm{eff}}^{H,\ell}\asymp n_{\cT_1}h^d\overline h$ uniformly in $\ell$.
\\
\noindent
It remains to bound $B_{H,\ell}$. Define the cellwise remainder, for all $x\in\cA_{H,\ell}$, by
\begin{equation*}
r_{H,\ell}(x)
:=
\target(x)-\big(a^\star_{H,\ell}(\source(x)-\source(x_{H,\ell}))+b^\star_{H,\ell}\big),
\qquad .
\end{equation*}
By Assumption~\ref{Ass:local_transfer} (taking again $x'=x_{H,\ell}$), for all $x\in\cA_{H,\ell}$,
\begin{equation*}
|r_{H,\ell}(x)|\le \frlloc\,\Delta_{H,\ell}^{1+\betaloc}.
\end{equation*}
Moreover, by Assumption~\ref{Ass:bounded_features},
$\|\phi_{H,\ell}(x)\|_2\le \phi_{\max}$, so Lemma~\ref{lem:cellwise_ols_tail_union_ess}
yields
\begin{equation*}
B_{H,\ell}
\le
\phi_{\max}\,\sup_{x\in\cA_{H,\ell}}|r_{H,\ell}(x)|
\lesssim
\Delta_{H,\ell}^{1+\betaloc}.
\end{equation*}
Combining these bounds, on $\cE_{\mathrm{ess}}\cap\mathcal O_{\delta/3}$ we obtain for all $\ell\in[L_H]$,
\begin{equation*}
\|\ov{\theta}_{H,\ell}-\theta^\star_{H,\ell}\|_2^2
\lesssim
\frac{\log(2L_H/\delta)}{n_{\cT_1}h^d\overline h}
+
\Delta_{H,\ell}^{2(1+\betaloc)}.
\end{equation*}
Finally, using $\|\phi_{H,\ell}(X)\|_2\le \phi_{\max}$ and the definition of $\mathrm{Fit}_{\cT_1}(H)$,
\begin{align*}
\mathrm{Fit}_{\cT_1}(H)
&=
\sum_{\ell\in[L_H]} p_{H,\ell}\,
\E\Big[
\langle\phi_{H,\ell}(X),\ov{\theta}_{H,\ell}-\theta^\star_{H,\ell}\rangle^2
\ \Big|\ X\in\cA_{H,\ell}
\Big]\\
&\le
\phi_{\max}^2
\sum_{\ell\in[L_H]} p_{H,\ell}\,\|\ov{\theta}_{H,\ell}-\theta^\star_{H,\ell}\|_2^2,
\end{align*}
where we have used Assumption \ref{Ass:bounded_features}.
Then on $\cE_{\mathrm{ess}}\cap\mathcal O_{\delta/3}$,
\begin{equation*}
\mathrm{Fit}_{\cT_1}(H)
\lesssim
\frac{\log(2L_H/\delta)}{n_{\cT_1}h^d\overline h}
+
\sum_{\ell\in[L_H]} p_{H,\ell}\Delta_{H,\ell}^{2(1+\betaloc)}
\lesssim
\frac{\log(2L_H/\delta)}{n_{\cT_1}h^d\overline h}
+
L_H^{-2(1+\betaloc)/d}.
\end{equation*}
The last term is of the same order as $\mathrm{Approx}(H)$ and can be absorbed into it.

\medskip
\noindent
\underline{\textbf{Control of $\mathrm{Plug}_{\cS}(H)$.}}
Write, for $x\in\cA_{H,\ell}$,
\begin{equation*}
\gabh(x)=\ov a_{H,\ell}\big(\source(x)-\source(x_{H,\ell})\big)+\ov b_{H,\ell},
\quad
\wtess(x)=\widehat a_{H,\ell}\big(\wsource(x)-\wsource(x_{H,\ell})\big)+\widehat b_{H,\ell}.
\end{equation*}
Decompose $\wtess-\gabh$ into a coefficient gap plus a pure plug-in term:
\begin{align*}
\wtess(x)-\gabh(x)
&=
(\widehat a_{H,\ell}-\ov a_{H,\ell})\big(\source(x)-\source(x_{H,\ell})\big)
+(\widehat b_{H,\ell}-\ov b_{H,\ell})\\
&\quad
+\widehat a_{H,\ell}\Big(\big(\wsource(x)-\wsource(x_{H,\ell})\big)-\big(\source(x)-\source(x_{H,\ell})\big)\Big).
\end{align*}
Using Assumption~\ref{Ass:bounded_features}
(to control $\|\phi_{H,\ell}(x)\|_2$), we obtain
\begin{multline*}
\mathrm{Plug}_{\cS}(H)
\lesssim
\sum_{\ell\in[L_H]} p_{H,\ell}\,\|\widehat\theta_{H,\ell}-\ov{\theta}_{H,\ell}\|_2^2
\;+\;
\Big(\max_{\ell\in[L_H]}|\widehat a_{H,\ell}|^2\Big)\,
\|\wsource-\source\|_{\L^2(\mu_X)}^2
\\+
\Big(\max_{\ell\in[L_H]}|\widehat a_{H,\ell}|^2\Big)\,
\max_{\ell\in[L_H]}|\wsource(x_{H,\ell})-\source(x_{H,\ell})|^2.
\end{multline*}
\noindent
Now intersect with a standard source event $\cE_{\cS}$ such that
$\P(\cE_{\cS})\ge 1-\delta/3$ and on which
\begin{equation*}
\|\wsource-\source\|_{\L^2(\mu_X)}^2
\lesssim
n_{\cS}^{-\frac{2\beta_{\cS}}{2\beta_{\cS}+d}},
\end{equation*}
and
\begin{equation*}
\max_{\ell\in[L_H]}|\wsource(x_{H,\ell})-\source(x_{H,\ell})|^2
\lesssim
\log\Big(\frac{L_H}{\delta}\Big)\,
n_{\cS}^{-\frac{2\beta_{\cS}}{2\beta_{\cS}+d}}.
\end{equation*}
Moreover, on $\cE_{\mathrm{ess}}\cap\mathcal O_{\delta/3}$ and using
Assumption~\ref{Ass:bounded_features}, we have a uniform control
$\max_{\ell}|\widehat a_{H,\ell}|\lesssim 1+\log(L_H/\delta)$ (the logarithm comes from
the same union bound over cells used to control the local fits).
Finally, applying Lemma~\ref{lem:oracle-plugin-gap} at level $\delta/(3L_H)$ and taking
a union bound over $\ell\in[L_H]$, we obtain an event $\mathcal E_{\delta/3}$ such that
$\P(\mathcal E_{\delta/3}| \cE_{\mathrm{ess}})\ge 1-\delta/3$ and on which
\begin{equation*}
\max_{\ell\in[L_H]}\|\widehat\theta_{H,\ell}-\ov{\theta}_{H,\ell}\|_2^2
\lesssim
\frac{\log(L_H/\delta)}{n_{\cT_1}h^d\overline h},
\end{equation*}
where we used the plug-in regime $n_{\cT_1}h^d\overline h\gtrsim \log(1/\delta)$ and
the smallness condition $\epsilon_{\cS}/\overline h\le \frc_0$ implicit in
Lemma~\ref{lem:oracle-plugin-gap}.
\\
Combining the previous displays yields, on
$\cE_{\mathrm{ess}}\cap\mathcal O_{\delta/3}\cap\mathcal E_{\delta/3}\cap\cE_{\cS}$,
\begin{equation*}
\mathrm{Plug}_{\cS}(H)
\lesssim
\frac{\log(L_H/\delta)}{n_{\cT_1}h^d\overline h}
+
\log\Big(\frac{L_H}{\delta}\Big)\Bigl(1+\log\Big(\frac{L_H}{\delta}\Big)\Bigr)\,
n_{\cS}^{-\frac{2\beta_{\cS}}{2\beta_{\cS}+d}}.
\end{equation*}

\medskip
\noindent
\underline{\textbf{Conclusion.}}
Let
\begin{equation*}
\mathcal G_\delta
:=
\cE_{\mathrm{ess}}\cap\mathcal O_{\delta/3}\cap\mathcal E_{\delta/3}\cap \cE_{\cS}.
\end{equation*}
By Assumption~\ref{subass:ESS} and the conditional probability statements above, a union bound gives
$\P(\mathcal G_\delta)\ge 1-\delta$.
On $\mathcal G_\delta$, combining the bounds on $\mathrm{Approx}(H)$,
$\mathrm{Fit}_{\cT_1}(H)$ and $\mathrm{Plug}_{\cS}(H)$ and absorbing the
$L_H^{-2(1+\betaloc)/d}$ contribution from $\mathrm{Fit}_{\cT_1}(H)$ into
$\mathrm{Approx}(H)$ yields
\begin{align*}
\cR(\wtess)-\cR(\target)
&=
\|\wtess-\target\|_{\L^2(\mu_X)}^2\\
&\lesssim
L_H^{-2(1+\beta_{\mathrm{loc}})/d}
+
\frac{1}{n_{\cT_1}h^d\overline h}
\log\Big(\frac{L_H}{\delta}\Big)
\\
&\quad+
\log\Big(\frac{L_H}{\delta}\Big)
\Bigl(1+\log\Big(\frac{L_H}{\delta}\Big)\Bigr)
\,n_{\cS}^{-\frac{2\beta_{\cS}}{2\beta_{\cS}+d}}.
\end{align*}
This proves the result.

\subsection{Empirical risk minimization}\label{App:ERM}

\subsubsection{Proof of Proposition 1 (ERM oracle inequality in expectation)} 
\label{App:ERM_expectation}
\noindent

\begin{lem}[Uniform second moment of the validation loss]\label{lem:uniform_second_moment_loss}
Suppose Assumptions~\ref{Ass:noise}, \ref{Ass:H_local-design}, 
\ref{Ass:bounded_features} and \ref{Ass:ESS_plugin} hold.
Then there exists a constant $\frc>0$ such that, on the stability event
$\cE_{\mathrm{ess}}$ 
\begin{equation*}
\sup_{H\in\cH}
\E\left[\cL_H(X,Y)^2 \,\big|\, \cE_{\mathrm{ess}}\right]
\le \frc,
\end{equation*}
where $\cL_H(X,Y)=(Y-\wh\mu_H(X))^2$.
\end{lem}

\begin{proof}
Write $Y=\target(X)+\varepsilon$ and set
$\Delta_H(X):=\wtess(X)-\target(X)$. Then
\begin{equation*}
\cL_H(X,Y)^2
=
\big(\varepsilon-\Delta_H(X)\big)^4
\le
8\big(\varepsilon^4+\Delta_H(X)^4\big).
\end{equation*}
By Assumption~\ref{Ass:noise}, $\E[\varepsilon^4]<\infty$.
It therefore suffices to bound
\begin{equation*}
\sup_{H\in\cH}\E\!\left[\Delta_H(X)^4 \,\big|\, \cE_{\mathrm{ess}}\right],
\end{equation*}
on the stability event $\cE_{\mathrm{ess}}$ and under the local Gram
regularity assumption (Assumption~\ref{Ass:H_local-design}) together with
the corresponding weighted Gram condition
\begin{equation*}
\lambda_{H,\min} I_2
\preceq
G_{H,\ell}
\preceq
\lambda_{H,\max} I_2,
\qquad
\forall H\in\cH,\ \forall \ell\in[L_H].
\end{equation*}
Fix $H\in\cH$ and let $\ell_H(X)$ denote the index such that
$X\in\cA_{H,\ell_H(X)}$.
On each cell, by construction of $\wtess$, for all $x\in\cA_{H,\ell}$ we have
\begin{equation*}
\wtess(x)
=
\widehat\theta_{H,\ell}^\top \phi_{H,\ell}(x),
\end{equation*}
and by the local linear decomposition, for all $x\in\cA_{H,\ell}$,
\begin{equation*}
\target(x)
=
(\theta^\star_{H,\ell})^\top \phi_{H,\ell}(x)
+
r_{H,\ell}(x).
\end{equation*}
Hence for $x\in\cA_{H,\ell}$,
\begin{equation*}
\Delta_H(x)
=
(\widehat\theta_{H,\ell}-\theta^\star_{H,\ell})^\top \phi_{H,\ell}(x)
-
r_{H,\ell}(x).
\end{equation*}
\noindent
Using $(u+v)^4\le 8(u^4+v^4)$ and Cauchy-Schwarz, we obtain
\begin{align*}
|\Delta_H(x)|^4
&\le
8\big|\langle \widehat\theta_{H,\ell}-\theta^\star_{H,\ell},
\phi_{H,\ell}(x)\rangle\big|^4
+
8|r_{H,\ell}(x)|^4\\
&\le
8\phi_{\max}^4
\|\widehat\theta_{H,\ell}-\theta^\star_{H,\ell}\|_2^4
+
8r_{\max}^4.
\end{align*}
On the stability event described above, the weighted least-squares estimator
$\widehat\theta_{H,\ell}$ is well defined and the Gram matrix $G_{H,\ell}$ is uniformly
well conditioned. In particular, by
Lemma~\ref{lem:cellwise_ols_tail_union_ess} and
Assumption~\ref{Ass:ESS_plugin}, the fourth moment of
$\|\widehat\theta_{H,\ell}-\theta^\star_{H,\ell}\|_2$
is finite and uniformly bounded over $(H,\ell)$ on $\cE_{\mathrm{ess}}$.
\\
Taking expectations yields
\begin{equation*}
\sup_{H\in\cH}
\E\!\left[\Delta_H(X)^4 \,\big|\, \cE_{\mathrm{ess}}\right]
\le
\frc_1
\end{equation*}
for some constant $\frc_1<\infty$ depending only on
$\phi_{\max}, r_{\max}$ and the Gram/ESS constants.
Combining with the first display concludes the proof.
\end{proof}

\noindent
\textit{Proof of Proposition 1.}
For $H\in\cH$, define
\begin{equation*}
\cR(H)=\E\!\left[(Y-\wtess(X))^2\right],
\qquad
\wh{\cR}(H)=\frac{1}{n_{\cT_2}}\sum_{i\in\cT_2}(Y_i-\wtess(X_i))^2.
\end{equation*}
Let $\cL_H(X,Y):=(Y-\wtess(X))^2$, so that
$\cR(H)=\E[\cL_H(X,Y)]$ and $\wh{\cR}(H)$ is the empirical mean of $\cL_H$
over $\cT_2$.\\
Conditionally on the training data $(\cD_{\cS},\cD_{\cT_1})$ used to construct
$\wtess$, the variables
$\{\cL_H(X_i,Y_i)\}_{i\in\cT_2}$ are i.i.d.\ with mean $\cR(H)$.
Hence
\begin{align*}
\E\left[
\big(\wh{\cR}(H)-\cR(H)\big)^2
\big|\,
\cD_{\cS},\cD_{\cT_1}
\right]
&=
\frac{\Var[\cL_H(X,Y)\mid \cD_{\cS},\cD_{\cT_1}]}{n_{\cT_2}}
\\
&\le
\frac{\E[\cL_H(X,Y)^2\mid \cD_{\cS},\cD_{\cT_1}]}{n_{\cT_2}}.
\end{align*}
\noindent
On the stability event $\cE_{\mathrm{ess}}$ together with the local Gram
regularity assumption (Assumption~\ref{Ass:H_local-design}) and the
corresponding weighted Gram condition
\begin{equation*}
\lambda_{H,\min} I_2
\preceq
G_{H,\ell}
\preceq
\lambda_{H,\max} I_2,
\end{equation*}
Lemma~\ref{lem:uniform_second_moment_loss} ensures that
\begin{equation*}
\sup_{H\in\cH}
\E\!\left[\cL_H(X,Y)^2\right]
\le
\frc_0
\end{equation*}
for some constant $\frc_0<\infty$.
Taking expectations therefore yields
\begin{equation*}
\sup_{H\in\cH}
\E\!\left[
\big(\wh{\cR}(H)-\cR(H)\big)^2
\right]
\le
\frac{\frc_0}{n_{\cT_2}}.
\end{equation*}

Now set $U_H:=\wh{\cR}(H)-\cR(H)$.
Using
\begin{equation*}
\sup_{H\in\cH}|U_H|
\le
\Big(\sum_{H\in\cH}U_H^2\Big)^{1/2}
\end{equation*}
and Jensen's inequality, we obtain
\begin{equation*}
\E\!\left[\sup_{H\in\cH}|U_H|\right]
\le
\E\!\left[\Big(\sum_{H\in\cH}U_H^2\Big)^{1/2}\right]
\le
\Big(\sum_{H\in\cH}\E[U_H^2]\Big)^{1/2}
\le
\sqrt{\frac{\frc_0|\cH|}{n_{\cT_2}}}.
\end{equation*}

Finally, by definition of $\wh H$ and $H^{\mathrm{or}}$,
\begin{equation*}
\cR(\wh H)
\le
\wh{\cR}(\wh H)
+
\sup_{H\in\cH}|\wh{\cR}(H)-\cR(H)|
\le
\wh{\cR}(H^{\mathrm{or}})
+
\sup_{H\in\cH}|\wh{\cR}(H)-\cR(H)|.
\end{equation*}
Adding and subtracting $\cR(H^{\mathrm{or}})$ and taking expectations yields
\begin{equation*}
\E[\cR(\wh H)]
\le
\cR(H^{\mathrm{or}})
+
2\,\E\!\left[\sup_{H\in\cH}|\wh{\cR}(H)-\cR(H)|\right]
\le
\cR(H^{\mathrm{or}})
+
\frc\,\sqrt{\frac{|\cH|}{n_{\cT_2}}},
\end{equation*}
for a universal constant $\frc>0$.

\subsubsection{Proof of Proposition 2 (Median-of-Means oracle inequality)}\label{App:ERM_MoM}

\noindent
We work conditionally on the training data used to build $\{\wh\mu_H:H\in\cH\}$, so that the validation
observations $\{(X_i,Y_i)\}_{i\in\cT_2}$ are i.i.d.\ and independent of $\{\wh\mu_H:H\in\cH\}$.
For $H\in\cH$, define the validation loss and risk
\begin{equation*}
\cL_H(X,Y):=\big(Y-\wh\mu_H(X)\big)^2,
\qquad
\cR(H):=\E\big[\cL_H(X,Y)\big].
\end{equation*}

\medskip
\noindent
\underline{\textbf{Step 1: Uniform second-moment bound for $\cL_H$.}}
Write $Y=\target(X)+\varepsilon$ and set $\Delta_H(X):=\wh\mu_H(X)-\target(X)$, so that
\begin{equation*}
\cL_H(X,Y)^2
=
\big(\varepsilon-\Delta_H(X)\big)^4
\le
8\big(\varepsilon^4+\Delta_H(X)^4\big).
\end{equation*}
By assumption, $\E[\varepsilon^4]<\infty$, hence it suffices to bound $\sup_{H\in\cH}\E[\Delta_H(X)^4]$.
\\
Fix $H\in\cH$ and let $\ell_H(X)$ be the cell index such that $X\in\cA_{H,\ell_H(X)}$.
On each cell $\cA_{H,\ell}$, by definition of $\wh\mu_H$ we may write
\begin{equation*}
\wh\mu_H(x)=\wh\theta_{H,\ell}^\top \phi_{H,\ell}(x),
\qquad x\in\cA_{H,\ell}.
\end{equation*}
Moreover, by the local linear decomposition (definition of $r_{H,\ell}$), for all $x\in\cA_{H,\ell}$,
\begin{equation*}
\target(x)=(\theta^\star_{H,\ell})^\top \phi_{H,\ell}(x)+r_{H,\ell}(x).
\end{equation*}
Therefore, for $x\in\cA_{H,\ell}$,
\begin{equation*}
\Delta_H(x)
=
(\wh\theta_{H,\ell}-\theta^\star_{H,\ell})^\top \phi_{H,\ell}(x)
-
r_{H,\ell}(x).
\end{equation*}
Using $(u+v)^4\le 8(u^4+v^4)$ and Cauchy-Schwarz, Assumption~\ref{Ass:bounded_features} yields
\begin{equation*}
|\Delta_H(x)|^4
\le
8\big|\langle \wh\theta_{H,\ell}-\theta^\star_{H,\ell},\phi_{H,\ell}(x)\rangle\big|^4
+
8|r_{H,\ell}(x)|^4
\le
8\phi_{\max}^4\|\wh\theta_{H,\ell}-\theta^\star_{H,\ell}\|_2^4
+
8r_{\max}^4.
\end{equation*}
Hence,
\begin{equation*}
\E\big[\Delta_H(X)^4\big]
\le
8\phi_{\max}^4\,\E\Big[\|\wh\theta_{H,\ell_H(X)}-\theta^\star_{H,\ell_H(X)}\|_2^4\Big]
+
8r_{\max}^4.
\end{equation*}
\noindent
It remains to show that
$\sup_{H\in\cH}\sup_{\ell\in[L_H]}\E[\|\wh\theta_{H,\ell}-\theta^\star_{H,\ell}\|_2^4] <\infty$.
Fix $(H,\ell)$ and condition on the design $\{X_i\}_{i\in\cT_1^{H,\ell}}$.
On $\cE_{\mathrm{ess}}$ and under the weighted Gram condition,
the weighted least-squares estimator satisfies the exact representation
\begin{equation*}
\wh\theta_{H,\ell}-\theta^\star_{H,\ell}
=
G_{H,\ell}^{-1}
\Bigg(
\frac{1}{n_{H,\ell}}\Psi_{H,\ell}^\top W_{H,\ell}\varepsilon
+
\frac{1}{n_{H,\ell}}\Psi_{H,\ell}^\top W_{H,\ell} r_{H,\ell}
\Bigg).
\end{equation*}
Using $\|G_{H,\ell}^{-1}\|_{\op}\le \lambda_{\min}^{-1}$ and the triangle inequality,
\begin{equation*}
\|\wh\theta_{H,\ell}-\theta^\star_{H,\ell}\|_2
\le
\frac{1}{\lambda_{\min}}
\Bigg(
\Big\|\frac{1}{n_{H,\ell}}\Psi_{H,\ell}^\top W_{H,\ell}\varepsilon\Big\|_2
+
B_{H,\ell}
\Bigg),
\end{equation*}
where
\begin{equation*}
B_{H,\ell}
:=
\Big\|
\frac{1}{n_{H,\ell}}\Psi_{H,\ell}^\top W_{H,\ell} r_{H,\ell}
\Big\|_2.
\end{equation*}
By Assumption~\ref{Ass:bounded_features}, for all $x\in\cA_{H,\ell}$ we have
$\|\phi_{H,\ell}(x)\|_2\le \phi_{\max}$ and $|r_{H,\ell}(x)|\le r_{\max}$, hence
\begin{equation*}
B_{H,\ell}
=
\Big\|
\frac{1}{n_{H,\ell}}\sum_{i\in\cT_1^{H,\ell}} w_{i,\ell}\phi_{H,\ell}(X_i)r_{H,\ell}(X_i)
\Big\|_2
\le
\frac{1}{n_{H,\ell}}\sum_{i\in\cT_1^{H,\ell}} |w_{i,\ell}|\,\phi_{\max}r_{\max}.
\end{equation*}
Since the kernels are bounded, $|w_{i,\ell}|\le \|K_x\|_\infty\|K_z\|_\infty h^{-d}\overline h^{-1}$,
and therefore $B_{H,\ell}<\infty$ deterministically for fixed $(h,\overline h)$.\\
For the noise term, note that each coordinate of
$n_{H,\ell}^{-1}\Psi_{H,\ell}^\top W_{H,\ell}\varepsilon$ is a weighted sum of the independent
centered variables $\{\varepsilon_i\}_{i\in\cT_1^{H,\ell}}$ with bounded weights and bounded
multipliers $\phi_{H,\ell}(X_i)$, hence has finite fourth moment whenever $\E[\varepsilon^4]<\infty$.
Moreover, on $\cE_{\mathrm{ess}}$ the effective sample size lower bound prevents degeneracy of the weights,
so the resulting fourth moments are bounded by a constant depending only on
$\E[\varepsilon^4]$, kernel envelopes, $(h,\overline h)$, and $\phi_{\max}$.
Consequently, there exists a constant $\frc_0<\infty$ such that, on $\cE_{\mathrm{ess}}$,
\begin{equation*}
\sup_{H\in\cH}\sup_{\ell\in[L_H]}
\E\Big[\|\wh\theta_{H,\ell}-\theta^\star_{H,\ell}\|_2^4 \,\Big|\, \cE_{\mathrm{ess}}\Big]
\le \frc_0.
\end{equation*}
Combining the previous displays yields that there exists $\frc_1<\infty$ such that, on $\cE_{\mathrm{ess}}$,
\begin{equation*}
\sup_{H\in\cH}\E\big[\cL_H(X,Y)^2 \,\big|\, \cE_{\mathrm{ess}}\big]
\le
\frc_1(\sigma^2+\cR_{\max})^2,
\end{equation*}
where we used that $\cR_{\max}$ controls $\sup_H \E[\Delta_H(X)^2]$ and thus fixes the scale of the
second moment (up to universal constants).

\medskip
\noindent
\underline{\textbf{Step 2: Median-of-means deviation for a fixed $H$.}}
Partition $\cT_2$ into $B$ disjoint blocks $\{\cB_b\}_{b=1}^B$ of equal size
$m:=\lfloor n_{\cT_2}/B\rfloor$.
For each $H\in\cH$ and $b\in[B]$, define
\begin{equation*}
\wh\cR_b(H)
:=
\frac{1}{m}\sum_{i\in\cB_b}\cL_H(X_i,Y_i),
\qquad
\wh\cR_{\mathrm{MoM}}(H)
:=
\mathrm{median}\big(\wh\cR_1(H),\dots,\wh\cR_B(H)\big).
\end{equation*}
Fix $H\in\cH$.
On $\cE_{\mathrm{ess}}$, by Chebyshev's inequality and Step~1, there exists a universal constant $\frc_2>0$
such that for all $t>0$,
\begin{equation*}
\P\Big(|\wh\cR_b(H)-\cR(H)|>t \,\Big|\, \cE_{\mathrm{ess}}\Big)
\le
\frac{\Var(\cL_H(X,Y)\mid \cE_{\mathrm{ess}})}{m t^2}
\le
\frac{\frc_2(\sigma^2+\cR_{\max})^2}{m t^2}.
\end{equation*}
Choose
\begin{equation*}
t
=
\frc_3(\sigma^2+\cR_{\max})\sqrt{\frac{1}{m}},
\end{equation*}
with $\frc_3>0$ large enough so that the above probability is at most $1/8$.
Then, for each $b$, the indicators $\mathbf 1_{\{|\wh\cR_b(H)-\cR(H)|>t\}}$ are Bernoulli with mean at most $1/8$.
By a Chernoff bound, there exists a universal constant $\frc_4>0$ such that
\begin{equation*}
\P\Big(
\#\{b\in[B]:|\wh\cR_b(H)-\cR(H)|>t\}\ge B/2
\,\Big|\,\cE_{\mathrm{ess}}
\Big)
\le
2\exp(-\frc_4 B).
\end{equation*}
Since the median is bad only if at least half of the blocks are bad, this yields
\begin{equation*}
\P\Big(
|\wh\cR_{\mathrm{MoM}}(H)-\cR(H)|>t
\,\Big|\,\cE_{\mathrm{ess}}
\Big)
\le
2\exp(-\frc_4 B).
\end{equation*}

\medskip
\noindent
\underline{\textbf{Step 3: Uniform control and oracle inequality.}}
Applying a union bound over $H\in\cH$ gives
\begin{equation*}
\P\Big(
\sup_{H\in\cH}|\wh\cR_{\mathrm{MoM}}(H)-\cR(H)|>t
\,\Big|\,\cE_{\mathrm{ess}}
\Big)
\le
2|\cH|\exp(-\frc_4 B).
\end{equation*}
Choosing $B\ge \frc\log(|\cH|/\delta)$ with $\frc$ large enough ensures that the right-hand side is at most $\delta$.
Moreover, since $m\asymp n_{\cT_2}/B$, the choice of $t$ in Step~2 gives
\begin{equation*}
t
\asymp
(\sigma^2+\cR_{\max})\sqrt{\frac{B}{n_{\cT_2}}}
\asymp
(\sigma^2+\cR_{\max})\sqrt{\frac{\log(|\cH|/\delta)}{n_{\cT_2}}}.
\end{equation*}
Hence, with conditional probability at least $1-\delta$ given $\cE_{\mathrm{ess}}$,
\begin{equation*}
\sup_{H\in\cH}|\wh\cR_{\mathrm{MoM}}(H)-\cR(H)|
\lesssim
(\sigma^2+\cR_{\max})\sqrt{\frac{\log(|\cH|/\delta)}{n_{\cT_2}}}.
\end{equation*}
On this event, let $H^{\mathrm{or}}\in\arg\min_{H\in\cH}\cR(H)$ and recall that
$\wh H_{\mathrm{MoM}}\in\arg\min_{H\in\cH}\wh\cR_{\mathrm{MoM}}(H)$, so
\begin{align*}
\cR(\wh H_{\mathrm{MoM}})
&\le
\wh\cR_{\mathrm{MoM}}(\wh H_{\mathrm{MoM}})
+
\sup_{H\in\cH}|\wh\cR_{\mathrm{MoM}}(H)-\cR(H)|\\
&\le
\wh\cR_{\mathrm{MoM}}(H^{\mathrm{or}})
+
\sup_{H\in\cH}|\wh\cR_{\mathrm{MoM}}(H)-\cR(H)|\\
&\le
\cR(H^{\mathrm{or}})
+
2\sup_{H\in\cH}|\wh\cR_{\mathrm{MoM}}(H)-\cR(H)|.
\end{align*}
This yields, with probability at least $1-\delta$ on $\cE_{\mathrm{ess}}$,
\begin{equation*}
\cR(\wh H_{\mathrm{MoM}})
\le
\min_{H\in\cH}\cR(H)
+
\frc'(\sigma^2+\cR_{\max})\sqrt{\frac{\log(|\cH|/\delta)}{n_{\cT_2}}}.
\end{equation*}
Finally, since the weighted Gram condition and the effective sample size event are assumed to hold on $\cE_{\mathrm{ess}}$,
we may absorb $\P(\cE_{\mathrm{ess}}^c)$ into the overall failure probability.
This concludes the proof.
\color{black}

\section[Proof of Theorem 5 (Oracle lower bound)]{Proof of Theorem 5: Oracle lower bound}\label{app:proof_lower_bound}

\noindent
To establish the minimax lower bound in the local linear transfer model, we restrict attention to a convenient parametric submodel contained in the function class under consideration. Moreover, since centered Gaussian noise $\mathcal N(0,\sigma^2)$ is sub-Gaussian and therefore sub-exponential, it suffices to derive the lower bound under the Gaussian noise submodel. Any lower bound proved in this restricted setting applies \emph{a fortiori} to the full model with sub-exponential noise.
\\
Throughout the proof, we assume that $\source,\target\in\mathrm{Höl}(\beta,\frl;[0,1]^d)$ and we grant the estimator oracle access to the true
tessellation $H^\star=\{\cA_\ell^\star:\ell\in[L^\star]\}$, the anchor points
$\{x_\ell\}_{\ell\in[L^\star]}$ (with $x_\ell=x_{H^\star,\ell}$).
We work under Assumptions~\ref{Ass:transfer_function},
\ref{Ass:star_local-design}, and~\ref{Ass:lower_mass-balance}.
This can only make the estimation problem easier, hence any lower bound derived
in this oracle setting applies to the original problem.

\medskip
\noindent
Fix $r_\star:=\Delta_{\min}(H^\star)$.
Let $\{\varphi_\ell\}_{\ell\in[L^\star]}$ be a collection of smooth bump functions
such that $\varphi_\ell$ is supported in $\cA_\ell^\star$,
$\|\varphi_\ell\|_{\L^2}^2\asymp |\cA_\ell^\star|$, and $\|\varphi_\ell\|_\infty\le 1$.
Let $\{\psi_\ell\}_{\ell\in[L^\star]}$ be a collection of functions supported in
$\cA_\ell^\star$ such that $\psi_\ell\in \mathrm{H\ddot{o}l}(1+\beta_{\mathrm{loc}},1;\cA_\ell^\star)$ and
\begin{equation*}
\|\psi_\ell\|_{\L^2(\cA_\ell^\star)}^2 \asymp r_\star^{2(1+\beta_{\mathrm{loc}})}\,|\cA_\ell^\star|.
\end{equation*}
For vectors $u,v,w\in\{-1,+1\}^{L^\star}$, define a source/target pair
$(\source^{w},\target^{u,v,w})$ as follows.
First, define the source function as
\begin{equation*}
\source^{w}(x)
:=
\sum_{\ell=1}^{L^\star} w_\ell\,\delta_{\cS}\,\varphi_\ell(x),
\end{equation*}
where $\delta_{\cS}>0$ will be chosen later so that
$\source^w\in\mathrm{H\ddot{o}l}(\beta,\frl;[0,1]^d)$ and the KL divergences remain bounded.
Next, on each leaf $\cA_\ell^\star$, define coefficients
\begin{equation*}
a_\ell(u):=a_0,\qquad b_\ell(u):=u_\ell\,\delta_{\cT},
\end{equation*}
with fixed $a_0\in(0,A_{\max}]$ and $\delta_{\cT}>0$ chosen later, and define for all $x\in\cA_\ell^\star$ the
remainder
\begin{equation*}
R_{\ell,v}(x):= v_\ell\,\psi_\ell(x).
\end{equation*}
Finally, define the target function on each leaf by the structured transfer form
\begin{equation*}
\target^{u,v,w}(x)
=
a_\ell(u)\bigl(\source^w(x)-\source^w(x_\ell)\bigr)
+
b_\ell(u)
+
R_{\ell,v}(x),
\end{equation*}
for all $x\in\cA_\ell^\star$.
By construction, for $\delta_{\cS},\delta_{\cT}$ sufficiently small (depending only on
$\beta,\frl,A_{\max},B_{\max}$ and the constants in the regularity assumptions),
we have $(\source^w,\target^{u,v,w})\in \cF(H^\star,\beta,\beta,\beta_{\mathrm{loc}})$
and the uniform bounds $\|a^\star\|_\infty\le A_{\max}$, $\|b^\star\|_\infty\le B_{\max}$ hold.
Moreover, the mapping $(u,v,w)\mapsto (\source^w,\target^{u,v,w})$ is injective.

\medskip
\noindent
Let $\mathbb P^{u,v,w}$ denote the joint law of all observations
(source sample of size $n_{\cS}$ and target sample of size $n_{\cT}$)
under $(\source^w,\target^{u,v,w})$ with Gaussian noise $\mathcal N(0,\sigma^2)$
on both samples.
We apply Assouad's lemma (in the standard form, see e.g. \cite{tsybakov2008nonparametric}, Theorem~2.12)
to the hypercube $\{-1,+1\}^{3L^\star}$, using the squared $\L^2(\mu)$-loss.
It suffices to control: $(i)$ the $\L^2(\mu)$-separation between neighboring vertices,
and $(ii)$ the KL divergence between neighboring laws. The construction is cellwise,
hence both quantities factorize over $\ell\in[L^\star]$, and the resulting lower
bounds add over the three blocks of signs $u,v,w$.

\medskip
\noindent
\textbf{Step 1: Parametric term $\sigma^2 L^\star/n_{\cT}$.}
Consider neighbors that differ only in one coordinate $u_\ell$.
Then the source is unchanged, and the target differs on $\cA_\ell^\star$ by
\begin{equation*}
\target^{u,v,w}(x)-\target^{u',v,w}(x)=2\delta_{\cT}\,\mathbf 1_{\{x\in\cA_\ell^\star\}}.
\end{equation*}
Hence, using $\mu(\cA_\ell^\star)\gtrsim 1/L^\star$ by Assumption~\ref{Ass:lower_mass-balance},
\begin{equation*}
\|\target^{u,v,w}-\target^{u',v,w}\|_{\L^2(\mu)}^2
\gtrsim
\delta_{\cT}^2\,\mu(\cA_\ell^\star)
\gtrsim
\frac{\delta_{\cT}^2}{L^\star}.
\end{equation*}
Moreover, the KL divergence between $\mathbb P^{u,v,w}$ and $\mathbb P^{u',v,w}$
comes only from the target sample, and equals (conditionally on the target design)
\begin{align*}
\mathrm{KL}(\mathbb P^{u,v,w}\|\mathbb P^{u',v,w}| X^{\cT})
&=
\frac{1}{2\sigma^2}\sum_{i\in\cT}\big(\target^{u,v,w}(X_i)-\target^{u',v,w}(X_i)\big)^2\\
&=
\frac{2\delta_{\cT}^2}{\sigma^2}\,\#\{i\in\cT: X_i\in\cA_\ell^\star\}.
\end{align*}
Taking expectation and using Assumption~\ref{Ass:design_target} and the mass-balance condition,
$\E[\#\{i\in\cT:X_i\in\cA_\ell^\star\}]\asymp n_{\cT}\mu(\cA_\ell^\star)\asymp n_{\cT}/L^\star$, we get
\begin{equation*}
\mathrm{KL}(\mathbb P^{u,v,w}\|\mathbb P^{u',v,w})
\lesssim
\frac{\delta_{\cT}^2}{\sigma^2}\,\frac{n_{\cT}}{L^\star}.
\end{equation*}
Choose $\delta_{\cT}^2 \asymp \sigma^2 L^\star/n_{\cT}$ so that this KL is bounded by a universal constant.
Assouad's lemma then yields
\begin{equation*}
\inf_{\widehat f}\ \sup_{u,v,w}\ \E_{u,v,w}\big[\|\widehat f-\target^{u,v,w}\|_{\L^2(\mu)}^2\big]
\gtrsim
L^\star\cdot \frac{\delta_{\cT}^2}{L^\star}
\asymp
\frac{\sigma^2 L^\star}{n_{\cT}}.
\end{equation*}

\medskip
\noindent
\textbf{Step 2: Approximation term $(\Delta_{\min}(H^\star))^{2(1+\beta_{\mathrm{loc}})}$.}
Consider neighbors that differ only in one coordinate $v_\ell$.
Then the target differs on $\cA_\ell^\star$ by $2\psi_\ell$ and thus
\begin{equation*}
\|\target^{u,v,w}-\target^{u,v',w}\|_{\L^2(\mu)}^2
\gtrsim
\|\psi_\ell\|_{\L^2(\mu)}^2
\asymp
r_\star^{2(1+\beta_{\mathrm{loc}})}\,\mu(\cA_\ell^\star)
\gtrsim
\frac{r_\star^{2(1+\beta_{\mathrm{loc}})}}{L^\star}.
\end{equation*}
The corresponding KL divergence again comes only from the target sample:
conditionally on $X^{\cT}$,
\begin{equation*}
\mathrm{KL}(\mathbb P^{u,v,w}\|\mathbb P^{u,v',w}| X^{\cT})
=
\frac{1}{2\sigma^2}\sum_{i\in\cT}\big(\psi_\ell(X_i)-(-\psi_\ell(X_i))\big)^2
=
\frac{2}{\sigma^2}\sum_{i\in\cT}\psi_\ell(X_i)^2.
\end{equation*}
Taking expectation and using $\E[\psi_\ell(X)^2]\asymp \|\psi_\ell\|_{\L^2(\mu)}^2$
gives
\begin{equation*}
\mathrm{KL}(\mathbb P^{u,v,w}\|\mathbb P^{u,v',w})
\lesssim
\frac{n_{\cT}}{\sigma^2}\,\|\psi_\ell\|_{\L^2(\mu)}^2
\asymp
\frac{n_{\cT}}{\sigma^2}\,\frac{r_\star^{2(1+\beta_{\mathrm{loc}})}}{L^\star}.
\end{equation*}
Since the theorem concerns a fixed $H^\star$ (hence fixed $r_\star$), we may choose the constant
hidden in the definition of $\psi_\ell$ small enough so that this KL is bounded by a universal constant
(depending only on $\sigma$ and the constants in the assumptions), uniformly in $(n_{\cT},L^\star)$.
Assouad's lemma (\cite{tsybakov2008nonparametric}, Theorem~2.12) then yields
\begin{align*}
\inf_{\widehat f}\ \sup_{u,v,w}\ \E_{u,v,w}\big[\|\widehat f-\target^{u,v,w}\|_{\L^2(\mu)}^2\big]
&\gtrsim
L^\star\cdot \frac{r_\star^{2(1+\beta_{\mathrm{loc}})}}{L^\star}\\
&=
r_\star^{2(1+\beta_{\mathrm{loc}})}
=
(\Delta_{\min}(H^\star))^{2(1+\beta_{\mathrm{loc}})}.
\end{align*}

\medskip
\noindent
\textbf{Step 3: Source-estimation term $n_{\cS}^{-\frac{2\beta}{2\beta+d}}$.}
Consider neighbors that differ only in one coordinate $w_\ell$.
Then the source differs on $\cA_\ell^\star$ by $2\delta_{\cS}\varphi_\ell$ and,
because $\target^{u,v,w}$ depends on $\source^w$ through the term
$a_0(\source^w(x)-\source^w(x_\ell))$, the induced difference in the target satisfies for all $x\in\cA_\ell^\star$,
\begin{equation*}
\target^{u,v,w}(x)-\target^{u,v,w'}(x)
=
a_0\bigl(\source^w(x)-\source^{w'}(x)\bigr)
-
a_0\bigl(\source^w(x_\ell)-\source^{w'}(x_\ell)\bigr).
\end{equation*}
By construction of $\varphi_\ell$ with $\varphi_\ell(x_\ell)=1$ and $\|\varphi_\ell\|_{\L^2}^2\asymp|\cA_\ell^\star|$,
we obtain
\begin{equation*}
\|\target^{u,v,w}-\target^{u,v,w'}\|_{\L^2(\mu)}^2
\gtrsim
a_0^2\,\delta_{\cS}^2\,\mu(\cA_\ell^\star)
\gtrsim
\frac{\delta_{\cS}^2}{L^\star}.
\end{equation*}
The KL divergence between $\mathbb P^{u,v,w}$ and $\mathbb P^{u,v,w'}$ splits into the source part
and the target part. Under Gaussian noise and conditional on the designs,
\begin{multline*}
\mathrm{KL}(\mathbb P^{u,v,w}\|\mathbb P^{u,v,w'}| X^{\cS},X^{\cT})
\\=
\frac{1}{2\sigma^2}\sum_{i\in\cS}\big(\source^w(X_i)-\source^{w'}(X_i)\big)^2
+
\frac{1}{2\sigma^2}\sum_{i\in\cT}\big(\target^{u,v,w}(X_i)-\target^{u,v,w'}(X_i)\big)^2.
\end{multline*}
Taking expectations and using the design bounds yields
\begin{equation*}
\mathrm{KL}(\mathbb P^{u,v,w}\|\mathbb P^{u,v,w'})
\lesssim
\frac{n_{\cS}\delta_{\cS}^2}{\sigma^2}
+
\frac{n_{\cT}\delta_{\cS}^2}{\sigma^2}.
\end{equation*}
Since we are proving a lower bound, we may further grant the estimator oracle access to the
entire target sample (thus potentially reducing the effect of the second term), and we keep only the source
contribution:
\begin{equation*}
\mathrm{KL}(\mathbb P^{u,v,w}\|\mathbb P^{u,v,w'})
\lesssim
\frac{n_{\cS}\delta_{\cS}^2}{\sigma^2}.
\end{equation*}
Choose $\delta_{\cS}\asymp n_{\cS}^{-\frac{\beta}{2\beta+d}}$ (as in the standard Le Cam/Assouad construction for $\beta$-Hölder regression)
so that the KL is bounded by a universal constant while the $\L^2(\mu)$-separation is of order
$\delta_{\cS}^2/L^\star$ on one cell.
Assouad's lemma then yields
\begin{equation*}
\inf_{\widehat f}\ \sup_{u,v,w}\ \E_{u,v,w}\big[\|\widehat f-\target^{u,v,w}\|_{\L^2(\mu)}^2\big]
\gtrsim
L^\star\cdot \frac{\delta_{\cS}^2}{L^\star}
\asymp
n_{\cS}^{-\frac{2\beta}{2\beta+d}}.
\end{equation*}

\medskip
\noindent
\textbf{Step 4: Conclusion.}
By the product structure of the hypercube and the additivity of the $\L^2(\mu)$-separation and KL bounds over the
three blocks of signs $(u,v,w)$, Assouad's lemma yields a lower bound given by the sum of the three contributions
obtained in Steps~1-3. Therefore, there exists a constant $\frc>0$ depending only on
$d,\beta,\sigma,A_{\max},B_{\max}$ and the constants in
Assumptions~\ref{Ass:design}, \ref{Ass:star_local-design}, and \ref{Ass:lower_mass-balance} such that
\begin{multline*}
\inf_{\widehat f}
\sup_{(\source,\target)\in\cF(H^\star,\beta,\beta,\beta_{\mathrm{loc}})}
\Big[
\cR(\widehat f)-\cR(\target)
\Big]
\\\ge
\frc\bigg[
\frac{\sigma^2 L^\star}{n_{\cT}}
+
(\Delta_{\min}(H^\star))^{2(1+\beta_{\mathrm{loc}})}+
n_{\cS}^{-\frac{2\beta}{2\beta+d}}
\bigg].
\end{multline*}
The second display of the theorem follows immediately since $\wTL$ is a particular estimator, hence its worst-case
excess risk is bounded below by the minimax risk.

\color{black}

\section[Proof of Theorem 6: transfer map)]{Proof of Theorem 6: transfer map estimation}\label{app:proof_transfer_function_bound}

\noindent
Assume that the target tessellation $(\cA_\ell^\star)_{\ell\in[L^\star]}$ on which
Assumption~\ref{Ass:transfer_function} holds is known. In this section, we study the
estimation of the cellwise transfer functions $(g_\ell^\star)_{\ell\in[L^\star]}$.

\subsection{Local polynomial estimators}

\textbf{We work conditionally on the source sample $\cD_{\cS}$.
Since the source estimator $\wsource$ depends only on $\cD_{\cS}$ and is independent
of the target sample $\cD_{\cT}$, it can be treated as deterministic in the
target-side analysis.} Throughout, the data follow the model: for $\ell\in[L^\star]$, 
\begin{equation*}\label{Eq:regression_wsource}
Y_i = g_\ell^\star\big(\source(X_i)\big) + \varepsilon_i,
\qquad i\in\cT_1,
\end{equation*}
where $\E[\varepsilon_i|X_i]=0$ and Assumption~\ref{Ass:noise} holds on the target
sample, i.e.\ $\|\varepsilon_i\|_{\psi_1}\le \sigma_{\cT}$ for all $i\in\cT_1$. Since $\source$ is unknown, the local regression estimator is constructed by replacing the covariate $\source(X_i)$ with its estimator $\wsource(X_i)$ in the weighted least-squares criterion.
\\
Assume that the functions $g_\ell^\star$ satisfy Assumption~\ref{Ass:regularity_transfer}.
In particular, for any $\ell\in[L^\star]$ and any $u$ in a neighborhood of
$\wsource(x)$, we have the first-order Taylor expansion
\begin{equation*}\label{Eq:Taylor_g}
g_\ell^\star(u)
=
g_\ell^\star(\wsource(x))
+
(g_\ell^\star)'\big(\wsource(x)\big)\,\big(u-\wsource(x)\big)
+
R_{\ell,x}(u),
\end{equation*}
with remainder bounded by
\begin{equation}\label{Eq:R_star_bound}
|R_{\ell,x}(u)|
\le
\frl_g\,|u-\wsource(x)|^{\beta_g}.
\end{equation}
\noindent
Let $K_x:\R_+\to\R$ and $K_z:\R_+\to\R$ be bounded kernels supported on $[0,1]$ satisfying Assumption \ref{Ass:kernels},
and let $h>0$ and $\overline h>0$ be bandwidths. For each cell $\cA_\ell^\star$,
with reference point $x_{\ell}$, define the \emph{cellwise local least-squares estimator}
\begin{multline}\label{Eq:LS_regression_wsource}
(\widehat b_{\ell},\widehat a_{\ell})
\in
\argmin_{b,a\in\R}
\Bigg\{
\frac{1}{n_{\cT_1}}\sum_{i\in\cT_1}
\Big[Y_i- a\big(\wsource(X_i)-\wsource(x_{\ell})\big)-b\Big]^2\\
\times
K_{x,h}(\|X_i-x_{\ell}\|)
\, K_{z,\overline h}\big(|\wsource(X_i)-\wsource(x_{\ell})|\big)
\Bigg\},
\end{multline}
where $K_{x,h}=h^{-d}K_x(\cdot/h)$ and $K_{z,\overline h}=\overline h^{-1}K_z(\cdot/\overline h)$.
Write $\psi(t)=(1,t)^\top$ and define, for $i\in\cT_1$,
\begin{equation*}\label{Eq:phi_w_def}
\phi_{i,\ell}
=
\psi\Big(\frac{\wsource(X_i)-\wsource(x_{\ell})}{\overline h}\Big)
\end{equation*}
and
\begin{equation*}
w_{i,\ell}
=
K_{x,h}(\|X_i-x_{\ell}\|)
K_{z,\overline h}\big(|\wsource(X_i)-\wsource(x_{\ell})|\big).
\end{equation*}
Let $S_\ell=\sum_{i\in\cT_1} w_{i,\ell}$. Then, letting
\begin{equation*}\label{Eq:M_and_m_wsource}
M_{\ell}
=\frac{1}{S_\ell}\sum_{i\in\cT_1} w_{i,\ell}\,\phi_{i,\ell}\phi_{i,\ell}^\top
\quad\text{and}\quad
m_{\ell}
=\frac{1}{S_\ell}\sum_{i\in\cT_1} w_{i,\ell}\,\phi_{i,\ell}\,Y_i,
\end{equation*}
we can write
\begin{equation*}\label{Eq:wtheta_w_def}
\widehat\theta_{\ell}
= M_{\ell}^{-1}m_{\ell},
\quad
\widehat b_{\ell}=e_1^\top\widehat\theta_{\ell}
\quad\text{and}\quad
\widehat a_{\ell}=e_2^\top\widehat\theta_{\ell}.
\end{equation*}
The associated predictor on $\cA_\ell^\star$ is
\begin{equation*}\label{Eq:predictor_wsource}
\widehat f_\ell(x')
=
\widehat b_{\ell}
+\widehat a_{\ell}\big(\wsource(x')-\wsource(x_{\ell})\big).
\end{equation*}
\noindent
For clarity, define for any $k\in\N_0$,
\begin{equation*}\label{Eq:Sk_Tk_def}
S_{k,\ell}=\sum_{i\in\cT_1} w_{i,\ell}\big(\wsource(X_i)-\wsource(x_{\ell})\big)^k
\;\text{and}\;
T_{k,\ell}=\sum_{i\in\cT_1} w_{i,\ell}\big(\wsource(X_i)-\wsource(x_{\ell})\big)^kY_i.
\end{equation*}
Then we have
\begin{equation*}\label{Eq:closed_forms_wsource}
\widehat b_{\ell}
=\frac{S_{2,\ell}T_{0,\ell}{-}S_{1,\ell}T_{1,\ell}}{S_{2,\ell}S_{0,\ell}{-}S_{1,\ell}^2}
=\sum_{i\in\cT_1} W_{i,\ell}Y_i
\end{equation*}
and
\begin{equation*}
\widehat a_{\ell}
=\frac{S_{0,\ell}T_{1,\ell}{-}S_{1,\ell}T_{0,\ell}}{S_{2,\ell}S_{0,\ell}{-}S_{1,\ell}^2}
=\sum_{i\in\cT_1}\overline W_{i,\ell}Y_i,
\end{equation*}
where
\begin{equation*}
\varpi_{i,\ell}
=
w_{i,\ell}\big[S_{2,\ell}{-}(\wsource(X_i)-\wsource(x_{\ell}))S_{1,\ell}\big],
\end{equation*}
as well as
\begin{equation*}\label{Eq:varpi_kappa_def}
\kappa_{i,\ell}
=
w_{i,\ell}\big[(\wsource(X_i){-}\wsource(x_{\ell}))S_{0,\ell}{-}S_{1,\ell}\big],
\end{equation*}
and the normalized weights are
\begin{equation*}\label{Eq:weights_def}
W_{i,\ell}=\frac{\varpi_{i,\ell}}{\sum_{j\in\cT_1}\varpi_{j,\ell}}
\quad\text{and}\quad
\overline W_{i,\ell}=\frac{\kappa_{i,\ell}}{\sum_{j\in\cT_1}\varpi_{j,\ell}}.
\end{equation*}
They satisfy
\begin{equation*}\label{Eq:moment_identity}
\sum_{i\in\cT_1}\kappa_{i,\ell}\big(\wsource(X_i)-\wsource(x_{\ell})\big)
=
\sum_{i\in\cT_1}\varpi_{i,\ell}.
\end{equation*}

\begin{lem}[Polynomial reproduction]\label{Lem:regression_polynomial}
Let $p:\R\to\R$ be linear: $p(y)=a_p\big(y-\wsource(x_{\ell})\big)+b_p$.
Then
\begin{equation}\label{Eq:poly_repro_intercept}
\sum_{i\in\cT_1} W_{i,\ell}\,p(\wsource(X_i))=p(\wsource(x_{\ell}))=b_p,
\end{equation}
and
\begin{equation}\label{Eq:poly_repro_slope}
\sum_{i\in\cT_1}\overline W_{i,\ell}\,p(\wsource(X_i))=p'(\wsource(x_{\ell}))=a_p.
\end{equation}
\end{lem}

\begin{proof}
Since $p$ is linear, $(b_p,a_p)$ minimizes the objective in
\eqref{Eq:LS_regression_wsource}. The identities
\eqref{Eq:poly_repro_intercept}-\eqref{Eq:poly_repro_slope} then follow from the
closed-form expressions \eqref{Eq:closed_forms_wsource}.
\end{proof}

\begin{lem}\label{Lem:weights_bounds_LP}
Consider the estimator \eqref{Eq:LS_regression_wsource}. Suppose Assumptions
\ref{Ass:kernels}, \ref{Ass:local-design} and \ref{Ass:ESS_plugin} hold. Then, for any
cell $\cA_\ell^\star$, the following statements hold:
\begin{enumerate}[(i)]
\item For any $i\in\cT_1$, if $\|X_i-x_{\ell}\|>h$ or
$|\wsource(X_i)-\wsource(x_{\ell})|>\overline h$, then
$W_{i,\ell}=0$ and $\overline W_{i,\ell}=0$.

\item On the event in Assumption~\ref{Ass:ESS_plugin},
\begin{equation}\label{Eq:bound_sum_W_LP}
\sum_{i\in\cT_1}|W_{i,\ell}|
\le
\frac{\sqrt{2}\frc_2^{\mathrm{ess}}\,
\|K_x\|_{\infty}\|K_z\|_{\infty}}{\lambda_0},
\end{equation}
and
\begin{equation}\label{Eq:bound_sum_Wbar_LP}
\sum_{i\in\cT_1}|\overline W_{i,\ell}|
\;\le\;
\frac{\sqrt{2}\frc_2^{\mathrm{ess}}}{\overline h}\,
\frac{\|K_x\|_{\infty}\|K_z\|_{\infty}}{\lambda_0}.
\end{equation}

\item On the same event,
\begin{equation}\label{Eq:bound_sum_W2_LP}
\sum_{i\in\cT_1}|W_{i,\ell}|^2
\le
\frac{2\frc_2^{\mathrm{ess}}\,
\|K_x\|_{\infty}^2\|K_z\|_{\infty}^2}
{n_{\cT_1} h^d \overline h\,\lambda_0^2},
\end{equation}
and
\begin{equation}\label{Eq:bound_sum_Wbar2_LP}
\sum_{i\in\cT_1}|\overline W_{i,\ell}|^2
\le
\frac{2\frc_2^{\mathrm{ess}}\,
\|K_x\|_{\infty}^2\|K_z\|_{\infty}^2}
{n_{\cT_1} h^d \overline h^{3}\lambda_0^2}.
\end{equation}
\end{enumerate}
\end{lem}

\begin{proof}
We only treat the bound for $W_{i,\ell}$; the case of $\overline W_{i,\ell}$ is
identical up to the additional factor $\overline h^{-1}$.
By \eqref{Eq:weights_def}-\eqref{Eq:phi_w_def} and Assumption~\ref{Ass:kernels},
\begin{multline*}
|W_{i,\ell}|
\le
\frac{1}{n_{\cT_1} h^d \overline h}\,
\|\widehat M_{\ell}^{-1}\|_{\op}\,
\Big\|\psi\Big(\frac{\wsource(X_i)-\wsource(x_{\ell})}{\overline h}\Big)\Big\|_2\,
\|K_x\|_{\infty}\|K_z\|_{\infty}
\\\times\mathbf 1_{\{\|X_i-x_{\ell}\|\le h,\ |\wsource(X_i)-\wsource(x_{\ell})|\le \overline h\}}.
\end{multline*}
On $\{|\wsource(X_i)-\wsource(x_{\ell})|\le \overline h\}$ we have
$\|\psi(\cdot)\|_2\le \sqrt{2}$. Moreover, by Assumption~\ref{Ass:local-design},
$\|\widehat M_{\ell}^{-1}\|_{\op}\le \lambda_0^{-1}$. Hence
\begin{equation}\label{Eq:Wi_pointwise_bound}
|W_{i,\ell}|
\le
\frac{\sqrt{2}}{n_{\cT_1} h^d \overline h}\,
\frac{\|K_x\|_{\infty}\|K_z\|_{\infty}}{\lambda_0}\,
\mathbf 1_{\{\|X_i-x_{\ell}\|\le h,\ |\wsource(X_i)-\wsource(x_{\ell})|\le \overline h\}}.
\end{equation}
Statement $(i)$ follows immediately.\\
\noindent
Summing \eqref{Eq:Wi_pointwise_bound} over $i\in\cT_1$ yields
\begin{equation*}
\sum_{i\in\cT_1} |W_{i,\ell}|
\le
\frac{\sqrt{2}}{n_{\cT_1} h^d \overline h}\,
\frac{\|K_x\|_{\infty}\|K_z\|_{\infty}}{\lambda_0}\,
N_\ell,
\end{equation*}
where
\begin{equation*}
N_\ell
:=
\sum_{i\in\cT_1}
\mathbf 1_{\{\|X_i-x_{\ell}\|\le h,\ |\wsource(X_i)-\wsource(x_{\ell})|\le \overline h\}}.
\end{equation*}
By Assumption~\ref{Ass:ESS_plugin}, on the event of that assumption,
$N_\ell\le \frc_2^{\mathrm{ess}}n_{\cT_1}h^d\overline h$, which gives
\eqref{Eq:bound_sum_W_LP}. The bound \eqref{Eq:bound_sum_Wbar_LP} follows similarly.\\
\noindent
For $(iii)$, combining \eqref{Eq:Wi_pointwise_bound} with \eqref{Eq:bound_sum_W_LP} yields
\begin{equation*}
\sum_{i\in\cT_1}|W_{i,\ell}|^2
\le
\Big(\max_{i\in\cT_1}|W_{i,\ell}|\Big)\sum_{i\in\cT_1}|W_{i,\ell}|
\le
\frac{2\frc_2^{\mathrm{ess}}\|K_x\|_{\infty}^2\|K_z\|_{\infty}^2}
{n_{\cT_1} h^d \overline h\,\lambda_0^2},
\end{equation*}
which is \eqref{Eq:bound_sum_W2_LP}. The bound \eqref{Eq:bound_sum_Wbar2_LP} is analogous.
\end{proof}

\begin{prop}[Risk bounds for $\widehat a_{\ell}$ and $\widehat b_{\ell}$]\label{Prop:risk_wsource}
Suppose Assumptions \ref{Ass:kernels}, \ref{Ass:star_local-design} and \ref{Ass:ESS_plugin}. Work conditionally on the source sample $\cD_{\cS}$, so that
$\wsource$ is deterministic. Fix $\ell\in[L^\star]$ and let $x_{\ell}$ be the
reference point of the cell $\cA_\ell^\star$.
Then, on the event in Assumption~\ref{Ass:ESS_plugin}, for all $x\in\cA_\ell^\star$
such that $|\wsource(x)-\wsource(x_{\ell})|\le \overline h$, we have
\begin{multline}\label{Eq:risk_slope_wsource}
\E\Big[\big(\widehat a_{\ell}-(g_\ell^\star)'\big(\wsource(x_{\ell})\big)\big)^2\, \big|\, \cD_{\cS}\Big]
\\\le
\frac{2\frc_2^{\mathrm{ess}}\|K_x\|_{\infty}^2\|K_z\|_{\infty}^2}{\lambda_0^2}
\Bigg[
\frac{4\frl_g^2\|\source-\wsource\|_{\infty}^{2\beta_g}}{\overline h^2}
+
4\frl_g^2 \overline h^{2(\beta_g-1)}
+
\frac{\sigma_{\cT}^2}{n_{\cT_1}h^d\overline h^3}
\Bigg].
\end{multline}
Moreover, on the same event,
\begin{multline}\label{Eq:risk_intercept_wsource}
\E\Big[\big(\widehat b_{\ell}-g_\ell^\star\big(\wsource(x_{\ell})\big)\big)^2\, \big|\, \cD_{\cS}\Big]
\\\le
\frac{2\frc_2^{\mathrm{ess}}\|K_x\|_{\infty}^2\|K_z\|_{\infty}^2}{\lambda_0^2}
\Bigg[
4\frl_g^2\Big(\|\source-\wsource\|_{\infty}^{2\beta_g}
+
\overline h^{2\beta_g}\Big)
+
\frac{\sigma_{\cT}^2}{n_{\cT_1}h^d\overline h}
\Bigg].
\end{multline}
\end{prop}

\begin{proof}
We prove \eqref{Eq:risk_slope_wsource}; the bound for $\widehat b_\ell$ follows by the
same argument and is omitted.

\medskip
\noindent
Fix $\ell\in[L^\star]$ and abbreviate $x_\ell:=x_{\ell}$ and
$y_\ell:=\wsource(x_\ell)$. Let $x\in\cA_\ell^\star$ satisfy
$|\wsource(x)-y_\ell|\le \overline h$.
Recall that $\widehat a_\ell=\sum_{i\in\cT_1}\overline W_{i,\ell}\,Y_i$,
where the weights $(\overline W_{i,\ell})_{i\in\cT_1}$ are defined in
\eqref{Eq:weights_def} with $z=\wsource$. Decompose
\begin{equation}\label{Eq:var_bias_decomp_slope}
\E\Big[\big(\widehat a_{\ell}-(g_\ell^\star)'(y_\ell)\big)^2\,\big|\, \cD_{\cS}\Big]
=
\Var\big(\widehat a_\ell| \cD_{\cS}\big)
+
\Big(\E[\widehat a_\ell| \cD_{\cS}]-(g_\ell^\star)'(y_\ell)\Big)^2.
\end{equation}
\underline{\textbf{Bias}}
Conditionally on $\cD_{\cS}$ and the target design $(X_i)_{i\in\cT_1}$, we have
$\E[Y_i| X_i,\cD_{\cS}]=g_\ell^\star(\source(X_i))$ on $\cA_\ell^\star$, hence
\begin{equation*}
\E[\widehat a_\ell\,|\,(X_i)_{i\in\cT_1},\cD_{\cS}]
=
\sum_{i\in\cT_1}\overline W_{i,\ell}g_\ell^\star(\source(X_i)).
\end{equation*}
Therefore,
\begin{align*}
\E[\widehat a_\ell| \cD_{\cS}]-(g_\ell^\star)'(y_\ell)
&=
\E\Bigg[
\sum_{i\in\cT_1}\overline W_{i,\ell}\,
\Big(g_\ell^\star(\source(X_i))-g_\ell^\star(\wsource(X_i))\Big)
\, \big|\, \cD_{\cS}\Bigg]\\
&\quad+
\E\Bigg[
\sum_{i\in\cT_1}\overline W_{i,\ell}\,
g_\ell^\star(\wsource(X_i))
\, \big|\, \cD_{\cS}\Bigg]
-(g_\ell^\star)'(y_\ell)\\
&=:\mathrm{bias}_1+\mathrm{bias}_2.
\end{align*}
Since
$g_\ell^\star\in\mathrm{H\text{\"o}l}(\beta_g,\frl_g)$ and $\|\source-\wsource\|_\infty<\infty$,
\begin{equation*}
|\mathrm{bias}_1|
\le
\frl_g\,\|\source-\wsource\|_\infty^{\beta_g}\,
\E\Bigg[\sum_{i\in\cT_1}|\overline W_{i,\ell}|\, \big|\, \cD_{\cS}\Bigg].
\end{equation*}
On the event in Assumption~\ref{Ass:ESS_plugin}, Lemma~\ref{Lem:weights_bounds_LP}
 gives
\begin{equation}\label{Eq:bias1_bound_wsource}
|\mathrm{bias}_1|
\le
\frac{\sqrt{2}\frc_2^{\mathrm{ess}}\frl_g\|K_x\|_\infty\|K_z\|_\infty}{\lambda_0\,\overline h}\,
\|\source-\wsource\|_\infty^{\beta_g}.
\end{equation}
\noindent
To control $\mathrm{bias}_2$, define the linear polynomial
\begin{equation*}
p(u)=g_\ell^\star(y_\ell)+(g_\ell^\star)'(y_\ell)\,(u-y_\ell).
\end{equation*}
Write $g_\ell^\star(u)=p(u)+R_\ell(u)$, where by \eqref{Eq:R_star_bound},
$|R_\ell(u)|\le \frl_g|u-y_\ell|^{\beta_g}$. By Lemma~\ref{Lem:regression_polynomial}
, we have
\begin{equation*}
\sum_{i\in\cT_1}\overline W_{i,\ell}\,p(\wsource(X_i))=(g_\ell^\star)'(y_\ell).
\end{equation*}
Hence,
\begin{equation*}
\mathrm{bias}_2
=
\E\Bigg[
\sum_{i\in\cT_1}\overline W_{i,\ell}\,R_\ell(\wsource(X_i))
\, \big|\, \cD_{\cS}\Bigg].
\end{equation*}
On the support of $\overline W_{i,\ell}$ we have
$|\wsource(X_i)-y_\ell|\le \overline h$, so $|R_\ell(\wsource(X_i))|\le \frl_g\overline h^{\beta_g}$.
Using again Lemma~\ref{Lem:weights_bounds_LP} on the $\cE_{\mathrm{ess}}$ (defined by \eqref{Eq:event_ESS}),
\begin{equation}\label{Eq:bias2_bound_wsource}
|\mathrm{bias}_2|
\le
\frl_g\,\overline h^{\beta_g}\,
\sum_{i\in\cT_1}|\overline W_{i,\ell}|
\le
\frac{\sqrt{2}\frc_2^{\mathrm{ess}}\frl_g\|K_x\|_\infty\|K_z\|_\infty}{\lambda_0}\,
\overline h^{\beta_g-1}.
\end{equation}
Combining \eqref{Eq:bias1_bound_wsource}-\eqref{Eq:bias2_bound_wsource} and using
$(a+b)^2\le 2a^2+2b^2$ yields, on the event $\cE_{\mathrm{ess}}$ (defined by \eqref{Eq:event_ESS}),
\begin{equation}\label{Eq:bias_sq_bound_wsource}
\Big(\E[\widehat a_\ell| \cD_{\cS}]-(g_\ell^\star)'(y_\ell)\Big)^2
\le
\frac{4\frc_2^{\mathrm{ess}}\|K_x\|_\infty^2\|K_z\|_\infty^2}{\lambda_0^2}
\Bigg[
\frac{4\frl_g^2\|\source-\wsource\|_{\infty}^{2\beta_g}}{\overline h^2}
+
4\frl_g^2\overline h^{2(\beta_g-1)}
\Bigg].
\end{equation}

\medskip
\noindent
\underline{\textbf{Variance}}
Conditionally on $(X_i)_{i\in\cT_1}$ and $\cD_{\cS}$, the weights
$(\overline W_{i,\ell})_{i\in\cT_1}$ are deterministic and
\[
\widehat a_\ell-\E[\widehat a_\ell| (X_i)_{i\in\cT_1},\cD_{\cS}]
=
\sum_{i\in\cT_1}\overline W_{i,\ell}\,\varepsilon_i.
\]
By Assumption~\ref{Ass:noise}, $\|\varepsilon_i\|_{\psi_1}\le \sigma_{\cT}$, hence
$\E[\varepsilon_i^2]\lesssim \sigma_{\cT}^2$ uniformly in $i$. Therefore,
\begin{equation*}
\Var(\widehat a_\ell| (X_i)_{i\in\cT_1},\cD_{\cS})
\le
\Big(\sup_{i\in\cT_1}\E[\varepsilon_i^2]\Big)\sum_{i\in\cT_1}\overline W_{i,\ell}^2
\lesssim
\sigma_{\cT}^2\sum_{i\in\cT_1}\overline W_{i,\ell}^2.
\end{equation*}
Taking expectation over the target design and using Lemma~\ref{Lem:weights_bounds_LP}
on the event $\cE_{\mathrm{ess}}$ (defined by \eqref{Eq:event_ESS}) gives
\begin{equation}\label{Eq:var_bound_wsource}
\Var(\widehat a_\ell| \cD_{\cS})
\le
\frac{2\frc_2^{\mathrm{ess}}\sigma_{\cT}^2
\|K_x\|_\infty^2\|K_z\|_\infty^2}{n_{\cT_1}h^d\overline h^3\,\lambda_0^2}.
\end{equation}

\medskip
\noindent
\underline{\textbf{Conclusion}}
Plugging \eqref{Eq:bias_sq_bound_wsource} and \eqref{Eq:var_bound_wsource}
into \eqref{Eq:var_bias_decomp_slope} yields \eqref{Eq:risk_slope_wsource}.
\end{proof}

\subsection{Proof of Theorem 6}\label{proof:thm_global_g_risk}

\textbf{Again, we work conditionally on the source sample $\cD_{\cS}$, so that $\wsource$ is fixed.}
Let $\ell=\ell^\star(X)$ and write $y_\ell:=\wsource(x_{\ell})$, so that
$g(X,y)=g_\ell^\star(y)$ and $\widehat g_{\wsource}(X,y)=\widehat g_{\wsource,\ell}(y)$.
Recall that
\begin{equation*}
\widehat g_{\wsource,\ell}(y)
=
\widehat a_\ell\,(y-y_\ell)+\widehat b_\ell
\quad\text{with}\quad
a_\ell^\star=(g_\ell^\star)'(y_\ell)
\;\text{and}\;
b_\ell^\star=g_\ell^\star(y_\ell).
\end{equation*}
On the event $\cE_y$ we have $|y-y_\ell|\le \overline h$. Then,
\begin{align}
\E\Big[\big(\widehat g_{\wsource,\ell}(y)-g_\ell^\star(y)\big)^2\, \big|\, \cD_{\cS}\Big]
&\le
3\,\overline h^2\,\E\Big[\big(\widehat a_\ell-a_\ell^\star\big)^2\, \big|\, \cD_{\cS}\Big]
+
3\,\E\Big[\big(\widehat b_\ell-b_\ell^\star\big)^2\, \big|\, \cD_{\cS}\Big]
\nonumber\\
&\quad+
3\,\Big(g_\ell^\star(y)-\big(a_\ell^\star(y-y_\ell)+b_\ell^\star\big)\Big)^2.
\label{Eq:basic_decomp_g}
\end{align}
By the Taylor remainder bound \eqref{Eq:R_star_bound} (with expansion point $y_\ell$), on $\cE_y$,
\begin{equation}\label{Eq:taylor_bias_term}
\Big(g_\ell^\star(y)-\big(a_\ell^\star(y-y_\ell)+b_\ell^\star\big)\Big)^2
\le
\frl_g^2|y-y_\ell|^{2\beta_g}
\le
\frl_g^2\overline h^{2\beta_g}.
\end{equation}
\noindent
On the event $\cE_{\mathrm{ess}}$ \eqref{Eq:event_ESS}, Proposition~\ref{Prop:risk_wsource} yields
\begin{equation}\label{Eq:slope_risk_bound_used}
\E\Big[\big(\widehat a_\ell-a_\ell^\star\big)^2\, \big|\, \cD_{\cS}\Big]
\lesssim
\frac{\frl_g^2\|\source-\wsource\|_\infty^{2\beta_g}}{\overline h^2}
+
\frl_g^2\,\overline h^{2(\beta_g-1)}
+
\frac{\sigma_{\cT}^2}{n_{\cT_1}h^d\overline h^3},
\end{equation}
and
\begin{equation}\label{Eq:intercept_risk_bound_used}
\E\Big[\big(\widehat b_\ell-b_\ell^\star\big)^2\, \big|\, \cD_{\cS}\Big]
\lesssim
\frl_g^2\Big(\|\source-\wsource\|_\infty^{2\beta_g}+\overline h^{2\beta_g}\Big)
+
\frac{\sigma_{\cT}^2}{n_{\cT_1}h^d\overline h}.
\end{equation}
Substituting \eqref{Eq:slope_risk_bound_used}-\eqref{Eq:intercept_risk_bound_used} and
\eqref{Eq:taylor_bias_term} into \eqref{Eq:basic_decomp_g}, and simplifying, gives
\begin{equation*}
\E\Big[\big(\widehat g_{\wsource,\ell}(y)-g_\ell^\star(y)\big)^2\, \big|\, \cD_{\cS}\Big]
\lesssim
\frl_g^2\|\source-\wsource\|_\infty^{2\beta_g}
+
\frl_g^2\,\overline h^{2\beta_g}
+
\frac{\sigma_{\cT}^2}{n_{\cT_1}h^d\overline h},
\end{equation*}
on $\cE_{\mathrm{ess}}\cap \cE_y$. Finally, on $\cE_{\cS}$ we have the sup-norm control
\eqref{Eq:epsS_rate} (see Appendix \ref{App:NW} for details), hence
\begin{equation*}
\|\source-\wsource\|_\infty^{2\beta_g}
\le
\frc\Bigg(
\frac{\log(c_{\cS}/\delta_{\cS})}{n_{\cS}}
\Bigg)^{\frac{2\beta_g\beta_{\cS}}{2\beta_{\cS}+d}},
\end{equation*}
which yields \begin{equation}
\E\Big[\big(\widehat g_{\wsource}(X,y)-g(X,y)\big)^2\, \big|\, \cD_{\cS}\Big]
\lesssim
\frl_g^2
\Bigg(
\frac{\log(c_{\cS}/\delta_{\cS})}{n_{\cS}}
\Bigg)^{\frac{2\beta_g\beta_{\cS}}{2\beta_{\cS}+d}}
+
\frl_g^2\,\overline h^{2\beta_g}
+
\frac{\sigma_{\cT}^2}{n_{\cT_1}h^d\overline h}.
\end{equation}
The optimized choice $\oh = (n_{\cT_1}h^d)^{-1/(2\beta_g+1)}$ balances
$\overline h^{2\beta_g}$ and $(n_{\cT_1}h^d\overline h)^{-1}$ and gives
\begin{equation}
\E\Big[\big(\widehat g_{\wsource}(X,y)-g(X,y)\big)^2\,\big|\, \cD_{\cS}\Big]
\lesssim
\Bigg(
\frac{\log(c_{\cS}/\delta_{\cS})}{n_{\cS}}
\Bigg)^{\frac{2\beta_g\beta_{\cS}}{2\beta_{\cS}+d}}
+
(n_{\cT_1}h^d)^{-\frac{2\beta_g}{2\beta_g+1}}.
\end{equation}

\section{Technical results}\label{app:technical_results}

\subsection{Nadarya-Watson estimator}\label{App:NW}
We briefly recall classical consistency and deviation results for the
Nadaraya-Watson (NW) estimator, which will be used to control the estimation
error of the source regression function.
\noindent
\textbf{\underline{Setting}}
Let $(X_i,Y_i)_{i\in\cS}$ be i.i.d.\ observations from the regression model
\begin{equation*}
Y_i = \source(X_i) + \varepsilon_i,
\qquad
\E[\varepsilon_i | X_i]=0,
\end{equation*}
where $X_i\in[0,1]^d$ has density $p_\cS$ satisfying Assumption \ref{Ass:design_source}.
We assume that the noise variables $(\varepsilon_i)_{i\in\cS}$ are independent and
sub-exponential, i.e., there exist constants $(\sigma_{\cS},b)>0$ such that
\begin{equation*}
\E\left[\exp(\lambda \varepsilon_i)|X_i\right]
\le
\exp\left(\frac{\sigma_{\cS}^2\lambda^2}{2}\right),
\end{equation*}
for all $|\lambda|<1/b$.
Let $K:\R_+\to\R_+$ be a bounded, Lipschitz kernel, supported on $[0,1]$ and satisfying
$\int K = 1$, and let $h_{\cS}>0$ denote a bandwidth.
The Nadaraya-Watson estimator of $\source$ is defined as
\begin{equation*}
\wsource(x)
=
\frac{\sum_{i\in\cS} K_{h_{\cS}}(\|x-X_i\|) Y_i}
{\sum_{i\in\cS} K_{h_{\cS}}(\|x-X_i\|)},
\qquad
K_{h_{\cS}}(u) := h_{\cS}^{-d} K(u/h_{\cS}).
\end{equation*}
\noindent
\textbf{\underline{Smoothness assumption}}
We assume throughout that the source regression function belongs to a
H\"older class $\mathrm{Höl}(\beta_{\cS},\frl_{\cS}),$ for some smoothness parameter $\beta_{\cS}>0$ and radius $\frl_{\cS}>0$.\\

\noindent
\textbf{\underline{Uniform risk bounds}}
Under the above assumptions, the bias-variance trade-off of the NW estimator
is well understood.
In particular, classical results due to Stone~\cite{stone1982optimal}, and later
refinements by Gy\"orfi et al.~\cite{gyorfi2002distribution} and
Tsybakov~\cite{tsybakov2008nonparametric}, yield the following bound
\begin{equation}\label{eq:MSE_NW}
\E\bigg[\int(\wsource(x)-\source(x))^2dx\bigg]\lesssim h_{\cS}^{2\beta_{\cS}}
+
\frac{1}{n_{\cS} h_{\cS}^d}.
\end{equation}
Assume additionally that Assumption \eqref{Ass:design_source} holds, we get the
uniform mean-squared error
bound: there exist constants $\mathfrak c_1,\mathfrak c_2>0$ such that
\begin{equation*}
\sup_{x\in[0,1]^d}
\E\left[|\wsource(x)-\source(x)|^2\right]
\le
\mathfrak c_1 h_{\cS}^{2\beta_{\cS}}
+
\frac{\mathfrak c_2}{n_{\cS} h_{\cS}^d}.
\end{equation*}

\paragraph{Uniform deviation bounds under sub-exponential noise.}
High-probability sup-norm bounds can still be obtained when the noise is
sub-exponential, at the price of larger constants.
Using Bernstein-type inequalities for kernel regression with unbounded noise
(see, e.g., Giné and Guillou~\cite{gine2002rates}, Einmahl and Mason~\cite{einmahl2005uniform},
and the discussion in Tsybakov~\cite{tsybakov2008nonparametric}),
there exist constants $c_{\cS},\frc>0$ such that, for any
$\delta_{\cS}\in(0,1)$, with probability at least $1-\delta_{\cS}$,
\begin{equation*}
\|\wsource-\source\|_\infty
\le
\frc
\Bigg(
h_{\cS}^{\beta_{\cS}}
+
\sqrt{\frac{\log(c_{\cS}/\delta_{\cS})}{n_{\cS} h_{\cS}^d}}
\Bigg).
\end{equation*}
\noindent
Choosing the bandwidth optimally as
\begin{equation*}
h_{\cS}
=\frc
\Bigg(
\frac{\log(c_{\cS}/\delta_{\cS})}{n_{\cS}}
\Bigg)^{\frac{1}{2\beta_{\cS}+d}},
\end{equation*}
and assuming that
$n_{\cS} h_{\cS}^d / \log n_{\cS} \to \infty$,
we obtain the uniform rate
\begin{equation*}
\|\wsource-\source\|_\infty
\le
\frc
\Bigg(
\frac{\log(c_{\cS}/\delta_{\cS})}{n_{\cS}}
\Bigg)^{\frac{\beta_{\cS}}{2\beta_{\cS}+d}}.
\end{equation*}
\noindent
Finally, this bound directly implies a control on higher-order powers of the
sup-norm error.
In particular, for any $\beta_g>0$, with the same probability,
\begin{equation*}
\|\source-\wsource\|_\infty^{2\beta_g}
\le
\frc
\Bigg(
\frac{\log(c_{\cS}/\delta_{\cS})}{n_{\cS}}
\Bigg)^{\frac{2\beta_g\beta_{\cS}}{2\beta_{\cS}+d}}.
\end{equation*}
\noindent
Such bounds will play a key role in controlling the additional error induced
by the estimation of the source regression function when composed with a
$\beta_g$-H\"older transfer function.

\subsection{Deviation and concentration inequalities}\label{App:concentration}

\begin{thm}[Chernoff's bound for Bernoulli random variables]
Let $(Z_i)_{i\in[n]}$ be i.i.d. Bernoulli random variables with parameter $p\in(0,1)$. For any $\rho\in(0,1)$,
\begin{equation*}
\P\left(
\sum_{i=1}^nZ_i \le \rho np
\right)
\le
\exp\left(
- \frac{(1-\rho)^2}{2}\, np
\right).
\end{equation*}
\end{thm}

\begin{thm}[Bernstein inequality, bounded summands]\label{thm:bernstein-bounded}
Let $(Z_i)_{i\in[n]}$ be independent, centered random variables with
$\lvert Z_i\rvert \le M$ almost surely, and set $v:=\sum_{i=1}^n \Var(Z_i)$.
Then, for any $t>0$,
\begin{equation*}
\P\left(\bigg|\sum_{i=1}^n Z_i\bigg| > t\right)\le 2 \exp\left(
 - \frac{t^2}{\,2\bigl(v + \tfrac{M}{3}\,t\bigr)}
\right)=2 \exp\!\left(
 - \frac{t^2}{\,2v + \tfrac{2}{3} M t}
\right).
\end{equation*}
Equivalently, for any $\delta\in(0,1)$, with probability at least $1-\delta$,
\begin{equation*}
\left|\sum_{i=1}^n Z_i\right|\le
\sqrt{2 v\log\Big(\frac{2}{\delta}\Big)}+
\frac{M}{3}\log\Big(\frac{2}{\delta}\Big).
\end{equation*}
\end{thm}

\begin{thm}\label{Thm:Bernstein_matrix}[Matrix Bernstein inequality]
Let $X_1, X_2, \dots, X_n$ be independent, random, self-adjoint matrices 
of dimension $d \times d$ such that $\E[X_i]=0$ for all $i\in[n]$ and there exists $R>0$ such that $\| X_i \|_{\mathrm{op}} \le R$ almost surely. Define the \emph{matrix variance parameter}
\begin{equation*}
\sigma^2 := \left\| \sum_{i=1}^n \mathbb{E}[X_i^2] \right\|_{\mathrm{op}}.
\end{equation*}
Then, for all $t\ge0$,
\begin{equation*}
\P\left( \Big\| \sum_{i=1}^n X_i \Big\|_{\mathrm{op}} \ge t \right)\le 2d \exp\left(\frac{-t^2/2}{\sigma^2 + Rt/3}\right).
\end{equation*}
As a corollary, there exists a constant $\frc>0$ such that
\begin{equation*}
\E\bigg[\Big\| \sum_{i=1}^n X_i \Big\|_{\mathrm{op}}\bigg]
\le C \, \max\!\big\{ \sqrt{\sigma^2 \log d}, \, R \log d \big\}
\end{equation*}
\end{thm}

\subsubsection*{Bernstein inequality for sub-exponential random variables}
Recall the definition of Orlicz norms.
\begin{df}\label{Def:Orlicz}
The Orlicz $\psi_1$-norm of a real-valued random variable $X$ is defined by:
\begin{equation*}
\|X\|_{\psi_1} = \inf \left\{ C > 0 \; : \; \mathbb{E}\left[ \exp\left( \frac{|X|}{C} \right) \right] \leq 2 \right\}.
\end{equation*}
A random variable is called \emph{sub-exponential} if $\|X\|_{\psi_1} < \infty$.\\
The Orlicz $\psi_2$-norm of a real-valued random variable $X$ is defined by:
\begin{equation*}
\|X\|_{\psi_2} = \inf \left\{ C > 0 \; : \; \mathbb{E}\left[ \exp\left( \frac{|X|^2}{C^2} \right) \right] \leq 2 \right\}.
\end{equation*}
A random variable is called \emph{sub-Gaussian} if $\|X\|_{\psi_2} < \infty$.
\end{df}
\noindent
The following examples can be found in \cite{vershynin2018high}, Section 2.5.
\begin{ex}\label{ex:sub-gaussian}
 \begin{enumerate}
 \item Any random variable $X\sim\cN(0,\sigma^2)$ is sub-gaussian with $\|X\|_{\psi_2}\leq C\sigma$, for some $C>0$.
 \item Any bounded random variable $X$ is sub-Gaussian with $\|X\|_{\psi_2}\leq C\|X\|_\infty$, for some $C>0$.
 \end{enumerate}   
\end{ex}
\noindent
Let us prove the useful lemma:
\begin{lem}\label{Lem:subexp-square}
Let $X$ be a sub-Gaussian random variable such that $\|X\|_{\psi_2}<\sigma$. Then, the random variable $X^2-\E[X^2]$ is sub-exponential, and
\begin{equation*}
\|X^2-\E[X^2]\|_{\psi_1}\le 4\sigma^2.
\end{equation*}
\end{lem}
\begin{proof}
Assume that $\|X\|_{\psi_2}<\sigma$, i.e. $\E[\exp(X^2/\sigma^2)]\le 2$. Using the standard inequality $e^{u}\ge 1+u$, we get
\begin{equation*}
2\ge \E\bigg[\exp\Big(\frac{X^2}{\sigma^2}\Big)\bigg]\ge 1+\frac{\E[X^2]}{\sigma^2},
\end{equation*}
so that $\E[X^2]\le \sigma^2$. Moreover, by Jensen's inequality,
\begin{equation*}
\E\bigg[\exp\Big(\frac{X^2}{4\sigma^2}\Big)\bigg]=\E\bigg[\bigg(\exp\Big(\frac{X^2}{\sigma^2}\Big)\bigg)^{1/4}\bigg]\le \E\bigg[\bigg(\exp\Big(\frac{X^2}{\sigma^2}\Big)\bigg)\bigg]^{1/4}\le 2^{1/4}.
\end{equation*}
Then,
\begin{align*}
\E\bigg[\exp\Big(\frac{|X^2-\E[X^2]|}{4\sigma^2}\Big)\bigg]
&\le \exp\Big(\frac{\E[X^2]}{4\sigma^2}\Big)\E\bigg[\exp\Big(\frac{X^2}{4\sigma^2}\Big)\bigg]\\
&\le e^{1/4}\cdot 2^{1/4}\le 2. 
\end{align*}
Hence the result.
\end{proof}

\noindent
We recall: 
\begin{prop}[\cite{vershynin2018high}, Proposition 2.5.2]\label{prop:equivalence_subGaussian}
Let $X$ be a random variable. Then the following properties are equivalent; the parameters $K_i > 0$ appearing in these properties differ from each other by at most an absolute constant factor.
\begin{enumerate}
\item For all $t \geq 0$, $\P(|X| \geq t) \leq 2 \exp(-t^2/K_1^2)$.
\item For all $p \geq 1$, $\|X\|_p = (\E|X|^p)^{1/p} \leq K_2 \sqrt{p}$.
\item For all $\lambda$ such that $|\lambda| \leq K_1$, we have
$\E[\exp(\lambda^2 X^2)] \leq \exp(K_3^2 \lambda^2)$.
\item We have
$\E[\exp(X^2 / K_4^2)] \leq 2$.
\item 
Moreover, if $\E[X] = 0$ then properties (i)–(iv) are also equivalent to the following property: For all $\lambda \in \R$,
$\E[\exp(\lambda X)] \leq \exp(K_5^2 \lambda^2)$ .
\end{enumerate}    
\end{prop}

\begin{lem}[Bernstein inequalities for sub-exponential/Gaussian random variables]\label{lem:Bernstein_subexp}
Let \(X_1, X_2, \dots, X_n\) be independent, mean-zero, random variables.
\begin{itemize}
\item Assume the $X_i$ are sub-exponential such that $\|X_i\|_{\psi_1} \leq K$ for all $i\in[n]$. Then, there exists $c > 0$ such that for all $t > 0$,
\begin{equation}\label{eq:Bernstein_ind_subexp}
\mathbb{P}\left( \left| \sum_{i=1}^n X_i \right| \geq t \right)
\leq 2 \exp\bigg(- c \, \min\left( \frac{t^2}{n K^2}, \frac{t}{K} \right)\bigg).
\end{equation}
Equivalently, with probability $1-\delta$, there exists $C>0$ such that
\begin{equation}\label{eq:Bernstein_ind_subexp2}
\bigg| \sum_{i=1}^n X_i \bigg|\le CK\Bigg(\sqrt{\frac{\log(2/\delta)}{n}}+\frac{\log(2/\delta)}{n}\Bigg).
\end{equation}
\item Assume the $X_i$ are sub-Gaussian such that $\|X_i\|_{\psi_2} \leq K$ for all $i\in[n]$. Then, there exists $c > 0$ such that for all $t > 0$,
\begin{equation}\label{eq:Bernstein_ind_subgauss}
\mathbb{P}\left( \bigg| \sum_{i=1}^n X_i \bigg| \geq t \right)
\leq 2 \exp\Big(- c \frac{t^2}{n K^2}\Big).
\end{equation}
\end{itemize}
\end{lem}

\begin{prop}[Bernstein inequality for weighted sums of sub-exponential variables]\label{prop:Bernstein_weight_subexp}
Let $X_1, X_2, \dots, X_n$ be independent sub-exponential random variables satisfying
$\|X_i\|_{\psi_1}\le K$ for all $i$, and let $(\alpha_i)_{i\in[n]}$ be deterministic weights.
Then there exists a universal constant $\frc>0$ such that, for all $\delta\in(0,1)$,
\begin{equation*}
\P\left(
\sum_{i=1}^n \alpha_i |X_i|
\ge
\frc K \Bigg(
\sqrt{\log(1/\delta)\sum_{i=1}^n \alpha_i^2}
+
\log(1/\delta)\max_{1\le i\le n}|\alpha_i|
\Bigg)
\right)
\le \delta.
\end{equation*}    
\end{prop}

\subsection{Approximation results}

\begin{df}
Let $A\subset\R^d$. The Kolmogorov $n$-width of a subset 
$K \subset \L^\infty(A)$ 
is defined by
\begin{equation}\label{Eq:komogorov_width_def}
d_n(K, \L^\infty(A))
:=
\inf_{\substack{V \subset \L^\infty(A) \\ \dim V = n}}
\;
\sup_{g \in K}
\;
\inf_{\phi \in V}
\| g - \phi \|_{\L^\infty(A)}.
\end{equation}
\end{df}
\noindent
Let $A\subset\R^d$ be a bounded Lipschitz (whose boundary can be described locally by graphs of Lipschitz functions) domain, $1\le p\le\infty$, and
\begin{equation*}
\mathcal B_s^{\mathrm{Höl}}(A)
:=\{ f\in\mathrm{Höl}(s;A): \|f\|_{\mathrm{Höl}(s;A)}\le 1\}
\end{equation*}
be the unit H\"older ball of smoothness $s>0$ (H\"older-Zygmund class).
The Kolmogorov $k$-width in $\L^p(A)$ is
\begin{equation*}
d_k(\mathcal B_s^{\mathrm{Höl}}(A),\L^p(A))
:= \inf_{\substack{V\subset \L^p(A)\\ \dim V=k}}
\ \sup_{f\in \mathcal B_s^{\mathrm{Höl}}(A)} \ \inf_{g\in V}\ \|f-g\|_{\L^p(A)}.
\end{equation*}
The following result can be found in (\cite{pinkus2012n}, Chapter 7) or (\cite{devore1993constructive}, Chapter 9).
\begin{thm}[Kolmogorov $k$-widths of H\"older balls in $\L^{p}$]\label{Thm:Kolmogorov_width}
Let $A\subset\R^{d}$ be a bounded Lipschitz (i.e. whose boundary can be described locally by graphs of Lipschitz functions) domain, let $1\le p\le\infty$, and let $s>0$. Then there exist constants $\frc,\frc'>0$ depending only on $s$, $d$, $p$, and $A$ such that
\begin{equation*}
\frc k^{-s/d}
\le
d_{k}(\mathcal B_s^{\mathrm{Höl}}(A),\L^{p}(A))
\le\frc' k^{-s/d},
\qquad k\ge 1.
\end{equation*}
In particular, the optimal $n$-dimensional approximation rate of $s$-Hölder functions 
in $\L^{p}(A)$ is $n^{-s/d}$.
\end{thm}
\begin{cor}[Scaling to a ball of radius $r$]\label{Cor:Kolmogorov_width}
Let $Q=\cB_d(x_0,r)\subset\R^d$ be the ball of radius $r>0$ centered at $x_0$, let $1\le p\le\infty$, and let
\begin{equation*}
\mathcal B_s^{\mathrm{Höl}}(Q)
:= \big\{f\in\mathrm{Höl}(s;Q)): \|f\|_{\mathrm{Höl}(s;Q)}\le 1\big\}.
\end{equation*}
Then there exist constants $\frc,\frc'>0$ depending only on $s,d,p$ such that, for all $k\ge 1$,
\begin{equation*}
\frc k^{-s/d} r^{s+d/p}
\le
d_k\big(\mathcal B_s^{\mathrm{Höl}}(Q),\L^p(Q)\big)
\le k^{-s/d}\frc' r^{s+d/p}.
\end{equation*}
In particular, for $k=2$ and $p=2$,
\begin{equation*}
d_k\big(\mathcal B_s^{\mathrm{Höl}}(Q),\L^2(Q)\big)
=\frc r^{s+d/2}.
\end{equation*}
\end{cor}

\emph{Proof (sketch).}
Let $B_0:=\cB_d(0,1)$ and $Q:=\cB_d(x_0,r)$.  Ley $S_r:\R^{Q}\to \R^{B_0}$ be the function such that $f(y):=f(x_0+ry)$, which maps functions on $Q$ to functions on $B_0$.
Then $\|S_r f\|_{\L^p(B_0)} = r^{d/p}\|f\|_{\L^p(Q)}$, and the $\cB_s^{\mathrm{Höl}}$ seminorm scales by $r^{s}$,
so $S_r$ maps $\mathcal B^{s}(Q)$ onto a set equivalent (up to constants independent of $r$) to
$\mathcal B^{s}(B_0)$.
Applying the $\L^p$ width result on $B_0$ and rescaling yields the claim.

\end{document}